\documentclass[11pt,onecolumn]{article}
\usepackage[utf8]{inputenc} 
\usepackage[T1]{fontenc}    
\usepackage{hyperref}       
\usepackage{url}            
\usepackage{booktabs}       
\usepackage{amsfonts}       
\usepackage{nicefrac}       
\usepackage{microtype}      
\usepackage{times}
\usepackage{amsmath}
\usepackage{mathrsfs}
\usepackage{amssymb}
\usepackage{multirow}
\usepackage{xspace}
\usepackage{theorem}
\usepackage{graphicx}
\usepackage{ifpdf}
\usepackage{latexsym}
\usepackage{euscript}
\usepackage{xspace}
\usepackage{color}
\usepackage{makeidx}
\usepackage{wrapfig}

\usepackage{algorithm,algorithmic}
\usepackage{environ}
\usepackage{nameref}
\usepackage{hyperref}
\usepackage{multirow}

\usepackage{blkarray}
\usepackage{multicol}
\usepackage{array}
\usepackage{bigstrut}
\usepackage{booktabs, makecell}
\usepackage{array}

\long\def\remove#1{}

\newtheorem{theorem}{Theorem}[section] 

\newenvironment{proof}{{\em Proof:}}{\hfill{\hfill\rule{2mm}{2mm}}}

\definecolor{darkred}{rgb}{0.8, 0.2, 0.2}
\definecolor{darkgreen}{rgb}{0.5, 0.8, 0.1}
\definecolor{darkpurple}{rgb}{1.0, 0, 1.0}
\definecolor{darkblue}{rgb}{0, 0, 1.0}

\usepackage{environ}
\makeatletter
\def\namedlabel#1#2{\begingroup
    #2%
    \def\@currentlabel{#2}%
    \phantomsection\label{#1}\endgroup:
}
\makeatother
\newif\ifsolution
\solutiontrue
\newcounter{fakeequation}
\newcounter{fakeeqtmp}
\NewEnviron{MIP.Len}{\setcounter{fakeeqtmp}{\value{equation}}\setcounter{equation}{\value{fakeequation}}%
    \ifsolution\BODY\fi%
    \setcounter{fakeequation}{\value{equation}}\setcounter{equation}{\value{fakeeqtmp}}}
    \NewEnviron{MIP.Unif}{\setcounter{fakeeqtmp}{\value{equation}}\setcounter{equation}{\value{fakeequation}}%
    \ifsolution\BODY\fi%
    \setcounter{fakeequation}{\value{equation}}\setcounter{equation}{\value{fakeeqtmp}}}
\newcounter{fakeequationLP}
\newcounter{fakeeqtmpLP}    
\NewEnviron{LP.Len}{\setcounter{fakeeqtmpLP}{\value{equation}}\setcounter{equation}{\value{fakeequationLP}}%
    \ifsolution\BODY\fi%
    \setcounter{fakeequationLP}{\value{equation}}\setcounter{equation}{\value{fakeeqtmpLP}}}
    \NewEnviron{LP.Unif}{\setcounter{fakeeqtmpLP}{\value{equation}}\setcounter{equation}{\value{fakeequationLP}}%
    \ifsolution\BODY\fi%
    \setcounter{fakeequationLP}{\value{equation}}\setcounter{equation}{\value{fakeeqtmpLP}}}
    \NewEnviron{LP.Vol}{\setcounter{fakeeqtmpLP}{\value{equation}}\setcounter{equation}{\value{fakeequationLP}}%
    \ifsolution\BODY\fi%
    \setcounter{fakeequationLP}{\value{equation}}\setcounter{equation}{\value{fakeeqtmpLP}}}
    
    \NewEnviron{MIP.Vol}{\setcounter{fakeeqtmpLP}{\value{equation}}\setcounter{equation}{\value{fakeequationLP}}%
    \ifsolution\BODY\fi%
    \setcounter{fakeequationLP}{\value{equation}}\setcounter{equation}{\value{fakeeqtmpLP}}}
    \NewEnviron{LP.Area}{\setcounter{fakeeqtmpLP}{\value{equation}}\setcounter{equation}{\value{fakeequationLP}}%
    \ifsolution\BODY\fi%
    \setcounter{fakeequationLP}{\value{equation}}\setcounter{equation}{\value{fakeeqtmpLP}}}
    
    \NewEnviron{MIP.Area}{\setcounter{fakeeqtmpLP}{\value{equation}}\setcounter{equation}{\value{fakeequationLP}}%
    \ifsolution\BODY\fi%
    \setcounter{fakeequationLP}{\value{equation}}\setcounter{equation}{\value{fakeeqtmpLP}}}
    
    \usepackage{enumitem}

\newcommand{\R}{\mathbb{R}}
\newcommand{\Z}{\mathbb{Z}}
\newcommand{\Q}{\mathbb{Q}}
\newcommand{\field}{\mathbb{F}}

\newcommand{\Chains}{\mathbf{C}}
\newcommand{\p}[0]{\mathbf{p}}
\newcommand{\q}[0]{\mathbf{q}}
\newcommand{\x}[0]{\mathbf{x}}

\newcommand{\Homologies}[0]{\mathbf{H}}
\newcommand{\Boundaries}[0]{\mathbf{B}}
\newcommand{\Simplices}[0]{\mathbf{S}}
\newcommand{\Cycles}[0]{\mathbf{Z}}
\newcommand{\originalrep}{\mathbf{x}^{Orig}}
\newcommand{\optimalrep}{\mathbf{x}}
\newcommand{\orepentry}{x}
\newcommand{\chain}{\mathbf{c}}
\newcommand{\cycle}{{\mathbf z}}

\newcommand{\cycleu}{\mathbf{u}}
\newcommand{\cyclev}{\mathbf{v}}
\newcommand{\cyclew}{\mathbf{w}}
\newcommand{\boundingchain}{\mathbf{w}}
\newcommand{\hclass}{h} 
\newcommand{\tab}{Table }
\newcommand{\se}{Section }
\newcommand{\fig}{Figure }
\newcommand{\volvec}{\mathbf{v}}
\newcommand{\supp}{\mathrm{supp}}
\newcommand{\eq}{Equation }

\newcommand{\NI}{^{NI}}
\newcommand{\I}{^I}
\newcommand{\unif}{_{Unif}}
\newcommand{\len}{_{Len}}
\newcommand{\area}{_{Area}}
\newcommand{\birth}{\mathrm{Birth}}
\newcommand{\death}{\mathrm{Death}}

\newcommand{\barcode}{\mathrm{Barcode}}

\newcommand{\persinterval}{\mathcal{L}}
\newcommand{\closedinterval}{[b_i,d_i]}
\newcommand{\interval}{J}
\newcommand{\loss}{\mathrm{loss}}
\newcommand{\argmin}{\mathrm{argmin}}
\newcommand{\fcyclebasis}{\mathcal{C}}
\newcommand{\setoffilteredcyclebases}{\mathrm{FCB}}
\newcommand{\setofhcyclebases}{\mathrm{HCB}}
\newcommand{\setofpersistenthcyclebases}{\mathrm{PrsHCB}}
\newcommand{\dimss}[1]{^{(#1)}}  
\newcommand{\pr}{Program }
\newcommand{\EU}{_{E\text{-}Unif}}
\newcommand{\EL}{_{E\text{-}Len}}
\newcommand{\TU}{_{T\text{-}Unif}}
\newcommand{\TA}{_{T\text{-}Area}}

\newcommand{\spann}{{\mathrm{span}}}
\renewcommand{\arraystretch}{1.5}
\newcommand{\hcyclebasis}{\mathcal B}

\newcommand{\simplex}{\sigma}
\newcommand{\feasibleset}{\mathcal{X}}    
\newcommand{\Edge}{\mathrm{Edge}}
\newcommand{\Tri}{\mathrm{Tri}}
\newcommand{\goodcycleindices}{\mathcal P}
\newcommand{\goodtriangles}{\mathcal Q}
\newcommand{\goodedges}{\mathcal R}
\newcommand{\goodvolmatrix}{\partial_{2}[\mathcal{F}_1, \hat {\mathcal{F}_2}]}

\newcommand{\deathbasis}{\mathcal D}
\newcommand{\obasis}{Z} 
\newcommand{\obasisel}{\mathbf{z}}  
\newcommand{\cald}{\mathcal D}
\newcommand{\cale}{\mathcal E}
 
\newcommand{\calm}{\mathcal M}
\newcommand{\caln}{\mathcal N}
\theoremstyle{plain}

\theoremstyle{definition}
\newtheorem{definition}[theorem]{Definition}

\newtheorem{remark}[theorem]{Remark}

\newcommand{\low}{\mathrm{low}}
\newcommand{\hatgraph}{{\hat{\Gamma}}}

\title{Minimal Cycle Representatives in Persistent Homology using Linear Programming: an Empirical Study with User’s Guide}
\author{Lu Li \thanks{Macalester College, Mathematics, Statistics, and Computer Science Department, Saint Paul, MN, USA.} \\ lli1@macalester.edu  \and Connor Thompson \thanks{Purdue University, Department of Mathematics, West Lafayette, IN, USA. } \\ thomp774@purdue.edu \and  Gregory Henselman-Petrusek \thanks{Mathematical Institute, University of Oxford, UK. } \\  henselmanpet@maths.ox.ac.uk \and  Chad Giusti \thanks{University of Delaware, Department of Mathematical Sciences, Newark, DE, USA. } \\ cgiusti@udel.edu \and  Lori Ziegelmeier \footnotemark[1] \\lziegel1@macalester.edu}
\date{}
\begin{document}

\maketitle

\begin{abstract}
Cycle representatives of persistent homology classes can be used to provide descriptions of topological features in data. However, the non-uniqueness of these representatives creates ambiguity and can lead to many different interpretations of the same set of classes. One approach to solving this problem is to optimize the choice of representative against some measure that is meaningful in the context of the data. In this work, we provide a study of the effectiveness and computational cost of several $\ell_1$-minimization optimization procedures for constructing homological cycle bases for persistent homology with rational coefficients in dimension one, including uniform-weighted and length-weighted edge-loss algorithms as well as uniform-weighted and area-weighted triangle-loss algorithms. We conduct these optimizations via standard linear programming methods, applying general-purpose solvers to optimize over column bases of simplicial boundary matrices. 

Our key findings are: 
(i) optimization is effective in reducing the size of cycle representatives, though the extent of the reduction varies according to the dimension and distribution of the underlying data, (ii) the computational cost of optimizing a basis of cycle representatives exceeds the cost of computing such a basis, in most data sets we consider, (iii) the choice of linear solvers matters a lot to the computation time of optimizing cycles, (iv) the computation time of solving an integer program is not significantly longer than the computation time of solving a linear program for most of the cycle representatives, using the Gurobi linear solver, (v) strikingly, whether requiring integer solutions or not, we almost always obtain a solution with the same cost and almost all solutions found have entries in $\{-1, 0, 1\}$ and therefore, are also solutions to a restricted $\ell_0$ optimization problem, and (vi) we obtain qualitatively different results for generators in Erd\H{o}s-Rényi random clique complexes than in real-world and synthetic point cloud data. 

\end{abstract}

\section{Introduction}
\label{intro}

Topological data analysis (TDA) uncovers 
mesoscale structure in data by quantifying its shape using methods from algebraic topology. 
Topological features have proven effective when characterizing complex data, as they are qualitative, independent of choice of coordinates, and robust to some choices of metrics and moderate quantities of noise \cite{ghrist2014elementary,Carlsson2009TopologyAD}. 
As such, topological features extracted from data have recently drawn attention from researchers in various fields including, for example, neuroscience \cite{giusti2016two, bendich2016persistent, robert}, computer graphics \cite{pointcloud-topo, singh2007topological}, robotics \cite{pathplanning, VASUDEVAN20113292},  and computational biology \cite{collectivemotion, selectingbiologicalexperiments/journal.pone.0213679, zebrafish} (including the study of protein structure   \cite{ Usingpersistenthomologyanddynamicaldistancestoanalyzeproteinbinding,xia2016multiscale,xia2014persistent}.)  

The primary tool in TDA is \textit{persistent homology} (PH) \cite{Ghrist08}, which describes how topological features of data, colloquially referred to as ``holes", evolve as one varies a real-valued parameter. Each hole comes with a geometric notion of  \emph{dimension} which describes the shape that encloses the hole: connected components in dimension zero, loops in dimension one, shells in dimension two, and so on. From a parameterized topological space $X = (X_t)_{t \in S \subset \R_{\ge 0}}$, for each dimension $n$, PH produces a collection $\barcode_n(X)$ of lifetime intervals $\persinterval$ which encode for each topological feature the parameter values of its birth, when it first appears, and death, when it no longer remains.

A basic problem in the practical application of PH is interpretability: given an interval $\persinterval \in \barcode_n(X)$, how do we understand it in terms of the underlying data? A reasonable approach would be to find an element of the homology class, also known as a cycle representative, that  witnesses structure in the data that has meaning to the investigator. In the context of geometric data, this takes the form of an ``inverse problem,"  constructing  geometric structures corresponding to each persistent interval in the original input data.
For example, a representative for an interval $\persinterval \in \barcode_1(X)$ consists of a closed curve or linear combination of closed curves which enclose a set of holes across the family of spaces $(X_t)_{t \in \persinterval \subset S}$. Cycle representatives are used in \cite{emmett2015multiscale} to annotate particular loops as chromatin interactions, and  \cite{wu} uses cycle representatives to study and locate and reconstruct fine muscle columns in  cardiac trabeculae restoration.

\begin{figure}[!h]
    \centering
    \includegraphics[width=0.8\textwidth]{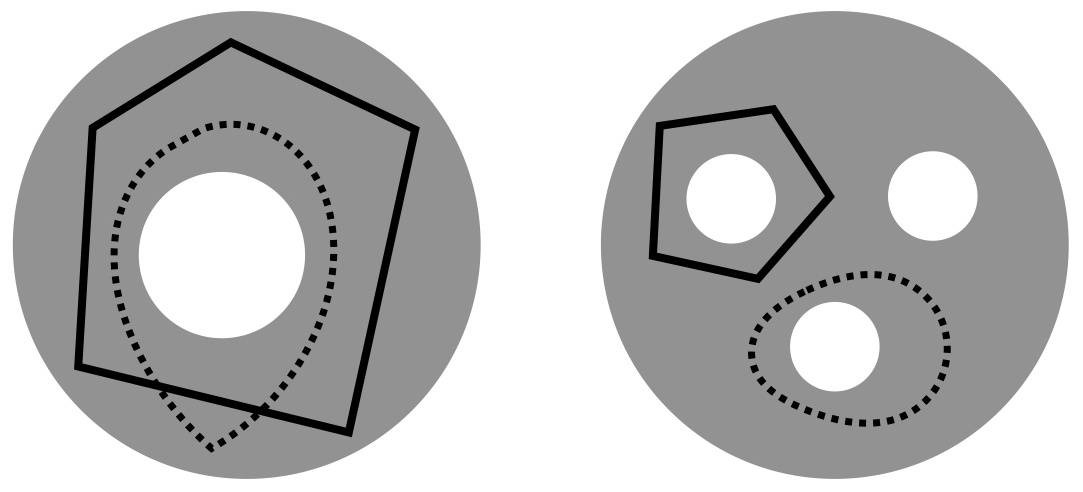}
    \caption{Two disks (grey) --- which we regard as 2-dimensional simplicial complexes, though the explicit decomposition into simplices is not shown --- with different numbers of holes (white) and cycle representatives (solid or dotted) from \cite{Carlsson2009TopologyAD}. The disk on the left has a single 2-dimensional  ``hole'' ($\beta_1 = 1$), and the two loops around it are cycle representatives for the same homology class. Similarly, the disk on the right has three ``holes'' ($\beta_1 = 3$) and the two loops shown are cycle representatives for different homology classes.
    }
    \label{fig:generatorExamples}
\end{figure}

An important challenge, however, is that cycle representatives are not uniquely defined. For example, in the left-hand image in \fig \ref{fig:generatorExamples} from \cite{Carlsson2009TopologyAD}, two curves enclose the same topological feature and thus, represent the same persistent homology class. We often want to find a cycle that captures not only the existence but also information about the location and shape of the hole that the homology class has detected. This often means optimizing an application-dependent property using the underlying data, e.g. finding a minimal length or bounding area/volume using an appropriate metric. The algorithmic problem of selecting such optimal representatives is currently an active area of research \cite{dey2011optimal,dey2018,Obayashi2018,wu,chen2010measuring}. 

There are diverse notions of optimality we may wish to consider in a given context, and which may have significant impact on the effectiveness or suitability of  optimization, including  
\begin{itemize}
    \item weight assignment to chains (uniform versus length or area weighted), 
 \item choice of loss function ($\ell_0$ versus $\ell_1$), 
 \item formulation of the optimization problem (cycle size versus bounded area or volume), and \item restrictions on allowable coefficients (rational, integral, or $\{0,1,-1\}$).  
 \end{itemize}
 Each has a unique set of advantages and disadvantages. For example, optimization using the $\ell_0$ norm with $\{0, 1, -1\}$-coefficients is thought to yield the most interpretable results, but $\ell_0$ optimization is NP-hard, in general \cite{chenhardness}. 
The problem of finding $\ell_1$ optimal cycles with rational coefficients, can be formulated as a more tractable linear programming problem.
While some literature exists to inform this choice \cite{dey2011optimal,Escolar2016,Obayashi2018}, questions of basic importance remain, including: 

\begin{enumerate}
  \item[Q1] How do the computational costs of the various optimization techniques compare? How much do these costs depend on the choice of a particular linear solver? 
  \item[Q2] What are the statistical properties of optimal cycle representatives? For example,  how often does the support of a representative form a single loop in the underlying graph? And,  how much do optimized cycles coming out of an optimization pipeline differ from the representative that went in?     
    \item[Q3] To what extent does choice of technique matter? For example, how often does the length of a length-weighted optimal cycle match the length of a uniform-weighted optimal cycle? 
    And, how often are $\ell_1$ optimal representatives $\ell_0$ optimal? 
\end{enumerate}

Given the conceptual and computational complexity of these problems (see \cite{chenhardness}), the authors expect that formal answers are unlikely to be available in the near future. However, even where theoretical results are available, strong \emph{empirical} trends may suggest different or even contrary principles to the practitioner. For example, while the persistence calculation is known to have matrix multiplication time complexity  \cite{primoz}, in practice the computation runs almost always in linear time. Therefore, the authors believe that a careful empirical exploration of questions 1-3 will be of substantial value. 

In this paper, we undertake such an exploration in the context of one-dimensional persistent homology over the field of rationals, $\Q$. We focus on linear programming (LP) and mixed-integer programming (MIP) approaches due to their ease of use, flexibility, and adaptability. In doing so, we present a new treatment of parameter-dependence (vis-a-vis selection of simplex-wise refinements) relevant to common cases of rational cycle representative optimization \cite{Obayashi2018, Escolar2016}, such as finding optimal cycle bases for the persistent homology of the Vietoris-Rips complex of a point cloud. We restrict our attention to one-dimensional homology to limit the number of reported statistics and data visualizations presented, although the methods discussed could be applied to any homological dimension. 

The paper is organized as follows. \se \ref{sec:background} provides an overview of some key concepts in TDA to inform a reader new to algebraic topology and establish notation. Then, we provide a survey of previous work on finding optimal persistent cycle representatives in \se \ref{problem formulation}, and formulate the methods used in this paper to find different notions of minimal cycle representatives via LP and MIP in 
\se \ref{methodsProblems}. \se \ref{methods} describes our experiments, including overviews of the data and the hardware and software we use for our analysis. In \se \ref{results},  we discuss the results of our experiments. We conclude and describe possible future work in \se \ref{discussion}.
    
 \section{Background: Topological Data Analysis and Persistent Homology} 
\label{sec:background}

In this section, we introduce key terms in algebraic and computational topology to provide minimal background and establish notation. For a more thorough introduction see, for example, \cite{Carlsson2009TopologyAD, hatcher2002algebraic, edelsbrunner2010computational, barcodeGhrist, persistenthomologyasurvey,TZH15}.

Given a discrete set of sample data, we approximate the topological space underlying the data by constructing a \textit{simplicial complex}. This construction expresses the structure as a union of vertices, edges, triangles, tetrahedrons, and higher dimensional analogues  \cite{Carlsson2009TopologyAD}.

\noindent \textbf{Simplicial Complexes.} A \textit{simplicial complex} is a collection $K$ of non-empty subsets of a finite set $V$. The elements of $V$ are called \textit{vertices} of $K$, and the elements of $K$ are called \textit{simplices}. A simplicial complex has the following properties: (1) $\{v\}$ in $K$ for all $v \in V$, and (2) $\tau \subset \sigma$ and $\sigma \in K$ guarantees that $\tau \in K$.

Additionally, we say that a simplex has \textit{dimension n} or is an \textit{n-simplex} if it has cardinality \textit{n+1}. We use $\Simplices_n(K)$ to denote the collection of \textit{n}-simplices contained in $K$.

While there are a variety of approaches to create a simplicial complex from data, our examples use a standard construction for approximation of point clouds.  Given a metric space $X$ with metric $d$ and real number $\epsilon \ge 0$, the \textit{Vietoris-Rips complex} for $X$, denoted by $\text{VR}_\epsilon(X)$, is defined as $$\text{VR}_\epsilon (X) = \{\sigma \in \Simplices_n(K) \mid d(x,y) \leq  \epsilon \text{ for all } x, y \in \sigma\}.$$
That is, given a set of discrete points $X$ and a metric $d$, we build a VR complex at scale $\epsilon$ by forming an $n$-simplex if and only if $n+1$ points in $X$ are pairwise within $\epsilon$ distance of each other. 

\noindent \textbf{Chains and chain complexes.}
Given a simplicial complex $K$ and an abelian group  $G$, the \emph{group of $n$-chains in $K$ with coefficients in $G$} is defined as
    \begin{align*}
        \Chains_n(K; G) 
        :=
        G^{\Simplices_n(K)}.
    \end{align*}
Formally, we regard $G^{\Simplices_n(K)}$ as a group of functions $\Simplices_n(K) \to G$ under element-wise addition. Alternatively, we may view $\Chains_n(K; G)$ as a group of formal $G$-linear combinations of $n$-simplices, i.e. $\left \{\sum_{\sigma}x_{\sigma} \sigma \mid x_{\sigma} \in G \text{ and } \sigma \in \Simplices_n(K) \right \}$.

\begin{remark}\label{rm:group}
We will focus on the cases where $G$ is $\Q$ (the field of rationals), $\Z$ (the group of integers), or $\field_2$ (the 2-element field).  Since we are most interested in the case $G = \Q$, we adopt the shorthand $\Chains_n(K) = \Chains_n(K,\mathbb{Q})$. 
\end{remark}

An element $\optimalrep = (\orepentry_\sigma)_{\sigma \in \Simplices_n(K)} \in G^{\Simplices_n(K)}$ is called an \emph{$n$-chain} of $K$.   As  in this example, we will generally use a bold-face symbol for the tuple $\optimalrep$ and corresponding light-face symbols for entries $\orepentry_\sigma$.  The \emph{support} of an $n$-chain is the set of simplices on which $\optimalrep_\sigma$ is nonzero:  
    \begin{align*}
        \supp(\optimalrep)  :=\{ \sigma \in \Simplices_n(K) \mid \orepentry_\sigma \neq 0 \}.
    \end{align*}
The $\ell_0$ norm\footnote{The $\ell_0$ ``norm'' is not a real norm as it does not satisfy the homogeneous requirement of a norm. For example, scaling a vector $\optimalrep$ by a constant factor does not change its $\ell_0$ ``norm''.} and  $\ell_1$ norm\footnote{ See Remark \ref{rm:group}. These choices of groups have a natural notion of absolute value.} of $\optimalrep$ are defined as 
    \begin{align*}
        ||\optimalrep||_0 := |\supp(\optimalrep) |
        &&
        \textstyle
        ||\optimalrep||_1 := \sum_{ \sigma \in \Simplices_n(K)} | \orepentry_\sigma  |
    \end{align*}

\begin{remark}[{Indexing conventions for chains and simplices}]
\label{rmk:indexingchains}
As chains play a central role in our discussion, it will be useful to establish some special conventions to describe them.  These conventions depend on the availability of certain linear orders, either on the set of vertices or the set of simplices.

\noindent \underline{Case 1:} \emph{Vertex set $V$ has a linear order $\le$. }  Every vertex set $V$ discussed in this text will be assigned a (possibly arbitrary) linear order.  Without  risk of ambiguity, we may therefore write
    \begin{align*}
        (v_0, \ldots, v_n)
    \end{align*}
for the $n$-chain that places a coefficient of 1 on $\sigma = \{v_0 \leq \cdots \leq v_n\}$ and 0 on all other simplices. 

\noindent \underline{Case 2:} \emph{Simplex set $\Simplices_n(K)$ has a linear order $\le$.}  We will sometimes define a linear order on $\Simplices_n(K)$.  This determines a unique bijection  $\sigma \dimss{n}: \{1, \ldots, |\Simplices_n(K)|\} \to  \Simplices_n(K)$ such that $\sigma_i\dimss{n} \le \sigma_j\dimss{n}$ iff $i \le j$.  This bijection determines an isomorphism
    $$
        \phi: 
        \Chains_n(K;G) = G^{\Simplices_n(K)}
        \to
        G^{|\Simplices_n(K)|}
    $$
such that $\phi(\optimalrep)_i = \orepentry_{\sigma_i}$ for all $i$.  
Provided a linear order $\le$,  we will use $\optimalrep$ to denote both $\optimalrep$ and $\phi(\optimalrep)$ and rely on context to clarify the  intended meaning.
\end{remark}

For each $n\geq 1$, the \textit{boundary map} $\partial_n: \Chains_n(K) \rightarrow \Chains_{n-1}(K)$ is the linear transformation defined on a basis vector  $(v_0, v_1, \ldots, v_n)$ by 
    \begin{align*}
    \textstyle
        \partial_n(v_0, v_1, \ldots, v_n) = \sum_{i=0}^n (-1)^i (v_0, \ldots, \hat{v_i}, \ldots, v_n)
    \end{align*}
where $\hat{v_i}$ omits $v_i$ from the vector. This map extends linearly from the basis of $n$-simplices to any $n$-chain in $\Chains_n(K)$. By an abuse of notation, we also denote the matrix representation of this boundary map, known as the \textit{boundary matrix}, as $\partial_n$. The boundary matrix is parametrized by the $n$-simplices $S_n(K)$ along the columns and $n-1$-simplices $S_{n-1}(K)$ along the rows.

The collection  $(\Chains_n(K))_{n\geq 0}$ along with the boundary maps $(\partial_n)_{n\geq 0}$ form a \textit{chain complex}
\[\ldots \Chains_{n+1}(K) \xrightarrow{\partial_{n+1}} \Chains_{n}(K) \xrightarrow{\partial_{n}} \Chains_{n-1}(K) \xrightarrow{\partial_{n-1}} \ldots \xrightarrow{\partial_3} \Chains_2(K) \xrightarrow{\partial_2} \Chains_1(K) \xrightarrow{\partial_1} \Chains_0(K) \xrightarrow{\partial_0} 0. \]
 
\begin{remark}[{Indexing conventions for boundary matrices}]
\label{rmk:boundarymatrixindexing}
In general, boundary matrix $\partial_n$ is regarded as an element of $G^{\Simplices_{n-1}(K) \times \Simplices_{n}(K)}$, that is, as an array with columns labeled by $n$-simplices and rows labeled by $n-1$-simplices.  However, given linear orders on  $\Simplices_{n-1}(K)$ and $\Simplices_{n}(K)$, we may naturally regard $\partial_n$ as an element of $G^{|\Simplices_{n-1}(K)| \times |\Simplices_{n}(K)|}$, see\ Remark \ref{rmk:indexingchains}. 
\end{remark}

\noindent \textbf{Cycles, boundaries.}  The \emph{boundary} of an $n$-chain $\optimalrep$ is  $\partial_n (\optimalrep)$.
An \textit{$n$-cycle} is an $n$-chain with zero boundary. The set of all $n$-cycles forms a subspace $\Cycles_n(K) := \textbf{ker}(\partial_n)$ of $\Chains_n(K).$ An \textit{$n$-boundary} is an $n$-chain that is the boundary of $(n+1)$-chains. The set of all $n$-boundaries forms a subspace $\Boundaries_n(K):= \textbf{im}(\partial_{n+1})$ of $\Chains_n(K).$   We refer to $\Cycles_n$ and $\Boundaries_n$ as the \emph{space of cycles} and \emph{space of boundaries}, respectively.

It can be shown that $\partial_n \circ \partial_{n+1}(\optimalrep) = 0$ for all $\optimalrep \in \Chains_{n+1}(K)$; colloquially,  ``a boundary has no boundary''. Equivalently,  $\partial_n \circ \partial_{n+1}$ is the zero map.
Since the boundary map takes a boundary to $0$, an $n$-boundary must also be an $n$-cycle. Therefore, $\Boundaries_n(K) \subseteq \Cycles_n(K)$.

\noindent \textbf{Homology, cycle representatives.} The \emph{$n$th homology group} of $K$ is defined as  the quotient
    \begin{align*}
        \Homologies_n(K): = \Cycles_n(K) / \Boundaries_n(K).
    \end{align*}
Concretely, elements of $\Homologies_n(K)$ are cosets of the form $[\cycle] = \{ \cycle'  \in \Cycles_n(K) | \cycle' - \cycle \in \Boundaries_n(K)\}$.\footnote{More generally, we denote the groups of cycles and boundaries with coefficients in $G$ as $\Cycles_n(K; G)$ and $\Boundaries_n(K; G)$.  The \emph{(dimension-$n$) homology of $K$ with coefficients in $G$} is $\Homologies_n(K; G) = \Cycles_n(K; G) / \Boundaries_n(K; G)$.}  An element $h \in \Homologies_n(K)$ is called an \emph{$n$-dimensional homology class}.  We say that a cycle $\cycle \in \Cycles_n(K)$ \emph{represents} $h$, or that $\cycle$ is a \emph{cycle representative of $h$} if $h = [\cycle]$.  We say that $\cycle$ and $\cycle'$ are \emph{homologous} if $[\cycle] = [\cycle']$.

\begin{figure}[H]
\begin{center}
\includegraphics[width=0.8\textwidth]{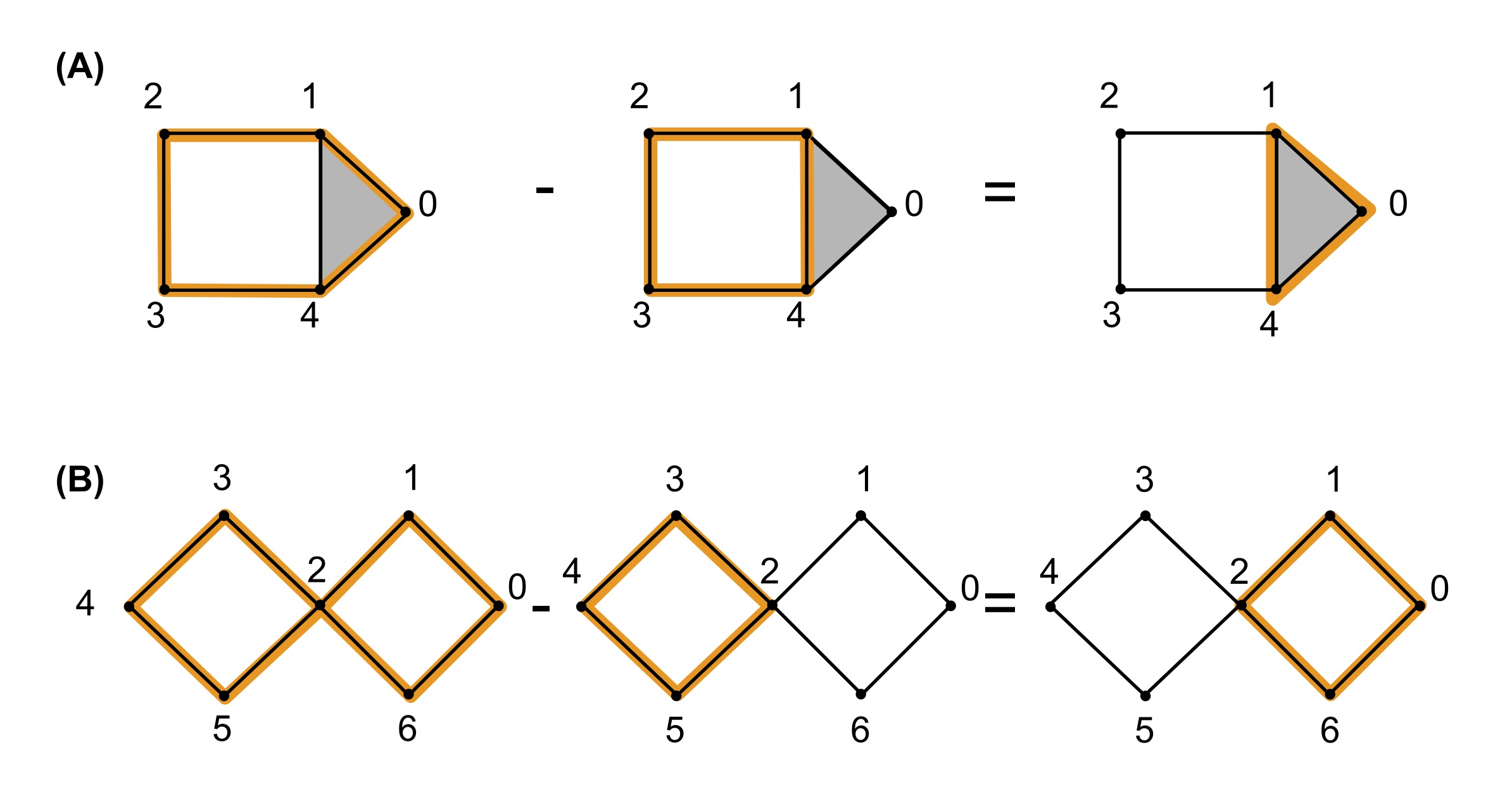} 
\end{center}
\caption{We show an example of homologous cycles in \textbf{(A)}, taken from \cite{TZH15}. The 1-cycle $(0,1) + (1,2) + (2,3) + (3,4) - (0,4)$ and the 1-cycle $(1,2) + (2,3) + (3,4) - (4,1)$ are homologous because their difference is the boundary of $(0,1,4)$. Subfigure \textbf{(B)} shows an example of non-homologous cycles. The 1-cycle $(\sum_{i=0}^4 (i, i+1))-(5,2)+(2,6)-(0,6)$ and the 1-cycle $(2,3) + (3,4)+(4,5)-(2,5)$ are not homologous because their difference is a cycle $(0,1)+(1,2)+(2,6)-(0,6)$ which is not a linear combination of boundaries of 2-simplices. } \label{fig:boundaryexample}
\end{figure}

\noindent \underline{Example} Consider the example in Figure \ref{fig:boundaryexample} (A), which illustrates two homologous 1-cycles and the example in Figure \ref{fig:boundaryexample} (B), which illustrates two non-homologous cycles.

\begin{remark}
The term \emph{homological generator} has been used differently by various authors: to refer to an arbitrary nontrivial homology class, an element in a (finite) representation of $\Homologies_n(K)$, as a set of cycles which generate the homology group, or (particularly in literature surrounding optimal cycle representatives)  interchangeably with cycle representative. We favor the term cycle representative, to avoid ambiguity.
\end{remark}

\noindent \textbf{Betti numbers, cycle bases.}  A \emph{(dimension-$n$) homological cycle basis} for $\Homologies_n(K)$ is a set of cycles $\hcyclebasis = \{ \cycle^1 , \ldots, \cycle^m\}$ such that $[\cycle^i] \neq [\cycle^j]$ when $i \neq j$, and $\{ [\cycle^1] , \ldots, [\cycle^m]\}$ is a  basis for $\Homologies_n(K)$.  Modulo boundaries, every $n$-cycle can be expressed as a unique linear combination in $\hcyclebasis$.  

Homological cycle bases have several useful interpretations.  It is common, for example, to think of a 1-cycle as a type of ``loop,'' generalizing the intuitive notion of a loop as a simple closed curve to include more intricate structures, and to regard the operation of adding boundaries as a generalized form of ``loop-deformation.''  Framed in this light, a homological cycle basis $\hcyclebasis$ for $\Homologies_1(K)$ can be regarded as a basis for the space of loops-up-to-deformation in $K$. Higher dimensional analogs of loops involve closed ``shells'' made up of $n$-simplices.

Another interpretation construes each nontrivial homology class $[\cycle] \neq 0$ as a \emph{hole} in $K$. Such holes are ``witnessed" by loops or shells that are not homologous to the zero cycle. Viewed in this light, $\Homologies_n(K)$ can naturally be regarded as the space of $(n+1)$-dimensional holes in $K$.  The rank of the $n$th homology group
    \[
    \beta_n(K) := \text{dim}(\Homologies_n(K)) = \text{dim}(\Cycles_n(K)) - \text{dim}(\Boundaries_n(K)),
    \]
therefore quantifies the ``number of independent holes'' in $K$.  We call $\beta_n$ the \emph{$n$th Betti number of $K$}.  
 
\noindent \underline{Example} Consider the gray disks in \fig \ref{fig:generatorExamples} (similar to a figure from  \cite{Carlsson2009TopologyAD}) with different numbers of holes and cycle representatives.

\noindent \textbf{Filtrations of simplicial complexes.} A \emph{filtration} on a simplicial complex $K$ is a nested sequence of  simplicial complexes $K_\bullet = (K_{\epsilon_i})_{i \in\{ 1, \ldots, T\}}$ such that
    $$
    K_{\epsilon_1} \subseteq K_{\epsilon_2} \subseteq \cdots \subseteq K_{\epsilon_T} = K
    $$
where $\epsilon_1 < \cdots < \epsilon_T$ are real numbers. A \emph{filtered simplicial complex} is a simplicial complex equipped with a filtration $K_\bullet$.

\noindent \underline{Example}
 Let  $X$ be a metric space with metric $d$, and let  $\epsilon_1 < \cdots < \epsilon_T$ be an increasing sequence of non-negative real numbers.  Then the sequence $K_\bullet = (K_{\epsilon_i})_{i \in\{ 1, \ldots, T\}}$ defined by $K_{\epsilon_i} = \text{VR}_{\epsilon_i}(X)$ is a filtration on $K$.

The data of a filtered complex is naturally captured by the \emph{birth} function on simplices, defined
    \begin{align*}
        \birth: K \to \R, \; \simplex \mapsto \min\{ \epsilon_i : \simplex \in K_{\epsilon_i} \}.
    \end{align*}
We regard the pair $(K, \birth)$ as a simpilicial complex whose simplices are weighted by the birth function.   For convenience, we will implicitly identify the sequence $K_\bullet$ with this weighted complex.   Thus, for example, when we say that $\simplex \in K$ has birth parameter $t$, we mean that  $\sigma\in K$ and  $\birth(\sigma) = t$.

\begin{definition}
A filtration $K_\bullet$ is \emph{simplex-wise} if one can arrange the simplices of $K$ into a sequence $(\simplex_1, \ldots, \simplex_{|K|})$ such that $K_{\epsilon_i} = \{\simplex_1, \ldots, \simplex_i\}$ for all $i$.  
A \emph{simplex-wise refinement}  of $K_\bullet$ is a simplex-wise filtration $K_\bullet'$ such that each space in $K_\bullet$ can be expressed in form $\{\simplex_1, \ldots, \simplex_j\}$ for some $j$.
\end{definition}

As an immediate corollary, given a simplex-wise refinement of $K_\bullet$, we may naturally interpret each boundary matrix $\partial_n$ as an element of $G^{|\Simplices_{n-1}(K)| \times |\Simplices_{n}(K)|}$, see Remark \ref{rmk:boundarymatrixindexing}.  Under this interpretation, columns (respectively, rows) with larger indices correspond to simplices with later birth times; that is, birth time increases as one moves left-to-right and top-to-bottom.

\noindent \textbf{Filtrations of chain complexes.} If we regard $\Chains_n(K_{\epsilon_i}; G)$ as a family of formal linear combinations in $\Simplices_n(K_{\epsilon_i})$, then it is natural to consider $\Chains_n(K_{\epsilon_i}; G)$ as a subgroup of $\Chains_n(K_{\epsilon_{j}}; G)$ for all $i<j$.  In particular, we have an inclusion map     \begin{align*}
    \textstyle
    \iota: \Chains_n(K_{\epsilon_i}; G) \to \Chains_n(K_{\epsilon_j}; G),
    \quad
    \sum_{\sigma \in \Simplices_n(K_{\epsilon_i})} x_\sigma \sigma
    \mapsto
    \sum_{\sigma \in \Simplices_n(K_{\epsilon_i})} x_\sigma \sigma
    +
    \sum_{\tau \notin \Simplices_n(K_{\epsilon_i})} 
    0 \cdot \tau
    \end{align*}

Given a simplex-wise refinement $K'_\bullet$, one can naturally regard $\chain$ as an element  $(c_1, c_2,  \ldots)$ of $ G^{|\Simplices_n(K_{\epsilon_i})|}$.  From this perspective, $\iota$ has a particularly simple interpretation, namely  ``padding'' by zeros:
    \begin{align*}
        \iota(\chain) = ( \underbrace{c_1, c_2, \ldots}_{\chain}, 0, \ldots, 0)
    \end{align*}
Similar observations hold when one replaces $\Chains_n$ with either $\Cycles_n$, the space of cycles, or $\Boundaries_n$, the space of boundaries.

\noindent \textbf{Persistent homology, birth, death.} The notion of birth for simplices has a natural extension to chains, as well as a variant called death.  Formally,  the \emph{birth} and \emph{death} parameters of  $\chain \in \Chains_n(K)$ are 
    \begin{align*}
    \birth(\chain) = \min \{\epsilon_i : \chain \in \Chains_n(K_{\epsilon_i}) \}
    &&
    \death(\chain) 
    = 
    \begin{cases}
    \min \{\epsilon_i : \chain \in \Boundaries(K_{\epsilon_i}) \} & \chain \in \Boundaries(K)
    \\
    \infty & else.
    \end{cases}
    \end{align*}
In the special case where $\chain$ is a cycle,  $\birth(\chain)$ is the first parameter value where $[\chain]$ represents a homology class, and $\death(\chain)$ is the first parameter value where $[\chain]$ represents the \emph{zero} homology class.   Thus, the half-open
\emph{lifespan interval} 
    \begin{align*}
        \persinterval(\chain) = [\birth(\chain), \death(\chain))
    \end{align*}
is the range of parameters over which $\chain$ represents a well-defined, nonzero homology class.

A \emph{(dimension-$n$) persistent homology cycle basis} is a subset $\hcyclebasis \subseteq \Cycles_n(K)$ with the following two properties:
    \begin{enumerate}
    \item Each $\cycle \in \hcyclebasis$ has a nonempty lifespan interval.
    \item For each $i \in \{1, \ldots, T\}$, the set 
        $$
        \hcyclebasis_{\epsilon_i} 
        := 
        \{\cycle \in \hcyclebasis : \epsilon_i \in \persinterval(\cycle) \}
        $$
    is a homological cycle basis for $\Homologies_n(K_{\epsilon_i})$.
    \end{enumerate}

Every filtration of simplicial complexes $ (K_{\epsilon_i})_{i \in\{ 1, \ldots, T\}}$ admits a  persistent homological cycle basis  $\hcyclebasis$ \cite{zomorodiancarlssoncomputingph}.  
Moreover, it can be shown that the \emph{multiset} of lifespan intervals (one for each basis vector), called the \emph{dimension-$n$ barcode of $K_\bullet$},
    \begin{align*}
        \barcode_n = 
        \{ \persinterval(\cycle) : \cycle \in \hcyclebasis \}
    \end{align*}
is invariant over all possible choices of persistent homological cycle bases $\hcyclebasis$ \cite{zomorodiancarlssoncomputingph}.

\begin{figure}[h!]
\begin{center}
\includegraphics[width=0.9\textwidth]{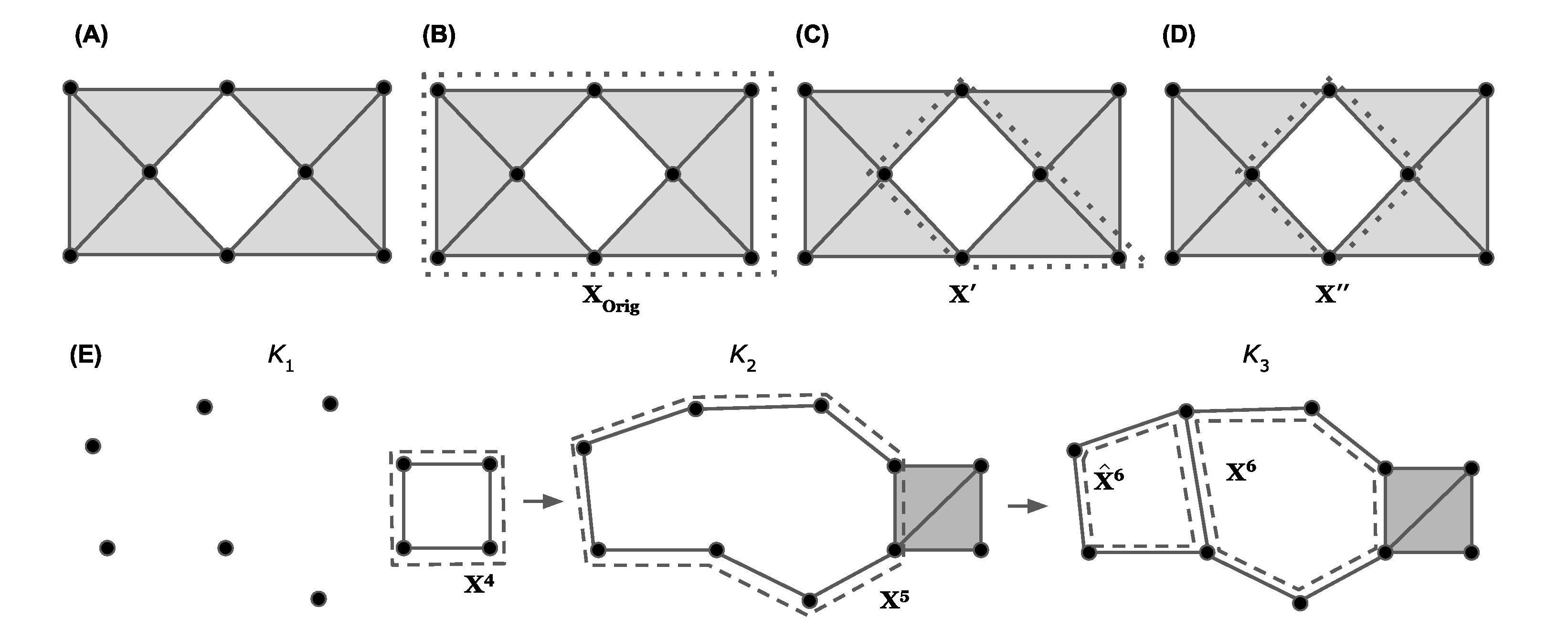} 
\end{center}
\caption{Examples of optimizing a cycle representative (using the notion of minimizing edges) within the same homology class (\textbf{A-D}) and using a basis of cycle representatives (\textbf{E}), modified examples taken from \cite{Escolar2016} and \cite{Obayashi2018}. The dotted lines represent a cycle representative for the enclosed ``hole''. Intuitively, we consider $\optimalrep''$ in (\textbf{D}) as the optimal cycle representative since it consists of the smallest number of edges. Subfigure (\textbf{E}) shows a case where we optimize a cycle representative using a basis of cycle representatives. In (\textbf{E}), $\{\optimalrep^4, \optimalrep^5, \optimalrep^6\}$ is the original basis of cycle representatives. We can substitute $\optimalrep^6$ with $\hat{\optimalrep}^6$, which we can obtain by adding $\optimalrep^5$ to $\optimalrep^6$, and thus obtain $\{\optimalrep^4, \optimalrep^5, \hat{\optimalrep}^6\}$ as the new basis of cycle representatives.}\label{fig:example-optimal}
\end{figure}

\noindent \underline{Example}  Consider the sequence of simplicial complexes $(K_1, K_2, K_3)$ shown in Figure \ref{fig:example-optimal} (E).  The set
    $
        \hcyclebasis = \{\optimalrep^4, \optimalrep^5, \optimalrep^6 \}
    $
is a (dimension-1) persistent homological cycle basis of the filtration.  The associated dimension-1 barcode is     
    $
    \barcode_1 = \{[1,2), [2,\infty), [3, \infty) \} 
    $ 
where $[2,\infty)$ and $[3,\infty)$ are the lifespans of  $\optimalrep^5$ and $\optimalrep^6$, respectively.

Barcodes are among the foremost tools in topological data analysis \cite{barcodeGhrist,  persistenthomologyasurvey}, and they contain a great deal of information about a filtration.  For example, it follows  immediately from the definition of persistent homological cycle bases  that
    $
        \beta_n(K_{\epsilon_i})
        =
        |\hcyclebasis_{\epsilon_i}|
    $
for all $n$ and $i$.  Consequently,
    \begin{align*}
        \beta_n(K_{\epsilon_i})
        =
        |\{\interval \in \barcode_n : \epsilon_i \in \interval \}|.
    \end{align*}

\noindent \textbf{Computing PH cycle representatives.} Barcodes and persistent homology bases may be computed via the so-called $R = DV$ decomposition \cite{cohen2006vines} of the boundary matrices $\partial_n$. Details are discussed in the Supplementary Material.

\section{Related work on minimizing cycle representatives}\label{problem formulation}

One important problem in TDA is interpreting homological features. In general, a lifetime interval $\persinterval$ corresponding to a feature may be represented by many different cycle representatives. As discussed in \cite{chenquantifying}, localizing homology classes can be characterized as finding a representative cycle with the most concise geometric measure. As an illustrative example from \cite{Escolar2016}, Figure \ref{fig:example-optimal} (A) shows a simplicial complex $K$ with $\Homologies_1(K)$ isomorphic to $\mathbb{Q}$ or equivalently, $\beta_1=1$; it contains one hole.  Figures \ref{fig:example-optimal} (B), (C), and (D) display three cycle representatives, $\originalrep$, $\optimalrep'$, and $\optimalrep''$, each of which represents the same homology class (heuristically, they encircle the same hole). We intuitively prefer $\optimalrep''$ as a representative, since it involves the fewest edges  and ``hugs'' the hole most tightly. Given a simplicial complex $K$ and a nontrivial cycle $\originalrep$ on it, we are interested in finding a cycle representative that is optimal with respect to some geometric criterion. In this section, we discuss previous studies on optimal cycle representatives. 

Minimal cycle representatives have proven  useful in many applications. Hiraoka et al. \cite{Hiraoka7035} use TDA to geometrically analyze amorphous solids. Their analysis using minimal cycle representatives explicitly captures hierarchical structures of the shapes of cavities and rings. Wu et al. \cite{wu} discuss an application of optimal cycles in Cardiac Trabeculae Restoration, which aims to reconstruct trabeculae, very complex muscle structures that are hard to detect by traditional image segmentation methods. They propose to use topological priors and cycle representatives to help segment the trabeculae. However, the original cycle representative can be complicated and noisy, causing the reconstructed surface to be messy. Optimizing the cycle representatives makes the cycle more smooth and thus, leads to more accurate segmentation results. Emmett et al. \cite{emmett2015multiscale} use PH to analyze chromatin interaction data to study chromatin conformation. They use loops to represent different types of chromatin interactions. To annotate particular loops as interactions, they need to first localize a cycle. Thus, they propose an algorithm to locate a minimal cycle representative for a given PH class using a breadth-first search, which finds the shortest path that contains the edge that enters the filtration at the birth time of the cycle and is homologically independent from the minimal cycles of all PH classes born before the current cycle.

There are several approaches used to define an optimal cycle representative. Dey et al. \cite{dey2011optimal} propose an algorithm to find an optimal homologous $1$-cycle for a given homology class via linear programming. That is, they consider a single homology class $[\optimalrep]$ and search for a homologous cycle representative that minimizes some geometric measure within that class, for instance, the number of $1$-simplices within the representative. Escolar and Hiraoka \cite{Escolar2016} extend this approach to find an optimal cycle by using cycles outside of a single homology class to ``factor out" redundant information. In this approach, an optimal cycle representative is no longer guaranteed to be homologous to the original representative, but the collection of cycle representatives have each been independently optimized and the collection still forms a homology basis. Further, \cite{Escolar2016} extends this approach to achieve a filtered cycle basis, although we note that it is not guaranteed to be a persistent homology basis. The two approaches in \cite{dey2011optimal,Escolar2016} aim to minimize the number of $1$-simplices in a cycle representative. Obayashi \cite{Obayashi2018} proposes an alternative algorithm for finding volume-optimal cycles in persistent homology, which minimize the number of $2$-simplices which the cycle representative bounds, also using linear programming. These methods serve as the foundation for our present paper and are discussed in more detail in the rest of this section.

In addition to linear programming, many researchers have contributed to the problem of computing optimal cycles: Wu et al. \cite{wu} propose an algorithm for finding shortest persistent $1$-cycles. They first construct a graph based on the given simplicial complex and then compute annotation for the given complex. The annotation assigns all edges different vectors and can be used to verify if a cycle belongs to the desired group of cycles. They then find the shortest path between two vertices of the edge born at the birth time of the original cycle representative using a new $A^*$ heuristic search strategy. Their algorithm is a polynomial time algorithm but in the worst case, the time complexity is exponential to the number of topological features. Dey et al. \cite{shortestonedimension} propose a polynomial-time algorithm that computes a set of loops from a VR complex of the given data whose lengths approximate those of a shortest basis of the one dimensional homology group $\Homologies_1$. In \cite{dey2018}, Dey et al. show that finding optimal (minimal) persistent $1$-cycles is NP-hard and then propose a polynomial time algorithm to find an alternative set of meaningful cycle representatives. This alternative set of representatives is not always optimal but still meaningful because each persistent $1$-cycle is a sum of shortest cycles born at different indices. They find shortest cycles using Dijkstra's algorithm by considering the $1$-skeleton as a graph.   This list is by no means exhaustive, and does not touch on the wide variety of related approaches, e.g. \cite{chenhardness}, which attempts to fit cycle representatives within a ball of minimum radius.

\label{sec:minimalgenerators}

In the next subsection, we briefly introduce some basic notions of linear programming, and then in the subsequent three subsections, we survey the optimization problems on which the present work is based.

\subsection{Background: Linear Programming}

Linear programming seeks to find a set of \emph{decision variables} $\textbf{x}=(x_1,\ldots,x_\eta)^T$ which optimize a linear \emph{cost} (or \emph{objective}) \emph{function} $\textbf{c}^T\textbf{x}$ subject to a set of linear (in)equality constraints $\textbf{a}_1^T\textbf{x}=b_1, \ldots, \textbf{a}_\mu^T\textbf{x}=b_\mu$. Any linear optimization problem can be written as a \emph{Linear Program} (LP) in \emph{standard form}  

\begin{align}
   \begin{split}
    \text{minimize } & \textbf{c}^T\textbf{x} \\
    \text{ subject to } & A\textbf{x}=\textbf{b}, \\
    & \textbf{x} \geq \textbf{0}
   \end{split}
   \label{eq:linearprogram}
\end{align}
where $A$ is the $\mu \times \eta$ matrix with coefficients of the constraints as rows and $\textbf{b}=(b_1,\ldots,b_\mu)^T$. Linear programming is well-studied and discussed in many texts \cite{bertsimas-LPbook, Vanderbei-LPbook,BoyVan2004}.

The \emph{optimal solution} $\textbf{x}^*$ satisfies the constraints while optimizing the objective function, yielding the \emph{optimal cost} $\textbf{c}^T\textbf{x}^*$. The \emph{feasible} set of solutions in a linear optimization problem is a polyhedron defined by the linear constraints. In general, the optimal solution of a (non-degenerate) LP will occur at a vertex of the polyhedron and can be solved with the standard \emph{simplex algorithm}, which traverses through the edges of the polytope to vertices in a cost reducing manner, or \emph{interior point methods}, which traverse along the inside of the polytope to reach an optimal vertex. In the worst-case, the complexity of the simplex method is exponential, yet it often runs remarkably fast, while interior point methods are polynomial time algorithms.

Standard LPs search for real-valued optimal solutions, but in some instances, a restriction of the decision variables, such as requiring integral solutions, may be necessitated. The \emph{mixed integer programming} (MIP) problem is written
\begin{align}
   \begin{split}
    \text{minimize } & \textbf{c}^T\textbf{x} +\textbf{d}^T\textbf{y}\\
    \text{ subject to } & A\textbf{x} + B \textbf{y}=\textbf{b}, \\
    & \textbf{x}, \textbf{y} \geq \textbf{0}\\
    & \textbf{x} \text{ integer}
   \end{split}
   \label{eq:linearprogram}
\end{align}
for matrices $A,B$ and vectors $\textbf{b},\textbf{c},\textbf{d}$. A standard LP has fewer constraints, and thus, will have optimal cost less than or equal to that of the analogous MIP. MIPs are much more challenging to solve than LPs, as they are discrete as opposed to convex optimization problems, and no efficient general algorithm is known \cite{bertsimas-LPbook}. However, LP \emph{relaxations}, (exponential-time) exact, (polynomial-time)  approximation, and heuristic algorithms can be used to obtain solutions to MIPs. 

In this paper, we determine optimal cycle representatives with both LP and MIP formulations.

\subsection{Minimal cycle representatives of a homology class} \label{singlecyclecase}

Given a homology class $\hclass =[\originalrep] \in \Homologies_n(K; G)$ and a function $\loss: \Cycles_n(K;G) \to \R$, how does one find a cycle representative of $\hclass$ on which $\loss$ attains minimum?  This problem is equivalent to  solving the following program defined in \cite{dey2011optimal}:
\begin{align}
   \begin{split}
    \text{minimize } & \loss(\optimalrep) \\
    \text{ subject to } & \optimalrep = \originalrep + \partial_{n+1} \boundingchain, \\
    & \boundingchain \in \Chains_{n+1}(K; G).
   \end{split}
   \label{eq:homologous}
\end{align}
This formulation considers all cycle representatives homologous to $\originalrep$, i.e. that differ by a boundary, and selects the optimal representative $\optimalrep$ which minimizes $\loss$.
\pr \eqref{eq:homologous} is correct because the coset $\hclass$ can be expressed in the form
    \begin{align*}
    \hclass
    =
    \originalrep + \Boundaries_n(K; G) 
    =
    \{ \originalrep + \partial_{n+1} \boundingchain \mid \boundingchain \in \Chains_{n+1}(K; G) \}
    \end{align*}  
 
In practice, a cycle representative $\originalrep$ is almost always provided together with the initial problem data (which consists of $K$, $G$, $\loss$, and $\hclass$), so the central challenge lies with solving \pr \eqref{eq:homologous}.

Several variants of \pr \eqref{eq:homologous} have been studied, especially where $\loss(\optimalrep) = ||\optimalrep||_0$ or $\loss(\optimalrep) = ||\optimalrep||_1$.  For a survey of results when $G = \field_2$, see \cite{chenhardness}.  For a discussion of results when $G = \Z$, see \cite{dey2011optimal}.  Broadly  speaking, minimizing against $\ell_0$  tends to be hard, even when $K$ has attractive properties such as embeddability in a low-dimensional Euclidean space \cite{borradaile2020minimum}.  Minimizing against $\ell_1$  is hard when $G = \field_2$ (since, in this case, $\ell_1 = \ell_0$),  but tractable via linear programming when $G \in \{\Q, \R\}$.

An interesting variant of the minimal cycle representative problem is the minimal \emph{persistent} cycle representative problem.  This problem was described in  \cite{chenquantifying} and may be formulated as follows:  given an interval $[a,b) \in \barcode_n(K_\bullet)$, solve 
\begin{align}
   \begin{split}
    \text{minimize } & \loss(\optimalrep) \\
    \text{ subject to } & \birth(\optimalrep) = a \\
    & \death(\optimalrep) = b \\
    & \optimalrep \in \Cycles_n(K_a; G)
   \end{split}
   \label{eq:minbarcoderep}
\end{align}
for $\optimalrep$.  An advanced treatment of this problem can be found in \cite{chenquantifying} for special case where (i)  $G = \field_2$, (ii) $\loss$ is a weighted sum of incident edges,  and (iii) the birth function assigns distinct values to any two simplices of the same dimension, and (iv) $n=1$.

\subsection{Minimal homological cycle bases}

\pr \eqref{eq:homologous} has a natural extension when $G$ is a field.  This extension focuses not on the smallest representative of a single homology class, but the smallest  \emph{homological cycle basis}.  It may be formally expressed as follows:
\begin{align}
   \begin{split}
    \text{minimize } & \textstyle \sum_{\optimalrep \in \hcyclebasis} \loss(\optimalrep) \\
    \text{ subject to } & \hcyclebasis \in \setofhcyclebases_n(K ; G)
   \end{split}
   \label{eq:generalminimalbasis}
\end{align}
where $\setofhcyclebases_n(K, G)$ is the family of dimension-$n$ homological cycle bases of $\Homologies_n(K;G)$. Thus, the program is finding a complete generating set $\hcyclebasis$ for all of the homological cycles of dimension $n$ where each element has been minimized in some sense.  

It is natural to wonder whether a solution to \pr \eqref{eq:generalminimalbasis} could be obtained by first calculating an arbitrary (possibly non-minimal) homological cycle basis $\hcyclebasis = \{\optimalrep^1, \ldots, \optimalrep^m \}$ and then selecting an optimal cycle representative $\cycle^i$ from each homology class $[\optimalrep^i]$.    Unfortunately, the resulting basis need not be optimal.  To see why, consider the simplicial complex $K_3$ shown in Figure \ref{fig:example-optimal} (E), taking $G$ to be $\Q$ and $\loss$ to be the $\ell_0$ norm.  Complex $K_{\bullet}$ has several different homological cycle bases in degree 1, including  $\hcyclebasis_0: = \{\hat \optimalrep^6, \optimalrep^6\}$, $\hcyclebasis_1: = \{\optimalrep^5,  \optimalrep^6\}$, and $\hcyclebasis_2: = \{ \optimalrep^5,  \hat \optimalrep^6 + \optimalrep^4\}$.  However, only $\hcyclebasis_0$ is $\ell_0$-minimal.  Moreover, each of the cycle representatives $\optimalrep^5, \optimalrep^6, \hat \optimalrep^6$ is already minimal within its homology class, so element-wise minimization will not transform  $\hcyclebasis_1$ or $\hcyclebasis_2$ into optimal bases, as might have been hoped.

As with the minimal cycle representative problem, the minimal homological cycle basis problem has been well-studied in the special case where $\loss$ is the $\ell_0$ norm and $G = \field_2$.  In this case, Equation \eqref{eq:generalminimalbasis} is NP-hard to approximate for $n>1$, but  $O(n^3)$ when $n=1$ \cite{dey2018efficient}. Several interesting variants and special cases have been developed in the $n=1$ case, as well \cite{shortestonedimension, erickson2005greedy, chen2010measuring}.  We are not currently aware of a systematic treatment for the case $G \in \{\Q, \R\}$.

A natural variant of the minimal homological cycle basis problem in \eq \eqref{eq:generalminimalbasis} is the minimal \emph{persistent homological cycle basis} problem  
\begin{align}
   \begin{split}
    \text{minimize } & \textstyle \sum_{\optimalrep \in \hcyclebasis} \loss(\optimalrep) \\
    \text{ subject to } & \hcyclebasis \in \setofpersistenthcyclebases_n(K_\bullet ; G)
   \end{split}
   \label{eq:persistentminimalbasis}
\end{align}
where $\setofpersistenthcyclebases_n(K_\bullet; G)$ is the set of \emph{persistent} homological cycle bases. This is a stricter condition than \pr \eqref{eq:generalminimalbasis} in that not only does it require that the elements of $\hcyclebasis$ form a generating set of all cycles of dimension $n$, but the barcode associated to $\hcyclebasis$ must match $\barcode_n(K_\bullet).$ That is, the multisets of birth/death pairs must be identical.

\pr \eqref{eq:persistentminimalbasis} is much more recent than \pr \eqref{eq:generalminimalbasis}, and consequently appears less in the literature. In the special case where every bar in the multiset $\barcode_n(K_\bullet)$ has multiplicity 1 (i.e. there are no duplicate bars), \pr \eqref{eq:persistentminimalbasis} can be solved by making one call to the minimal persistent cycle representative \pr \eqref{eq:minbarcoderep} for each bar.   In particular, the method of \cite{chenquantifying} may be applied to obtain a minimal persistent basis when the correct hypotheses are satisfied: $G = \field_2$, loss is a weighted sum of incident simplices, there are distinct birth times for all simplices of the same dimension, and $n=1$. In general, however, bars of multiplicity 2 are possible, and in this case repeated application of \pr \eqref{eq:minbarcoderep} will be insufficient.  

\subsection{Minimal filtered cycle space bases}

A close cousin of the minimal homological cycle basis \pr \eqref{eq:generalminimalbasis} is the minimal \emph{filtered cycle basis} problem, which may be formulated as  follows
\begin{align}
   \begin{split}
    \text{minimize } & \textstyle \sum_{\optimalrep \in \fcyclebasis} \loss(\optimalrep) \\
    \text{ subject to } & \fcyclebasis \in \setoffilteredcyclebases(K_\bullet ; G)
   \end{split}
   \label{eq:filteredminimalbasis}
\end{align}
where $\setoffilteredcyclebases(K_\bullet)$ is the family of all bases $\fcyclebasis$ of $\Cycles_n(K_{\epsilon_T})$ such that $\fcyclebasis$ contains a basis for each  subspace $\Cycles_n(K_{\epsilon_i})$, for $i \in \{1, \ldots, T\}$.  

Escolar and Hiraoka \cite{Escolar2016} provide a polynomial time solution via linear programming when
    \begin{enumerate}
        \item $\loss$ is the $\ell_1$ norm,
        \item $G = \Q$, and
        \item $K_\bullet$ is a simplex-wise filtration (without loss of generality, $K_\bullet = (K_1, \ldots, K_T)$).
    \end{enumerate}

Their key observation is that $\fcyclebasis$ is an optimal solution to \pr \eqref{eq:persistentminimalbasis} if and only if $\fcyclebasis$ can be expressed as a collection $\{\cycle^j : j \in J\}$ where  \begin{enumerate}
    \item the the set $J = \{j :  \Cycles_n(K_{j-1}) \subsetneq \Cycles_n(K_j) \}$ that indexes the cycles is the list of filtrations at which a novel $n$-cycle appears, and
    \item for each $j \in J$, the cycle $\cycle^j$ first appears in $K_j$ and is a minimizer for the loss function among all such cycles, i.e. $\cycle^j \in \argmin_{\cycle \in \Cycles_n(K_j) \backslash \Cycles_n(K_{j-1})} \loss(\cycle).$
\end{enumerate}
The authors formulate this problem as

\begin{align}
\begin{split}
\text{minimize } & ||\mathbf{x} ||_1  \\
\text{ subject to } & \mathbf{x} = \originalrep + \sum_{r\in R} w_r g^r + \sum_{s \in S} v_s f^s \\
& \boundingchain \in \Q^R  \\
& \mathbf{v} \in \Q^S
\end{split}
\label{eq:escolarargmin}
\end{align}
where $\originalrep \in\Cycles_N(K_j) \backslash \Cycles_N(K_{j-1})$ is a novel cycle representative at filtration $j$; $\{g^r : r \in R\}$ is a basis for $\Boundaries_n(K_{j-1})$\footnote{Because of the assumption that $K_\bullet$ is a simplex-wise filtration, if there is a new $n$-cycle in $K_j$ then there cannot also be a new $(n+1)$-simplex, so this is also a basis for $\Boundaries_n(K_j).$}; and $\{g^r : r \in R\} \cup \{f^s : s \in S\}$ is an extension of the given basis for $\Boundaries_n(K_{j-1})$ to a basis for $\Cycles_n(K_{j-1})$. That is, $\originalrep$ is a cycle that has just appeared in the filtration. To optimize it, we are allowed to consider linear combinations of both boundaries, $\{g^r\}$, and cycles, $\{f^s\}$, born before $\originalrep.$ The cycle $\optimalrep$ obtained in this way cannot have a birth time before that of $\originalrep$, but may have a different death time if $[\sum_{s\in S}v_sf^s]$ dies later than $[\originalrep]$.

The algorithm developed in \cite{Escolar2016} is cleverly constructed to extract $\originalrep$, $\{g^r : r \in R\}$, and $\{f^s : s \in S\}$ from matrices which are generated in the normal course of a barcode calculation.

\begin{figure}[]
\begin{center}
\includegraphics[width=0.9\textwidth]{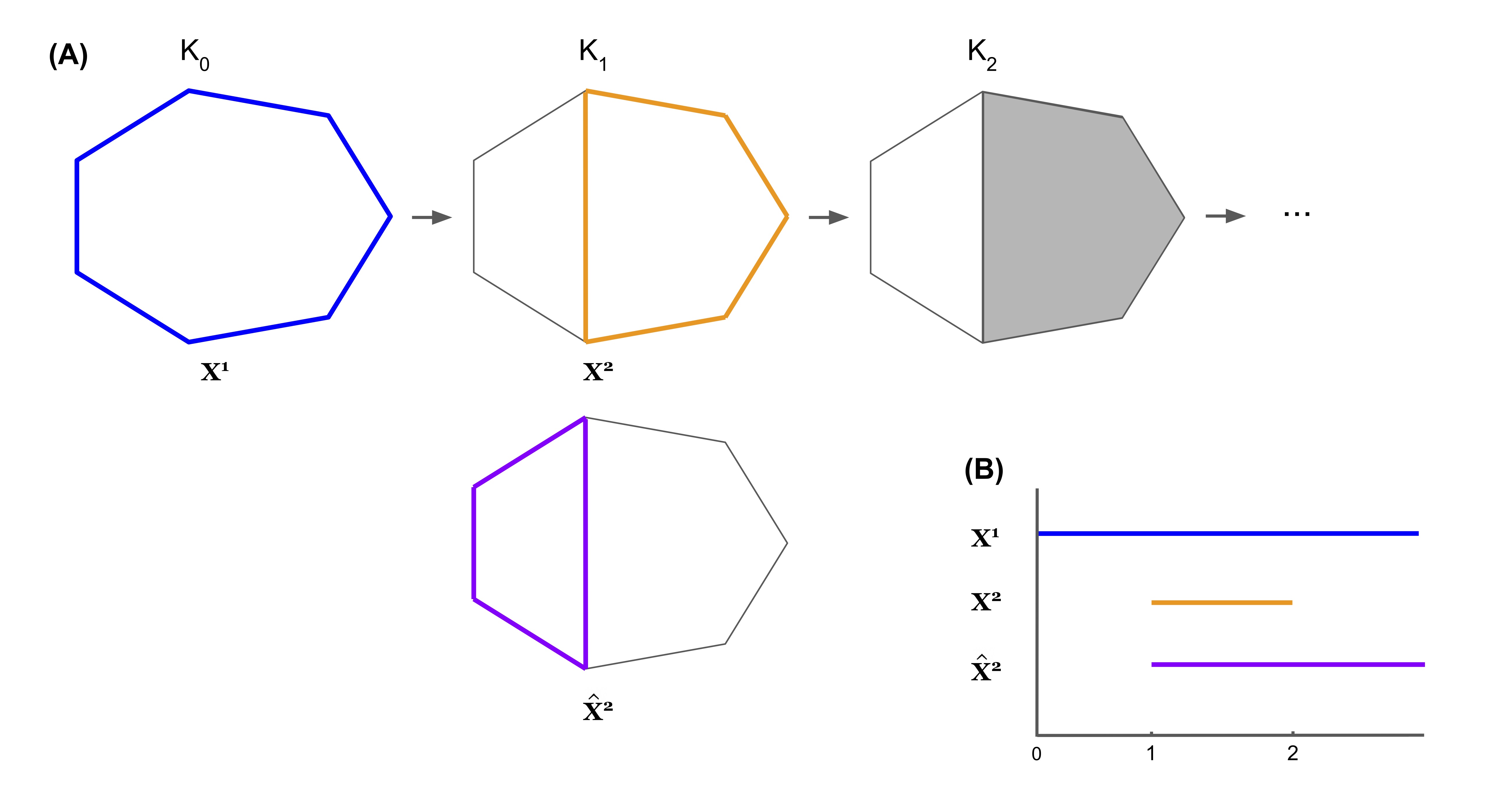}
\end{center}
\caption{An example where the optimal cycles obtained from \pr \eqref{eq:escolarargmin} do not form a persistent homological cycle basis. The thickened colored cycles in Subfigure (\textbf{A}) represent a cycle representative for the hole it encloses, and the bar with the corresponding color in Subfigure (\textbf{B}) records the lifespan of the cycle. In Subfigure (\textbf{A}), we see $\persinterval(\optimalrep^1) = [0,\infty), \persinterval(\optimalrep^2) = [1,2).$ Then, $\{\optimalrep^1, \optimalrep^2\}$ forms a basis for the persistent homological cycles. The cycle representative $\hat \optimalrep^2$ is an optimal cycle representative obtained by solving \pr (\ref{eq:filteredminimalbasis}) for the filtered simplicial complex $K_2$. However, $\persinterval(\hat \optimalrep_2) = [1, \infty)$, and thus  $\{\optimalrep^1, \hat \optimalrep^2\}$ is no longer a persistent homological cycle basis.} \label{fig:example-persBasis}
\end{figure}

\begin{remark}
\label{rmk:filteredversuspersistent}
It is important to distinguish between  $\setofpersistenthcyclebases$ and $\setoffilteredcyclebases$, hence between the optimization Programs \eqref{eq:persistentminimalbasis} and \eqref{eq:filteredminimalbasis}.  As Escolar and Hiraoka \cite{Escolar2016} point out, given $\hcyclebasis \in \setofpersistenthcyclebases$ and $\fcyclebasis \in \setoffilteredcyclebases$, one can always find an injective function $\phi: \hcyclebasis \to \fcyclebasis$ such that $\birth(\cycle) = \birth(\phi(\cycle))$ for all $\cycle$.  However, this does not imply that $\phi(\hcyclebasis) \in \setofpersistenthcyclebases$, as the deaths of each cycle may not coincide.  Indeed, the question of whether a persistent homological cycle basis can be extracted from $\fcyclebasis$ \emph{by any means} is an open question, so far as we are aware. We provide an example in \fig \ref{fig:example-persBasis} where the cycle basis obtained by optimizing each cycle using \pr \eqref{eq:filteredminimalbasis} is not a persistent homology cycle basis $\hcyclebasis$. 
\end{remark}

Though Remark \ref{rmk:filteredversuspersistent} is a bit disappointing for those interested in persistent homology, the machinery developed to study \pr \eqref{eq:filteredminimalbasis} is nevertheless interesting, and we will discuss an adaptation.

\subsection{Volume-optimal cycles: minimizing over bounding chains}\label{sec:volume}

Schweinhart \cite{schweinhart2015statistical} and  Obayashi \cite{Obayashi2018} consider a different notion of minimization: \emph{volume}\footnote{This notion of volume differs from that of \cite{chenhardness}. The latter refers to volume as the $\ell_0$ norm of a chain, while the former (which we discuss in this section) refers to the $\ell_0$ norm of a \emph{bounding} chain.} optimality.  This approach focuses on the ``size'' of a bounding chain; it is specifically designed for cycle representatives in a persistent homological cycle basis.

Obayashi \cite{Obayashi2018} formalizes the approach as follows.  First, assume a simplex-wise filtration $K_\bullet$; without loss of generality, $K_\bullet = (K_1, \ldots, K_T)$, and we may enumerate the simplices of $K_T$ such that $K_i = \{\simplex_1, \ldots, \simplex_i\}$ for all $i$.  Since each simplex has a unique birth time, each interval in  $\barcode_n(K_\bullet)= \{[b_1, d_1), \ldots, [b_N, d_N)\}$ has a unique left  endpoint.  Fix $[b_i,d_i) \in \barcode_n(K_\bullet)$ such that $d_i < \infty$ (in the case $d_i = \infty$, volume is undefined).    It can be shown that $\sigma_{b_i}$ is an $n$-simplex and  $\sigma_{d_i}$ is an $(n+1)$-simplex.

A \emph{persistent volume} $\volvec$ for $[b_i, d_i)$ is an $(n+1)$ chain $\volvec \in \Chains_{n+1}(K_{d_i})$ such that\footnote{If we regard $\partial_{n+1}\volvec$ as a function $\Simplices_{n}(K_{d_i}) \to \Q$, then $(\partial_n \volvec)_\tau$ is the value taken by $\partial_n \volvec$ on simplex $\tau$.  Alternatively, if we regard $\partial_n \volvec$ as a linear combination of $n$-simplices, then $(\partial_n \volvec)_\tau$ is the coefficient placed by $\partial_n \volvec$ on $\tau$.}
\begin{align}
    \volvec   & = \sigma_{d_i} + \sum_{\sigma_k \in \mathcal{F}_{n+1}} \alpha_k\sigma_k \label{obacond1} \\
    (\partial_{n+1} \volvec)_\tau  & = 0 \quad \forall \tau \in \mathcal{F}_n \label{obacond2}\\
    (\partial_{n+1} \volvec)_{\sigma_{b_i}}  & \ne 0, \label{obacond3}
\end{align}
where $\mathcal{F}_n = \{\sigma_k \in \Simplices_n(K) : b_i < k < d_i \}$ denotes the $n$-simplices alive in the window between the birth and death time of the interval under consideration.   

We interpret these equations as follows: Given a persistence interval $[b_i,d_i)$, condition \eqref{obacond1} implies that $\volvec$ only contains $n+1$-simplices born between $b_i$ and $d_i$ and must contain the $n+1$-simplex born at $d_i$. Condition \eqref{obacond2} ensures that the boundary of $\volvec$ contains no $n$-simplex born after $b_i$, and condition \eqref{obacond3} ensures that the boundary of $\volvec$ contains the $n$-simplex born at $b_i$. This guarantees that $\partial_{n+1}\volvec$ exists at step $b_i$, does not exist before step $b_i$, and dies at step $d_i$.

\begin{theorem}[Obayashi \cite{Obayashi2018}]  
\label{thm:obayashi}
Suppose that $[b_i, d_i) \in \barcode_n(K_\bullet)$ and  $d_i < \infty$.
    \begin{enumerate}
        \item Interval $[b_i, d_i)$ has a persistent volume.
        \item If $\volvec$ is a persistent volume for $[b_i, d_i)$ then $\persinterval(\partial_{n+1}\volvec) = [b_i, d_i)$.
        \item Suppose that $\hcyclebasis$ is an $n$-dimensional persistent homological cycle basis for $K_\bullet$, that $\originalrep \in \hcyclebasis$ is the basis vector corresponding to $[b_i, d_i)$, and that $\volvec$ is a persistent volume for $[b_i, d_i)$.  Then, $(\hcyclebasis \backslash \{\originalrep\}) \cup \{\partial_{n+1}\volvec\} $  
        is also a persistent homological cycle basis.
    \end{enumerate}
\end{theorem}

By Theorem \ref{thm:obayashi}, for any barcode composed of finite intervals, one can construct a persistent homological cycle basis from nothing but (boundaries of) persistent volumes!  Were we to build such a basis, it would be natural to ask for volumes that are optimal with respect to some loss function; that is, we might like to solve
\begin{align}
\begin{split}
    \text{minimize } & \loss(\volvec) \\
    \text{subject to } 
    & \eqref{obacond1}, \eqref{obacond2}, \eqref{obacond3}\\
    & \textbf{v} \in \Chains_{n+1}(K_{d_i}) 
\end{split}
\label{eq:generalminimalvolume}
\end{align}
for each barcode interval $[b_i, d_i)$.  A solution $\volvec$ to \pr \eqref{eq:generalminimalvolume} is called an \emph{optimal volume}; its boundary, $\optimalrep=\partial_{n+1}\volvec$ is called a \emph{volume-optimal cycle}.

It is interesting to contrast $\ell_0$-minimal cycle representatives for an interval\footnote{Technically, this notion is not well-defined; to be formal, we should fix a persistent homology cycle basis $\hcyclebasis$, fix a cycle representative $\cycle \in \hcyclebasis$ with lifespan interval $[b_i, d_i)$, and ask for an $\ell_0$ cycle representative in the same homology class, $[\cycle] \in \Homologies_n(K_{b_i})$, as per \pr \eqref{eq:homologous}.  However, in simple cases the intended meaning is clear.} $[b_i, d_i)$ with  $\ell_0$ \emph{volume}-optimal cycle for the same interval.  Consider, for example, \fig \ref{fig:volumeoptimal}.  For the persistence interval $[b_i,d_i)$, the cycle with minimal number of edges is $(a,b) + (b,c) + (c,d)  + (d,a)$. However, the volume-optimal cycle would be found as follows: considering $K_{d_i}$, we must find the fewest $2$-simplices whose boundary captures the persistence interval. In this case, we would have an optimal volume  $(a,b,e) + (b,c,e) + (a,d,e)$ and volume-optimal cycle $(a,b)+ (b,c) + (c,e) + (e,d)+ (d,a)$.

 \begin{figure}[h!]
\begin{center}
\includegraphics[width=1\textwidth]{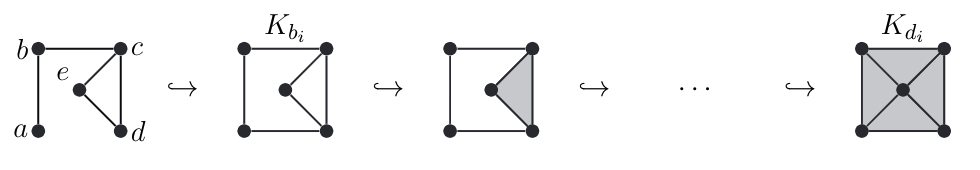}
\end{center}
\caption{A situation in which a volume-optimal cycle is different from the uniform minimal cycle. Consider the filtered simplicial complex pictured. For the persistence interval $[b_i,d_i)$, the cycle with minimal $0$-norm (fewest number of edges) is $(a,b) + (b,c) + (c,d)  + (d,a)$.
However, the volume-optimal cycle would be found as follows: considering $K_{d_i}$, we must find the fewest $2$-simplices whose boundary captures the persistence interval. In this case, we would have an optimal volume $(a,b,e) + (b,c,e) + (a,d,e)$ and volume-optimal cycle $(a,b) + (b,c) + (c,e) + (e,d) + (d,a)$. 
}\label{fig:volumeoptimal}
\end{figure}

\subsection{$\ell_0$ versus $\ell_1$-optimization} \label{secl0l1}

As mentioned above, it is common to choose $\loss(\optimalrep) = ||\optimalrep||_0$ or $\loss(\optimalrep) = ||\optimalrep||_1$.\footnote{Other choices of loss function, e.g. the $\ell_p$ norm, are common throughout mathematical optimization.  While we focus on $\ell_0$ and $\ell_1$ due to their tendency to produce sparse solutions, other choices may be better or worse suited, depending on the intended application.  For example, since $\ell_2$ loss  imposes lighter penalties on small errors and heavier penalties on large ones (as compared to $\ell_1$), it is especially sensitive to outliers; this makes it useful for tasks such as function estimation.  On the other hand, by imposing relatively heavy penalties on small errors, $\ell_1$ loss encourages sparsity \cite{dohono,NPhardL0}.} A linear program (LP) with $\ell_1$ objective function is polynomial time solvable. However, an objective function with the $\ell_0$ norm restricted to $\{0,1,-1\}$-coefficients is often preferred as the output of such a problem is highly interpretable: a cycle representative with minimal number of edges or enclosing the minimal number of triangles. Yet, $\ell_0$-optimization is known to be NP-hard \cite{NPhardL0}.

The  $\ell_1$ norm promotes sparsity and often gives a good approximation of $\ell_0$ optimization \cite{dohono,NPhardL0}, but the solution may not be exact. Yet, if all of the coefficients of the solution $\optimalrep$ are restricted to $0$ or $\pm 1$ in the optimization problem, then the $\ell_0$ and $\ell_1$ norms are identical. A looser restriction, as proposed in Escolar et al. \cite{Escolar2016}, would be to solve an optimization with $\ell_1$ objective function with integer constraints on the solution.

Requiring the solution to be integral also allows us to understand the optimal solution more intuitively than having fractional coefficients. Such an optimization problem is called a \textit{mixed integer program} (MIP), which is known to be slower than linear programming and is NP-hard \cite{Obayashi2018}. Many variants of integer programming special to optimal homologous cycles, in particular, have been shown to be hard as well \cite{borradaile2020minimum}. In \se \ref{methodsProblems}, we discuss the optimization problems we implement, where each is solved both as an LP with an $\ell_1$-norm in the objective function and an MIP by adding the constraint that $\optimalrep$ is integral. 

Dey et al. \cite{dey2011optimal} gives the \textit{totally-unimodularity} sufficient condition which guarantees that an LP and MIP give the same optimal solution. A matrix is totally unimodular if the determinant of each square submatrix is $-1, 0$, or $1$. Dey et al. \cite{dey2011optimal} give conditions for when the $\partial_{n+1}$ matrix is totally unimodular. If the totally-unimodularity condition is not satisfied, then an LP may not give the desired result. As totally unimodularity is not guaranteed for all boundary matrices \cite{henselman2014combinatorial}, we cannot rely on this condition. 

\subsection{Software implementations}
\label{sec:existingimplementations}

\emph{Edge-minimal cycles}  Software implementing the edge-loss method introduced in \cite{Escolar2016} can be found at \cite{OptiPersLP}.  This is a C++ library specialized for 3d point clouds.

\emph{Triangle-loss optimal cycles} The volume optimization technique introduced in \cite{Obayashi2018} is available through the software platform \emph{HomCloud}, available at  \cite{homcloud}.  The code can be accessed by unarchiving the HomCloud package  (for example,
\url{https://homcloud.dev/download/homcloud-3.1.0.tar.gz}) and picking the
file \url{homcloud-x.y.z/homcloud/optvol.py}.

\section{Programs and solution methods}\label{methodsProblems}
\label{sec:programsandmethods}

The present work focuses on linear  programming (LP) and mixed integer programming (MIP) optimization of $1$-dimensional persistent homology cycle representatives with $\Q$-coefficients. While the methods discussed below can be applied to any homological dimension, we limit the scope of the present work to dimension one. As described in Section \ref{problem formulation}, we follow two general approaches: those that measure loss as a function of $n$-simplices, and those that measure loss as a function of $n+1$-simplices.  Motivated by the $n=1$ case, we refer to the former as \emph{edge-loss} methods and the latter as \emph{triangle-loss} methods.  
For our empirical analysis, four variations (corresponding to two binary parameters) are chosen from each approach, yielding a total of 8 distinct optimization problems.

Concerning implementation, we find that triangle-loss methods (namely, \cite{Obayashi2018}) can be applied essentially as discussed in that paper.  The greatest challenge to implementing this approach is the assumption of an underlying simplex-wise filtration. This necessitates parameter choices and preprocessing steps not included in the optimization itself; we discuss how to execute these steps below.  
 
Implementation of edge-loss methods is slightly more complex.  For binary coefficients ($G = \field_2$) a variety of combinatorial techniques have been implemented in dimension 1 \cite{chenquantifying, zhang2019heuristic}.  Escolar and Hiraoka \cite{Escolar2016} provide an approach for $\Q$-coefficients, but in general this may not yield a persistent homology cycle basis, see Remark \ref{rmk:filteredversuspersistent}.   
In addition to the triangle-loss method mentioned in \se \ref{sec:volume}, Obayashi \cite{Obayashi2018} introduces a modified form of this edge-loss method which \emph{does} guarantee a persistent homology basis, but assumes a simplex-wise filtration.  We show that this approach can be modified to remove the simplex-wise filtered constraint.

Neither of the approaches presented here is guaranteed to solve the minimal persistent homology cycle basis problem, \eq \eqref{eq:persistentminimalbasis}.  In the case of triangle-loss methods, this is due to the (arbitrary) choice of a total order on simplices.  In the case of edge-loss methods, it is due to the choice of an initial persistent homology cycle basis.  

In the remainder of this section, we present the 8 programs studied, including any modifications from existing work.

\subsection{Structural parameters}
\label{sec_structuralparams}

Each program addressed in our empirical study may be expressed in the following form
 
\begin{align}
\begin{split}
    \text{minimize } & || W \optimalrep||_1 = \sum_i w_{i,i} (\optimalrep^+ + \optimalrep^-) \\
    \text{subject to } 
    & \optimalrep = \optimalrep^+ - \optimalrep^- \\
    & \optimalrep^+, \optimalrep^- \ge 0 \\
    & \optimalrep \in \feasibleset
\end{split}
\label{eq:generalformofourprgorams}
\end{align}
where $\feasibleset$ is a space of feasible solutions and $W = (w_{i,j})$ is a diagonal matrix with nonnegative entries.  These programs vary along 3 parameters:
    \begin{enumerate}
        \item \emph{Chain dimension of $x$}.  If $\feasibleset$ is a family of 1-chains, then we say that \eqref{eq:generalformofourprgorams} is an \emph{edge-loss} program.  If $\feasibleset$ is a family of 2-chains, we say that \eqref{eq:generalformofourprgorams} is a \emph{triangle-loss} program.
        
        \item \emph{Integrality}  The program is \emph{integral} if each $x \in \feasibleset$ has integer coefficients; otherwise we call the problem \emph{non-integral}.
        
        \item \emph{Weighting}  For each loss type (edge vs.\ triangle) we consider two possible values for $W$: identity and non-identity.  In the identity case, all edges (or triangles) are weighted equally; we call this a \emph{uniform}-weighted problem.  In the non-identity case we weigh each entry according to some measurement of ``size'' of the underlying simplex (\emph{length}, in the case of edges, and \emph{area}, in the case of triangles).\footnote{These notions make sense due to our use of coefficient field $\Q$. The distance used to form a simplicial complex can be used to define length. We restrict our attention of area to points in Euclidean space.}  There is precedent for such weighting schemes in existing literature \cite{dey2011optimal, chenquantifying}.
    \end{enumerate}

Edge-loss and triangle-loss programs will be denoted $\Edge$ and $\Tri$, respectively.  Integrality will be indicated by a superscript $I$ (integer) or $NI$ (non-integer).  Uniform weighting will be denoted by a subscript $Unif$ (uniform); non-uniform weighting will be indicated by subscript $Len$ (for edge-loss programs) or $Area$ (for triangle-loss programs).  Thus, for example, $\Edge^{I}_{Len}$ denotes a length-weighted edge-loss program with integer constraints.

\subsection{Edge-loss methods}
\label{sec:edgelossmethods}

Our approach to edge-loss minimization, based on work by Escolar and Hiraoka \cite{Escolar2016}, is summarized in Algorithm \ref{alg:edge}.  As in \cite{Escolar2016}, we obtain $\optimalrep$  by taking a linear combination of $\originalrep$ with not only boundaries but \emph{cycles} as well; consequently $\optimalrep$ need not be homologous to $\originalrep$.  
\begin{algorithm}
\caption{Edge-loss persistent cycle minimization}
\label{alg:edge}
\begin{algorithmic}[1]
\STATE Compute a persistent homology basis $\hcyclebasis$ for homology in dimension 1, with coefficients in $\Q$,  using the standard matrix decomposition procedure described in the Supplementary Material. Arrange the elements of $\hcyclebasis$ into an ordered sequence $\obasis^0 = (\obasisel^{0,1}, \ldots, \obasisel^{0,m})$.
\FOR{$j = 0, \ldots, m-1$}
\STATE Solve Program \eqref{eq:edgelossgeneral} to optimize the $j+1$th element of $\obasis^{j}$.  Let $\optimalrep$ denote the solution to this problem, and define $\obasis^{j+1}$ by replacing the $j+1$th element of $\obasis^{j}$ with $\optimalrep$.  Concretely, $\obasisel^{j+1,j+1} = \optimalrep$, and $\obasisel^{j+1,k} = \obasisel^{j,k}$ for $k \neq j$.
\ENDFOR
\STATE Return $\hcyclebasis^*: = \{\obasisel^{m,1}, \ldots, \obasisel^{m,m}\}$, the set of elements in $\obasis^m$.
\end{algorithmic}
\end{algorithm}

Our pipeline differs from \cite{Escolar2016} in three respects.  First, we perform all optimizations \emph{after} the persistence calculation has run.   On the one hand, this means that our persistence calculations  fail to  benefit from the memory advantages offered by optimized cycles; on the other hand, separating the calculations allows one to ``mix and match'' one's favorite persistence solver with one's favorite linear solver, and we anticipate that this will be increasingly important as new, more efficient solvers of each kind are developed.  Second, we introduce additional constraints which guarantee that $\hcyclebasis^* \in \setofpersistenthcyclebases$  
(and, moreover, $\persinterval( \optimalrep) = \persinterval( \originalrep)$ for each $\originalrep \in \hcyclebasis$). Third, we remove the hypothesis of a simplex-wise filtration; this requires some technical modifications, whose motivation is explained in the Supplementary Material. The crux of this modification lies with the for loop, which replaces cycles that have been optimized in the cycle basis for later cycle optimization.

Program \eqref{eq:edgelossgeneral} optimizes the $j$th element of an ordered sequence of cycle representatives $\obasis = (\obasisel^1, \ldots, \obasisel^m)$.  In particular, it seeks to minimize $\originalrep := \cycle^j$.  To define this program, we first construct a matrix $A$ such that $A[:, i] = \cycle^i$ for $i = 1, \ldots, m$.  We then define  three index sets, $\goodcycleindices, \goodtriangles, \goodedges$ such that 
    \begin{align*}
        \goodcycleindices &= \{ i :  \birth(\cycle^i) \le \birth(\originalrep), \;  \death(\cycle^i) \le \death(\originalrep), \; i \neq j \} \\
        \goodtriangles &= \{\sigma \in \Simplices_{n+1}(K) : \birth(\sigma) \le \birth(\originalrep)\} 
        \\
        \goodedges &= \{\sigma \in \Simplices_n(K) : \birth(\sigma) \le \birth(\originalrep)\}
    \end{align*} 
That is, $\goodcycleindices$ indexes the set of cycles $\cycle^i$ such that $\cycle^i$ is born  (respectively, dies) by the time that $\cycle^j$ is born (respectively, dies),  excluding the original cycle $\cycle^j$ itself. Set $\goodtriangles$ is the family of triangles born by $\birth(\originalrep)$, and set $\goodedges$ is the family of edges born by $\birth(\originalrep)$. 

With these definitions in place, we now formalize the general edge-loss problem as Program \eqref{eq:edgelossgeneral}, where  $\partial_{n+1}[\goodedges,\goodtriangles]$ denotes the  submatrix of $\partial_{n+1}$ indexed by triangles born by $\birth(\originalrep)$ (along columns) and edges indexed by edges born by $\birth(\originalrep)$.  Likewise $A[ \goodedges ,\goodcycleindices]$ is the column submatrix of $A$ corresponding to cycles that are born  before the birth time of $\originalrep$ (and which die before the death time of $\originalrep$), excluding $\originalrep$ itself.

\begin{align}
\begin{split}
    \text{minimize   } & ||W \optimalrep ||_1 = \sum_{i=1}^N  (x^+_i + x_i^-)\\
   \text{subject to  } &  
      (\optimalrep^+ - \optimalrep^- )= \originalrep[\goodedges] +   \partial_{n+1}[\goodedges, \goodtriangles]  \q + A[\goodedges, \goodcycleindices] \p \\
      & \p \in \Q^{\goodcycleindices} \\
      & \q \in \Q^{\goodtriangles} \\      
      & \optimalrep \in G^{\goodedges } \\      
      & \optimalrep^+, \optimalrep^- \geq 0 
      \end{split}
      \label{eq:edgelossgeneral}
\end{align}

 Recall from \se \eqref{sec_structuralparams} that this program varies along two parameters (integrality and weighting).  In \emph{integral} programs $G = \Z$, whereas in \emph{nonintegral} programs $G = \Q$.  The weight matrix $W$ is always diagonal, but in \emph{uniform}-weighted programs $W[i,i] = 1$ for all $i$, whereas in \emph{length}-weighted programs $W[i,i]$ is the length of edge $i$.  Program \ref{eq:edgelossgeneral} thus results in four variants:
  
\begin{itemize}[leftmargin=1in]
    \item[\namedlabel{itm:edge_NIU}{$\Edge\NI\unif$}] Nonintegral edge-loss with uniform weights.
    \item[\namedlabel{itm:edge_IU}{$\Edge\I\unif$}] Integral edge-loss with uniform weights.
    \item[\namedlabel{itm:edge_NIL}{$\Edge\NI\len$}] Nonintegral edge-loss with edges weighted by length. 
    \item[\namedlabel{itm:edge_IL}{$\Edge\I\len$}] Integral edge-loss with edges weighted by length. 
\end{itemize}
  
Program \eqref{eq:edgelossgeneral} may have many more variables than needed, because $\partial_{n+1}$ is often highly singular.  Indeed, in  applications, $\partial_{n+1}$ can have hundreds or thousands of times as many columns as rows!

A simple means to reduce the size of Program \eqref{eq:edgelossgeneral}, therefore, is to replace $\goodtriangles$ with a subset $\hat \goodtriangles \subseteq \goodtriangles$ such that $\partial_{n+1}[\goodedges, \hat \goodtriangles]$ is a column basis for $\partial_{n+1}[\goodedges, \goodtriangles]$.  Replacing $\goodtriangles$ with $\hat \goodtriangles$ will not change the space of feasible values for $\optimalrep$ in Program \eqref{eq:edgelossgeneral}, but it can cut the number of decision variables significantly. In particular, one may take $\hat \goodtriangles : = \{ \simplex : R[:, \simplex] \neq 0\}$ in the $R = \partial_{n+1} V$ decomposition of $\partial_{n+1}$ described in the Supplementary Material.  We also show correctness of this choice of $\hat \goodtriangles$ there.

\subsection{Triangle-loss methods}
\label{sec:trianglelossmethdos}

Our approach to triangle-loss optimization is essentially that of Obayashi \cite{Obayashi2018}, plus a preprocessing step that converts more general problem data into the simplex-wise filtration format assumed in \cite{Obayashi2018}.  There are several noteworthy methods for time and memory performance enhancement developed in \cite{Obayashi2018}  which we do not implement (e.g.  using restricted neighborhoods $\mathcal{F}_q^{(r)}$ to reduce problem size), but which may substantially improve runtime and memory performance.
 
The original method makes the critical assumption that $K_\bullet$ is a simplex-wise filtration, more precisely, that there exists a linear order $\sigma_1 \le \cdots \le \sigma_{|K|}$ such that $K_i = \{\simplex_1, \ldots, \simplex_i\}$. This hypothesis allows one to map each finite-length interval $[i,j) \in \barcode_n(K_\bullet)$ to a unique pair of simplices $(\simplex_i, \simplex_j)$, called a \emph{birth/death pair}, where  $\sigma_i \in \Simplices_n(K)$ and $\sigma_j \in \Simplices_{n+1}(K)$.    This mapping makes it possible to formulate Program \eqref{eq:generalminimalvolume}. Unlike the general edge-loss Program \eqref{eq:generalformofourprgorams}, one can formulate Program \eqref{eq:generalminimalvolume} without ever needing to choose an initial (non-optimal) cycle.  Thus, for simplex-wise filtrations, the method of \cite{Obayashi2018} has the substantial advantage of being ``parameter free.''

However, in many applied settings the filtration $K_\bullet$ is not simplex-wise.   Indeed, even accessing information about the filtration can be difficult in modern workflows.  Such is the case, for example, for the filtered Vietroris-Rips (VR) construction. In many VR applications, the user  presents raw data in the form of a point cloud or distance matrix to a ``black box'' solver; the solver returns the barcode without ever exposing information about the filtered complex to the user. Thus, the problem of mapping intervals back to pairs of simplices has practical challenges in common applied settings.

To accommodate this more general form of problem data, we employ Algorithm \ref{alg:rdvvolumeoptimization}.  This procedure works by (implicitly) defining a simplex-wise refinement $K_\bullet'$ of $K_\bullet$, applying the method of \cite{Obayashi2018} to this refinement, then extracting a persistent homology cycle basis for the subspace of finite intervals from the resulting data.
 More details, including recovery of a complete persistent homology cycle basis with infinite intervals\footnote{Recall volume is undefined for infinite intervals.}, and a proof of correctness can be found in the Supplementary Material.

\begin{algorithm}
\caption{Triangle-loss persistent cycle minimization}
\label{alg:rdvvolumeoptimization}
\begin{algorithmic}[1]
\STATE Place a filtration-preserving linear order $\le\dimss{l}$ on $\Simplices_l(K)$ for each $l$.
\STATE Compute an $R = \partial_{n+1} V$ decomposition as described in \cite{cohen2006vines} and the Supplementary Material.  We then obtain a set $\Gamma$ 
 of birth/death pairs $(\sigma, \tau)$.
 \STATE For each $(\sigma, \tau) \in \Gamma$ such that $\birth(\sigma) < \birth(\tau)$,  put 
    \begin{align*}
        \mathcal{F}_n &:= \{\sigma' \in \Simplices_n(K) : \birth(\sigma') \le \birth(\tau), \; \sigma \lneq^{(n)} \sigma'\} 
        \\
        \mathcal{F}_{n+1} &: = \{ \tau' \in \Simplices_{n+1}(K) : \birth(\sigma) \le \birth(\tau'), \; \tau' \lneq^{(n+1)} \tau \} 
    \end{align*}
    and ${\hat {\mathcal{F}}}_{n+1}:= \mathcal{F}_{n+1} \cup \{\tau\}$.  Compute a  solution to the corresponding Program \eqref{eq:trianglelossgeneral}, and denote this solution by  $\optimalrep^{\sigma, \tau}$. 
    \STATE Put   
        $
            \hat \deathbasis: = \{ \partial_{n+1} (\optimalrep^{\sigma, \tau}) : (\sigma, \tau ) \in  \Gamma \; \text{ and } \; \birth(\sigma) < \birth(\tau)\}$ 
            and let $\hat \deathbasis' := \{ \cycle \in \calm : \death(\cycle) = \infty  \}$, where $\calm$ is a persistent homology cycle basis calculated by the standard $R=DV$ method.
    \STATE Return $\deathbasis: = \hat \deathbasis \cup \hat \deathbasis'.$
\end{algorithmic}
\end{algorithm}

A key component of Algorithm \ref{alg:rdvvolumeoptimization} is Program \eqref{eq:trianglelossgeneral}, which we refer to as the \emph{triangle-loss program}.
\begin{align}
\begin{split}
 \text{minimize } & ||W \mathbf{v} ||_1 = \sum_{i=1}^N (v_i^+ + v_i^-)  \\
\text{subject to } &  \partial_{n+1}[ \sigma , \hat {\mathcal{F}}_{n+1} ] \volvec \neq 0     \\
&  \partial_{n+1}[\mathcal{F}_n, \hat {\mathcal{F}}_{n+1} ] \volvec = 0 \\
 & \volvec_{\tau} = 1\\
     & \mathbf{v}^+, \mathbf{v}^- \ge 0 \\
& \mathbf{v}^+, \mathbf{v}^- \in G^{ \hat {\mathcal{F}}_{n+1}}
\end{split}
\label{eq:trianglelossgeneral}
\end{align} 
This terminology is motivated by the special case $n=1$, which is our focus for  empirical studies.  As with the general edge-loss program, Program \eqref{eq:trianglelossgeneral} varies along two  parameters (integrality and weighting).  In \emph{integral} programs $G = \Z$, whereas in \emph{nonintegral} programs $G = \Q$.  The weight matrix $W$ is always diagonal, but in \emph{uniform}-weighted programs $W[\upsilon, \upsilon] = 1$ for all $\upsilon$, whereas in \emph{area}-weighted programs $W[\upsilon, \upsilon]$ is the area of triangle $\upsilon$.\footnote{We compute the area of a $2$-simplex using Heron's Formula.  We calculate area only for VR complexes whose vertices are points in Euclidean space, though more general metrics could also be considered.} Program \eqref{eq:trianglelossgeneral} thus results in four variants:

\begin{center}
\begin{itemize}[leftmargin=1in]
     \item[\namedlabel{itm:tri_NIU}{$ \Tri\NI\unif$}] Nonintegral triangle-loss with uniform weights.
     \item[\namedlabel{itm:tri_IU}{ $\Tri\I\unif$}] Integral triangle-loss with uniform weights.
    \item[\namedlabel{itm:tri_NIA}{ $\Tri\NI\area$}] Nonintegral triangle-loss with edges weighted by area. 
      \item[\namedlabel{itm:tri_IA}{ $\Tri\I\area$}] Integral triangle-loss with edges weighted by area. 
\end{itemize}
 \end{center}

\begin{remark} Algorithm \ref{alg:rdvvolumeoptimization} offers an effective means to apply the methods of \cite{Obayashi2018} to some of the most common data sets in TDA.  However, this is done at the cost of parameter-dependence; in particular, outputs depend on the choice of linear orders $\le\dimss{l}$.  
 A brief discussion on how the choice of a total order $\le$ in Algorithm \ref{alg:rdvvolumeoptimization} may impact the difficulty of the linear programs one must solve is discussed in the Supplementary Material.  In particular, we explain why the total order implicitly chosen in Algorithm \ref{alg:rdvvolumeoptimization} is reasonable,  from a computational/performance standpoint.
\end{remark}

\subsection{Acceleration techniques} \label{acceleratation technique}
 
We consider acceleration techniques to reduce the computational costs of Programs \eqref{eq:edgelossgeneral} and \eqref{eq:trianglelossgeneral}.
\vspace{0.1in}
\noindent \emph{Edge-loss methods} 
\vspace{0.1in}
The technique used for edge-loss problems aims to reduce the number of decision variables in Program \eqref{eq:edgelossgeneral}.  It does so by replacing a (large) set of decision variables indexed by $\goodtriangles$ with a much smaller set, $\hat \goodtriangles$.  See \se \ref{sec:edgelossmethods} for details.
\vspace{0.1in}
\noindent\emph{Triangle-loss methods}  
\vspace{0.1in}
When $\partial_n$ is large, the memory and computation time needed to construct the constraint matrix $\partial_{n+1}[\mathcal{F}_n, \hat {\mathcal{F}}_{n+1} ]$ can be nontrivial.  In applications that require an optimal representative for every interval in the barcode, these costs can be incurred for hundreds or even thousands of programs. We consider two ways to generate the constraint matrices $\partial_{n+1}[\mathcal{F}_n, \hat {\mathcal{F}}_{n+1} ]$ for each of the intervals in a barcode: build $\partial_{n+1}[\mathcal{F}_n, \hat {\mathcal{F}}_{n+1} ]$ from scratch for each program, or build the complete boundary matrix $\partial_{n+1}$ in advance; rather than recompute block submatrices for each program, we pass a slice of the complete matrix stored in memory.  

The difference between these two techniques can be seen as a speed/memory tradeoff.  As we will see in \se \ref{accelerateresults}, the first approach is generally faster to optimize the entire basis of homology cycle representatives, but when the data set is large, the full boundary matrix $\partial_{n+1}$ may be too large to store in memory.

\section{Experiments}\label{methods}

In order to address the questions raised in \se \ref{intro}, we conduct an empirical study of minimal homological cycle representatives in dimension one --- as defined by the optimization problems detailed in Section \ref{sec:programsandmethods} --- on a collection of point clouds, which includes both real world data sets and point samples drawn from four common probability distributions of varying dimension.  
\newcommand{\sample}{\mathbf{S}}

\subsection{Real-world data sets} \label{sec: realworlddata}

We consider $11$ real world data sets from \cite{roadmap2017}, a widely used reference for benchmark statistics concerning persistent homology computations. There are $13$ data sets considered by \cite{roadmap2017}, however, one of them (gray-scale image) is not available, and one of them is a randomly generated data set 
similar to our own synthetic data. We summarize information about the dimension, number of points, persistence computation time of each point cloud in Table \ref{tab:realworldata}. Below we provide brief descriptions of each data set, but we refer the interested reader to \cite{roadmap2017} for further details.\footnote{We use the distance matrices found on the associated github page \cite{roadmapgithub}, except in two cases. For the \textbf{Vicsek} data, we use a distance to account for the intended periodic boundary conditions of the model, and for the \textbf{genome} data, we use Euclidean distance as the distance matrix in \cite{roadmapgithub} resulted in an integer overflow error.}

\begin{enumerate}
    \item Vicsek biological aggregation model. The Vicsek model is a dynamical system describing the motion of particles.  
    It was first introduced in \cite{vicsek} and was analyzed using PH in \cite{TZH15}. We consider a snapshot in time of a single realization of the model with each point specified by its $(x,y)$ position and heading. To compute distances, the positions and headings are scaled to be between 0 and 1, and then distance is calculated on the unit cube with periodic boundary conditions. The distance between $a$ and $b$ is computed as $\min \{d(a,q) : q-b \in \{0,1,-1\}^3 \}$. We denote this data by \textbf{Vicsek}.  
    \item Fractal networks. These networks are self-similar and are used to explore the connection patterns of the cerebral cortex \cite{fractr}. The distances between nodes in this data set are defined uniformly at random by \cite{roadmap2017}. In another data set, the authors of \cite{roadmap2017} define distances between nodes by using linear weight-degree correlations. We consider both data sets and found the results to be similar. Therefore, we opt to use the one with distances defined uniformly at random.  
    We denote this data set by \textbf{fract r}.  
    \item C.elegans neuronal network. This is an undirected network in which each node is a neuron, and edges represent synapses. It was studied using PH in \cite{celegans}. Each nonzero edge weight is converted to a distance equal to its inverse by \cite{roadmap2017}.  
    We denote this data by \textbf{eleg}. 
    \item Genomic sequences of the HIV virus. This data set is constructed by taking $1,088$ different genomic sequences of dimension $673$. The aligned sequences were studied using PH in \cite{hiv} with sequences retrieved from \cite{HIVdata}. Distances are defined using the Hamming distance, which is equal to the number of entries that are different between two genomic sequences. We denote this data by \textbf{HIV}. 
    \item Genomic sequences of H3N2. This data set contains $1,173$ genomic sequences of H3N2 influenza in dimension $2,722$. Distances are defined using the Hamming distance. We denote this data set as \textbf{H3N2}.   
    \item Human genome. This is a network representing a sample of the human genome studied using PH in \cite{celegans}, which was created using data retrieved from \cite{genome}. 
    Distances are measured using Euclidean distance. We denote this data set by \textbf{genome}.
    \item US Congress roll-call voting networks. In the two networks below, each node represents a legislator, and the edge weight is a number in $[0,1]$ representing the similarity of the two legislators' past voting decisions. Distance between two nodes $i,j$ are defined to be $1-w_{i,j}$. 
    \begin{enumerate}
        \item \textbf{House}. This is a weighted network of the House of Representatives from the 104th United States Congress.  
        \item \textbf{Senate}. This is a weighted network of the Senate from the 104th United States Congress.
    \end{enumerate}
    \item Network of network scientists. This data set represents the largest connected component of a collaboration network of network scientists \cite{newman2006finding}. The edge weights indicate the number of joint papers between two authors. Distances are defined as the inverse of edge weight. 
    We denote this data set by \textbf{network}. 
    \item Klein. The Klein bottle is a non-orientable surface with one side. This data set was created in \cite{roadmap2017} by linearly sampling $400$ points from the Klein bottle using its  `figure-8' immersion in $\mathbb{R}^3$. This data set originally contains $39$ duplicate points, which we remove. Distances are measured using the Euclidean distance. 
    We denote this data set by \textbf{Klein}. 
    \item Stanford Dragon graphic. This data set contains $1000$ points sampled uniformly  at random by \cite{roadmap2017} from $3$-dimensional scans of the dragon \cite{drag}. Distances are measured using the Euclidean distance. We denote this data set \textbf{drag}. 
\end{enumerate}

\begin{table}[!h]
\caption{Summary of the experimental results of the data sets from \cite{roadmap2017} as described in \se \ref{sec: realworlddata}. The rows include the ambient dimension, number of points, the number of cycle representatives in $\Homologies_1$, and the time (measured in seconds) it took to compute persistent homology for each data set. We also include the computation time taken to optimize the set of cycle representatives under six different optimization problems, and computation time of two different implementation choices for the triangle-loss optimal cycles: building the full $\partial_2$ boundary matrix once and extracting the part needed, or constructing part of the $\partial_2$ boundary matrix for each cycle representative. In this table, $T$ stands for computation time measured in seconds with subscripts indicating the type of the optimal cycle and superscripts indicating whether the program was solved using linear programming (NI) or integer programming (I). The time taken to construct the input to the optimization problem is included in the optimization time for edge-loss minimal cycle representatives, but is excluded and separately listed in the last two rows for the triangle-loss minimal cycle representatives. For triangle-loss cycles, we were able to compute $115$ out of the $117$ cycle representatives for the \textbf{Genome} data set and $52$ out of $57$ cycle representatives for the \textbf{H3N2} data set due to memory constraints. The numbers in the parenthesis represent the other optimization statistics corresponding to the triangle-loss optimal cycles we were actually able to compute. The last two rows compare two ways of building the input $\partial_2[:,\hat {\mathcal{F}}_{2}]$ matrix to the triangle-loss optimal cycle program. The penultimate row records the time of building the entire $\partial_{2}$ matrix once and then extracting columns born in the interval $\closedinterval$ for each representative. The last row records the total time to iteratively build the part of the boundary matrix $\partial_2[:,\hat {\mathcal{F}}_{2}]$ for each cycle representative.}
    {\footnotesize{ 
    \begin{tabular}{ |>{\centering}m{8em} *{11}{>{\centering\arraybackslash}m{4em} }|}
 \hline
  & \textbf{Klein} & \textbf{Vicsek}  & \textbf{C.elegans} & \textbf{HIV} & \textbf{genome} & \textbf{fractal R} & \textbf{network} & \textbf{house} & \textbf{senate} & \textbf{drag} & \textbf{H3N2}\\[0.5ex] 
 \hline 
 \hline
 \textbf{Ambient dimension} & 3 & 3 & 202 &  673  & 688 &  259 & 300 & 261 & 60 & 3 &  1,173\\   
 \textbf{\# points}   & 400 &  300  &  297 &   1088 &  1397    & 512 & 379 & 445  & 103 & 1,000 & 2,722\\ 
 \textbf{\# representatives}& 257  & 149  &  107 &174  &  117 (115)   & 438 & 7  & 126  & 12 & 311 & 28 (26) \\  
 \textbf{$T_\textbf{persistence}$}   & 100.97  & 129.39 & 5.14 &728.51  & 967.61  & 143.07 & 12.18  &  9.62 & 0.10 & 1,053.53 & 71,081.77  \\ 
 [0.5ex] 
\hline
\multicolumn{7}{c}{\textbf{\qquad Edge-loss persistent homological cycle representatives (\pr \eqref{eq:edgelossgeneral})}}  &&&&& \\
\hline
 \textbf{$T\I\EL$  } & 16.01  & 8.20 &  19.64&466.85 & 656.05 &   150.46 & 0.17 & 63.93  & 0.31 & 45.14 & 4,732.59	\\ 
 \textbf{$T\NI\EL$  } & 11.28  & 6.61 &16.07  &403.63 & 491.69  &    86.95 & 0.13 & 48.65  & 0.22 & 34.73 & 4,540.55 \\ 
 \textbf{$T\I\EU$  } & 14.59  & 9.09 & 19.22 & 473.82 & 689.51&    119.94 & 0.23 &  63.34 & 0.33 & 45.51 & 4,714.90	\\ 
  \textbf{$T\NI\EU$  } & 11.38  & 5.55 & 15.63 & 404.95& 492.66 & 83.40 & 0.12 & 48.88  & 0.22 & 33.88 & 4,547.37 \\
  [0.5ex] 
\hline
\multicolumn{7}{c}{\textbf{Edge-loss filtered homological cycle represnetatives (\pr \eqref{eq:escolarargmin})}} &&&&& \\
\hline
 \textbf{$T\I\EL$  } & 16.93  &8.64  &20.41  & 468.22 & 1144.17 &    155.08&  0.17 &  62.20  & 0.30 & 67.77 & 2,999.24	\\ 
 \textbf{$T\NI\EL$  } & 10.29  & 5.51 &16.15  & 403.74& 973.15 &     88.66 &0.13  &  48.24 & 0.22 & 50.25 & 2,829.12	\\ 
 \textbf{$T\I\EU$  } & 15.14  & 8.32 &19.76  & 476.84 & 1191.44 &     142.4&  0.24  & 61.82  & 0.31 & 68.63 & 2,937.16	\\ 
  \textbf{$T\NI\EU$  } & 11.07  & 5.63 & 16.23 & 406.97& 981.72  & 87.59 &  0.12 & 48.11  &0.22  & 54.05 & 2,833.06	 \\
  [0.5ex] 
\hline
\multicolumn{7}{c}{\textbf{Triangle-loss persistent homological cycle representatives (\pr \eqref{eq:trianglelossgeneral})\qquad}} &&&&& \\
\hline
 \textbf{$T\I\TU$  } & 316.33 & 24.52 & 657.53  & 25,402.56  & 16,379.86  &  20,440.33   & 2.91 & 234.05  & 0.29  & 384.91 & 39,140.67 \\ 
  \textbf{$T\NI\TU$  } & 154.36 & 19.18 & 540.06 & 23,260.12  & 14,535.42  &   18,279.82  & 2.47 & 206.63  & 0.18  & 277.93 & 36,401.50  \\ 
 \textbf{$T$ build all}  & 2.16 & 0.32  & 4.88 & 268.57 & - & 138.46 & 0.06 &  6.23 & 0.03  & 5.94 & - \\ 
\textbf{Total $T$ build part}  & 9.18 & 3.51  & 28.47 & 1,688.10 & 415.79& 917.42 & 0.28 & 45.02  & 0.05  & 106.64 & 1,236.80 \\ \hline 
\end{tabular}
} 
}
\label{tab:realworldata}
\end{table}

\setlength{\tabcolsep}{10pt}

\renewcommand{\arraystretch}{1.5}

\subsection{Randomly generated point clouds}\label{sec: randompointclouds}
We also generate a large corpus of synthetic point clouds, each containing $100$ points in $\mathbb{R}^q$ with $q = 2,\ldots,10$, drawn from normal, exponential, gamma, and logistic distributions. We produce $10$ realizations for each distribution and dimension combination, for a total of $360$ randomly generated point clouds. We use Euclidean distance to measure similarity between points and the Vietoris- Rips filtered simplicial complex to compute persistent homology. 

\subsection{Erd\H{o}s-R\'enyi random complexes}\label{sec:erdos}

To investigate which properties of homological cycle representatives could arise as the result of the underlying geometry of the point clouds, we also consider a common non-geometric model for random complexes: Erd\H{o}s-R\'enyi random clique complexes. Here, we construct 100  symmetric dissimilarity matrices of size $100 \times 100$ by drawing entries i.i.d. from the uniform distribution on $[0,1]$ for each pair of distinct points. As these dissimilarities are fully independent, they are in particular not subject to geometric constraints like the triangle inequality. A natural filtration is placed on these dissimilarity matrices by forming filtered simplicial complex $K_\bullet = (K_{\epsilon_i})_{i \in\{ 1, \ldots, T\}}$ where $0=\epsilon_1 < \cdots < \epsilon_T=1$ to compute persistent homology.

\subsection{Computations}

For each of the data sets, we perform Algorithms \ref{alg:edge} and \ref{alg:rdvvolumeoptimization} (using Vietoris-Rips complexes with $\Q$-coefficients) to find optimal bases $\hcyclebasis^*, \deathbasis \in \setofpersistenthcyclebases$.  For comparison to the edge-loss problem in Algorithm \ref{alg:edge}, we also apply \pr \eqref{eq:escolarargmin} to each representative in the persistent homology cycle basis to find a basis $\fcyclebasis \in \setoffilteredcyclebases$.

\subsection{Hardware and software}
\label{subsec:hardwaresoftware}

We test our programs on an iMac (Retina 5K, 27-inch, 2019) with a 3.6 GHz Intel Core i9 processor and 40 GB 2667 MHz DDR4 memory.

Software for our experiments is implemented in the programming language Julia; source code is available at \cite{li_thompson}.  This code specifically implements Algorithms \ref{alg:edge}\footnote{ \pr \eqref{eq:escolarargmin} is implemented analogously.} and \ref{alg:rdvvolumeoptimization}.

Since our interest lies not only with the outputs of these algorithms but with the structure of the linear programs themselves, \cite{li_thompson} implements a standalone workflow that exposes the objects built internally within each pipeline.  This library is simple by design, and does not implement the performance-enhancing techniques  developed in \cite{Escolar2016, Obayashi2018}. Users wishing to work with optimal cycle representatives for applications may consider these approaches 
discussed in \se \ref{sec:existingimplementations}.

To implement  Algorithms \ref{alg:edge} and \ref{alg:rdvvolumeoptimization} in homological dimension one, the test library \cite{li_thompson} provides three key functions:  \emph{A novel solver for persistence with $\Q$-coefficients.} To compute cycle representatives for persistent homology with $\Q$-coefficients, we implement a new persistent homology solver adapted from  \url{Eirene}  \cite{eirenecode}.  The adapted version uses native Eirene code as a subroutine to reduce the number of columns in the top dimensional boundary matrix in a way that is guaranteed not to alter the outcome of the persistence computation \cite{eirene}.
\\
\\
\noindent\emph{Formatting of inputs to linear programs.} 

Having computed barcodes and persistent homology cycle representatives, library \cite{li_thompson}  provides built-in functionality to format the linear Programs \eqref{eq:edgelossgeneral} and \eqref{eq:trianglelossgeneral} for input to a linear solver.  This ``connecting'' step is executed in pure Julia. 
\\
\\
\noindent \emph{Wrappers for linear solvers.} \label{linear solvers} 

We use the Gurobi linear solver \cite{gurobi} and the GLPK  solver \cite{glpk}. Both solvers can optimize both LPs and MIPs. Experiments indicate that Gurobi executes much faster than GLPK on this class of problems, and thus, we use it in the bulk of our computations. Both solvers are free for academic users.

\section{Results and Discussion}\label{results}

 In this section, we investigate each of the questions raised in \se \ref{intro} with the following analyses.

\subsection{Computation time comparisons} 
\label{sec:timecomparisons}

We summarize results for Programs \ref{itm:edge_NIU}, \ref{itm:edge_IU}, \ref{itm:edge_NIL},
\ref{itm:edge_IL},
\ref{itm:tri_NIU}, and 
\ref{itm:tri_IU} in \tab \ref{tab:realworldata} for data described in \se \ref{sec: realworlddata} and \tab \ref{tab:distributiondata} for data described in \se \ref{sec: randompointclouds} and \se \ref{sec:erdos}. Further, we summarize results for Programs \ref{itm:tri_NIA} and \ref{itm:tri_NIA} in \tab \ref{tab:distributiondata} for data described in \se \ref{sec: randompointclouds}.\footnote{We compute the area of a $2$-simplex using Heron's Formula for data whose distances are measured using the Euclidean distance. For data with non-Euclidean distances, we find that there are triangles that do not obey the triangle inequality, thus, we only compute area-weighted triangle-loss cycles for data described in \se \ref{sec: randompointclouds}. As such,  \ref{itm:tri_NIA}, 
\ref{itm:tri_NIA} do not appear in \tab \ref{tab:realworldata} and the Er\H{o}s R\'enyi column of \tab \ref{tab:distributiondata}.}
We use $T_\textbf{persistence}$ to denote the time taken to compute all original cycle representatives and their lifespans $\persinterval$. We use $T_\bullet^*$ to denote the computation time for optimizing all generators found by the persistence algorithm, where the subscript denotes the cost function e.g. $\EU$ or $\TU$, and the superscript denotes the nonintegral $\NI$ or integral $\I$ constraint. 
The $T_\bullet^*$ computations include the time required to construct the inputs to the solver for the edge-loss methods, and exclude the time required to construct the inputs to the triangle-loss methods, whose computation time is separately recorded in order to compare two ways of constructing the input matrix, as discussed in \se \ref{acceleratation technique}. In each table, rows 1-3 provide information about the data by specifying ambient dimension, number of points, and number of cycle representatives. Row 4, labeled as $T_\textbf{persistence}$, gives the total time to compute persistent homology for the data, measured in seconds. Rows 5-12 (\tab \ref{tab:realworldata}) and rows 5-14 (\tab \ref{tab:distributiondata}) give the total time to optimize all cycle representatives that are feasible to compute using each optimization technique. In the last two rows of each table, we provide the time of constructing the input to the triangle-loss methods using two different approaches described in \se \ref{acceleratation technique}. The penultimate row records the time of building the entire $\partial_{2}$ matrix once and then extracting $\partial_2[\mathcal{F}_1, \hat {\mathcal{F}}_{2}]$ for each representative. The last row records the total time to iteratively build the part of the boundary matrix $\partial_{2}[ \mathcal{F}_1 , \hat {\mathcal{F}}_{2} ]$ for each cycle representative. In \tab \ref{tab:distributiondata}, the computation times displayed average all random samples from each dimension for each distribution. 
\setlength{\tabcolsep}{10pt}

\renewcommand{\arraystretch}{1.5}
\begin{table}[!h]
\caption{Summary of the experimental results for the synthetic, randomly generated data sets described in \se \ref{sec: randompointclouds}. For each distribution, we sample $10$ data sets each containing $100$ points in ambient dimensions from $2$-$10$. The computation time in this table averages the $10$ random samples for each dimension and distribution combination. The number of cycle representatives is totaled over the $90$ samples for each distribution. The rows of this table are analogous to those of \tab \ref{tab:realworldata}, excluding the penultimate row of that table, as the time comparison is only done for the large real-world data sets. } 
\footnotesize
    \centering
    \begin{tabular}{ |c || c |c |c |c | c|}
 \hline
 & \textbf{Normal} & \textbf{Gamma}  & \textbf{Logistic} & \textbf{Exponential}  & \textbf{Erd\H{o}s-R\'enyi}  \\[0.5ex] 
 \hline 
 \hline
 \textbf{Ambient dimension} & 2-10 & 2-10    & 2-10 &  2-10 & NA \\\hline  
 \textbf{\# points} &  100 &  100  &  100 &   100 & 100 \\\hline  
  \textbf{Total \# representatives} & 4,815 & 3,706  & 4,456 & 3,788 & 34,214\\ \hline
 \textbf{Average  \textbf{$T_\textbf{persistence}$} (seconds)} &   2.80  & 2.12    & 2.01 & 2.63 
 & 2.20 \\  [0.5ex] \hline
\multicolumn{4}{c}{\textbf{Edge-loss persistent homological cycle representatives (\pr \eqref{eq:edgelossgeneral})}}  \\
\hline
 \textbf{Average total $T\I\EL$  } &5.52 & 6.01 & 5.65 & 5.91 & 5.99 \\ \hline
 \textbf{Average total $T\NI\EL$  } &  4.37 & 4.55 & 4.32 & 4.47 & 4.99\\ \hline 
 \textbf{Average total $T\I\EU$  } & 5.31 & 5.97 & 5.45 & 5.90 &6.16\\ \hline
 \textbf{Average total $T\NI\EU$  } & 4.08 & 4.58 & 4.23 & 4.51 & 4.87\\ 
 [0.5ex] 
\hline
\multicolumn{4}{c}{\textbf{Edge-loss filtered homological cycle representatives (\pr \eqref{eq:escolarargmin})}}  \\
\hline
 \textbf{Average total $T\I\EL$  } &5.32 & 6.46 & 6.27 & 6.88& 7.44\\ \hline
 \textbf{Average total $T\NI\EL$  } & 4.07 & 5.05 & 4.78 &5.11 & 4.69 \\ \hline 
 \textbf{Average total $T\I\EU$  } &5.23 & 6.46 & 6.25 & 6.66& 6.25\\ \hline
 \textbf{Average total $T\NI\EU$  } & 4.17 & 4.94 & 4.61 & 5.29 & 4.64\\ 
[0.5ex] 
\hline
\multicolumn{4}{c}{\textbf{Triangle-loss persistent homological cycle representatives (\pr \eqref{eq:trianglelossgeneral})}} & \\
\hline
 \textbf{Average total $T\I\TU$  } & 6.56 & 9.91 & 7.06 & 9.68 & 4.64\\ 
 \hline 
 \textbf{Average total $T\NI\TU$  }&  5.24 & 7.99 & 5.79 & 7.75 & 4.49 \\  \hline
 \textbf{Average total $T\I\TA$  } &  6.59 & 10.20 &  7.30 & 9.99 & - \\  \hline
 \textbf{Average total $T\I\TA$  } &  5.19 & 7.89 & 5.80 &  7.57 & - \\  \hline
 \hline
  \textbf{Average total $T$ build all} & 1.40 & 1.71 &  1.56 & 1.07& 1.24 \\ 
  \hline
  \textbf{Average total $T$ build part } & 3.51 & 1.54 &    1.61 & 1.56 & 0.85\\ 
 \hline
 
\end{tabular}
\label{tab:distributiondata}  
\end{table}

The two numbers in parenthesis in the third row of \tab \ref{tab:realworldata} indicate the actual number of representatives we were able to optimize using the triangle-loss methods (all edge-loss representatives were optimized). For the \textbf{genome} and \textbf{H3N2} data sets, we are not able to compute all triangle-loss cycle representatives due to the large number of 2-simplices born between the birth and death interval of some cycles. For instance, for a particular cycle representative in the \textbf{genome} data set, there were $10{,}522{,}991$ 2-simplices born in this cycle's lifespan. 
Also, given the large number of $2$-simplices in the simplicial complex, we are not able to build the full $\partial_2$ matrix due to memory constraints, denoted by - in the penultimate row of \tab \ref{tab:realworldata}. 

Below we describe some insights on computation time drawn from the two tables. 
\\
\\
\noindent \emph{Persistence and optimization ($T_\textbf{persistence}$ vs. $T_\bullet^*$)}

We observe that $T_\bullet^*$\footnote{Including the time of constructing the input to the optimization programs.} $>$ $T_\textbf{persistence}$  e.g. for 5 out of the 11 real-world data sets described in \se \ref{sec: realworlddata} when using the 4 edge-loss methods.  The same inequality holds in 9 out of the 11 data sets when using the two uniform-weighted triangle-loss methods. For all of the synthetic data described in Sections \ref{sec: randompointclouds} and \ref{sec:erdos}, we have $T_\bullet^* > T_\textbf{persistence}$  when using all eight optimization programs. Therefore, the computational cost of optimizing a basis of cycle representatives generally exceeds the cost of computing such a basis.

This somewhat surprising result highlights the computational complexity of the algorithms used both to compute persistence and to optimize generators.  A common feature of both the persistence computation and linear optimization is that empirical performance  typically outstrips asymptotic complexity by a wide margin; the persistence computation, for example, has cubic complexity in the size of the complex, but usually runs in linear time.  Thus, worst-case complexity paints an incomplete picture.   Moreover, naive ``back of the envelope'' calculations are often hindered by lack of information.  For example, the persistence computation (which essentially reduces to Gaussian elimination) typically processes each of the $m$ columns of a boundary matrix $\partial_n$ in sequence. The polytope of feasible solutions for an associated linear program (edge-loss or triangle-loss) may have many fewer or many more vertices than $m$, depending on the program; moreover, even if the number of vertices is very high, the number of \emph{visited} vertices (e.g., by the simplex algorithm) can be much lower.  Without knowing these numbers \emph{a priori}, run times can be quite challenging to estimate. Empirical studies, such as the present one, give a picture of how these algorithms perform in practice.
\\
\\
\noindent \emph{Integral and nonintegral programs ($T^I$ vs. $T^{NI}$)}

In Tables \ref{tab:realworldata} and \ref{tab:distributiondata}, we observe that $T^I$ $>$ $T^{NI}$, i.e., the total computation time of optimizing a basis of cycle representatives using an integer program exceeds the computation time using a non-integer constrained program. Yet, $T^I$ and $T^{NI}$ are on the same order of magnitude, for both edge-loss methods and triangle-loss method.

Let $r\EU = \frac{t\I\EU}{t\NI\EU},$ where $t^*_\bullet$ represents the computation time for optimizing a single cycle representative. We define $r\EL$ and $r\TU$ similarly. We compute each for every cycle representative for data described in Tables \ref{tab:realworldata} and \ref{tab:distributiondata}. Let $\bar{r}_\bullet$ denote the average of $r_\bullet$ and $\sigma_{r_\bullet}$ denote the standard deviation of $r_\bullet$. We have $\Bar{r}\EU = 1.49, \sigma_{r\EU} = 1.34$, $\Bar{r}\EL = 1.55,  \sigma_{r\EL} = 1.38$, $ \sigma_{r\TU} = 1.35, \sigma_{r\TU} = 2.86$. \fig \ref{fig:lp_mip_ratio_df}(A, C, E) plots $r_\bullet$ using scatter plots and \fig \ref{fig:lp_mip_ratio_df}(B, D, F) displays the same data using box plots. The vertical axis represents the ratio between the MIP time and LP time of optimizing a cycle representative. The horizontal axis in the scatter plots represents the computation time to solve the LP. The red line in each subfigure represents the horizontal line $y=1$. As we can see from the box plots, the ratio between the computation time of MIP and LP for most of the cycle representatives ($>50\%$) is around $1$ and less than $2$. Although there are cases where the computation time of solving an MIP is $71.03$ times the computation time of solving an LP, such cases happen only for cycle representatives with a very short LP computation time.  

\begin{figure}[]
\begin{center}
\includegraphics[width=\textwidth]{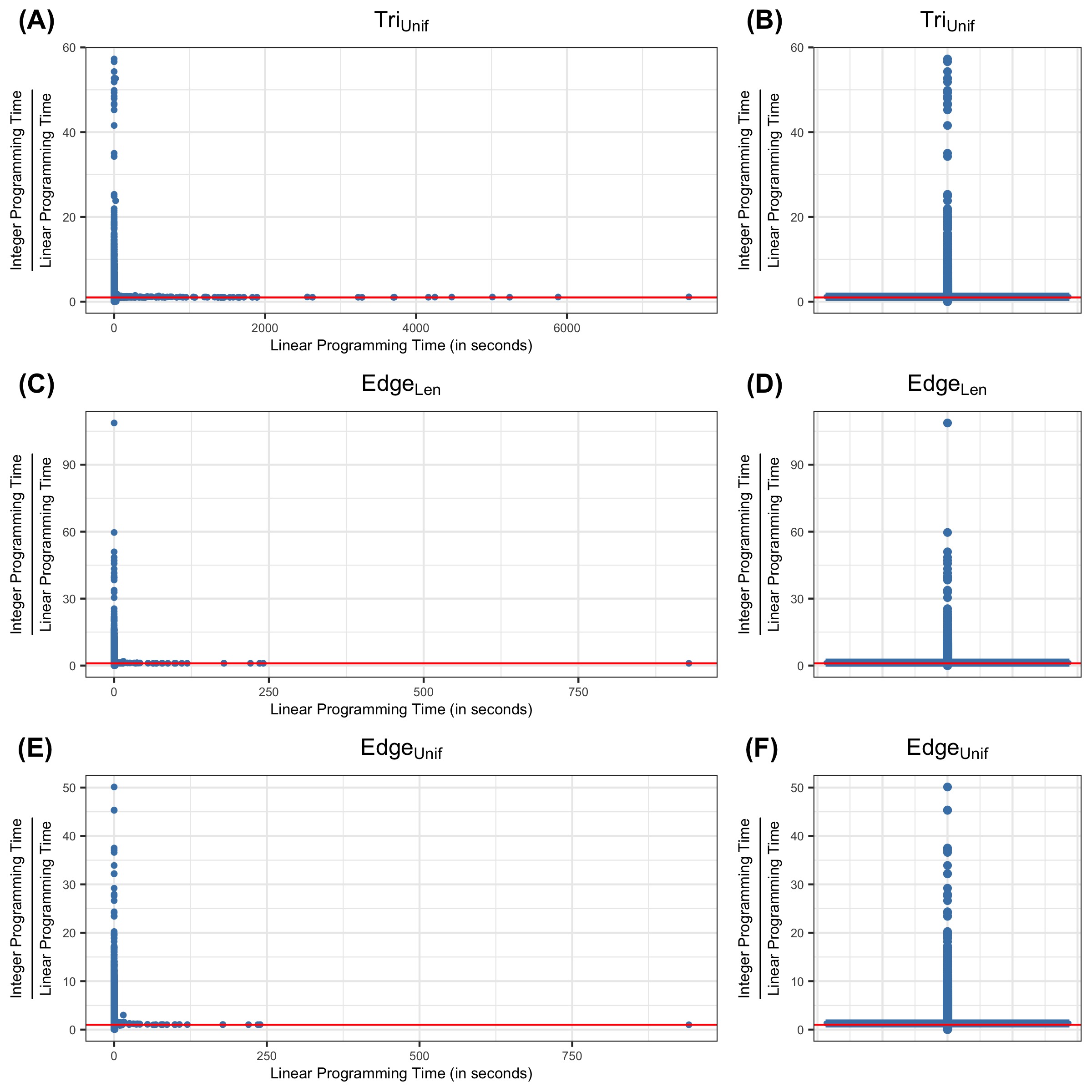} 
\end{center}
 \caption{ 
 Ratio between the computation time of solving an integer programming problem (Programs \ref{itm:tri_IU}, \ref{itm:edge_IL}, \ref{itm:edge_IU}) and the computation time of solving a linear programming problem (Programs \ref{itm:tri_NIU}, \ref{itm:edge_NIL}, \ref{itm:edge_NIU}) for all of the cycle representatives from data sets described in Sections \ref{sec: realworlddata}, \ref{sec: randompointclouds}, and \ref{sec:erdos}. Subfigures  \textbf{(A), (C), (E)} plot the data using scatter plots and subfigures  \textbf{(B), (D), (F)} show the same data using box plots. The vertical axis represents the ratio between the integer programming time and linear programming time of optimizing a cycle representative and the horizontal axis represents the computation time to solve the linear program. The red line in each subfigure represents the horizontal line $y=1$, where the time for each optimization is equivalent. As we can see from the box plots, the ratio between the computation time of integer programming and linear programming for most of the cycle representatives ($>50\%$) center around $1$.}\label{fig:lp_mip_ratio_df}
\end{figure}

\vspace{.1in}
\noindent \emph{Triangle-loss versus edge-loss programs. ($T_{T\text{-}\bullet}$ vs. $T_{E\text{-}\bullet}$)} 

We observe that the edge-loss optimal cycles are more efficient to compute than the triangle-loss cycles for more than $65.70\%$ of the cycle representatives\footnote{Obayashi \cite{Obayashi2018} proposes a few techniques for accelerating the triangle-loss methods which we did not implement.}. This aligns with our intuition because for representatives with a longer persistence, the number of columns in the boundary matrix $\partial_{2}[ \mathcal{F}_1 , \hat {\mathcal{F}}_{2} ]$ grows faster than that of $\partial_1[:, \mathcal{Q}]$. Consequently, the edge-loss programs are feasible for all cycle representatives we experiment with, whereas the triangle-loss technique fails for $6$ representatives due to the large problem size (with greater than twenty million triangles born between the life span of those cycle representatives).
 
\noindent \emph{Different linear solvers}

The choice of linear solver can significantly impact the computational cost of the optimization problems. We perform experiments on length/uniform-minimal cycle representatives using the GLPK \cite{glpk} and Gurobi \cite{gurobi} linear solvers on $90$ data sets drawn from the normal distribution with dimensions from $2$ to $10$ with a total of $4{,}815$ cycle representatives. The median of the computation time ratio between using the GLPK solver and Gurobi solver is $2.22$ for \pr
\ref{itm:edge_NIU}, $1.68$ for \pr \ref{itm:edge_IU}, $2.28$ for \pr \ref{itm:edge_NIL}, and $1.73$ for \pr \ref{itm:edge_IL}, and the computation time using the GLPK solver can be $30$ times larger than the computation time using the Gurobi solver for some cycles, see figure in the Supplementary Material. Therefore, we use the Gurobi solver in all other analyses in this paper.

\subsection{Performance of acceleration techniques} \label{accelerateresults}

\noindent \emph{Edge-loss optimal cycles} 

As discussed in \se \ref{acceleratation technique}, we accelerate edge-loss problems by replacing $\partial_2[:, \goodtriangles]$ with the column basis submatrix of $\partial_2[:, \hat \goodtriangles]$. We further reduce the size of $\partial_2[:, \hat \goodtriangles]$ by only including the rows corresponding to 1-simplices born before the birth time of the cycle, denoted as $\partial_2[\goodedges, \hat \goodtriangles]$. We perform experiments on a small-sized data set (\textbf{Senate}) that consists of 103 points in dimension $60$ and a medium-sized data set (\textbf{House}) that contains $445$ points in dimension $261$. In \tab
\ref{unif-acceleration-table}, we report the computation time of solving the optimization problems in Programs \ref{itm:edge_NIU}, \ref{itm:edge_IU}, \ref{itm:edge_NIL}, and \ref{itm:edge_IL} using these three techniques of varying the size of the input boundary matrix. The results align with intuition that the optimizations are faster with fewer input variables, and thus, the third implementation is the most efficient among the three.

\newcolumntype{L}{>{\centering\arraybackslash}m{3cm}}

\renewcommand{\arraystretch}{1.5}
\begin{table}[!h]
\centering
\caption{Computation time of three differently sized input boundary matrices to edge-loss and triangle-loss methods. The superscripts denote whether the program requires an integral solution or not, and the subscripts indicate the type of optimal cycle. All time is measured in seconds. We perform experiments on a small-sized data set (\textbf{Senate}) that consists of $103$ points in dimension $60$ and a medium-sized data set (\textbf{House}) that contains $445$ points in dimension $261$. For edge-loss methods, we consider three implementations to solve these optimization problems: using the full boundary matrix $\partial_2$, using the basis columns and all rows $\partial_2[:, \hat \goodtriangles]$, and using the basis columns and deleting rows corresponding to edges born after the birth time of the cycle $\partial_2[\goodedges, \hat \goodtriangles]$. For triangle-loss methods, we consider three approaches to solve these optimization problems: zeroing out the columns in the boundary matrix outside of $[b_i,d_i]$ denoted as $\partial_{2_{zero}}$, deleting columns outside of this range $\partial_2[:,\hat {\mathcal{F}}_{2}]$, and deleting both columns outside of $[b_i, d_i]$ and rows born after $d_i$ denoted $\goodvolmatrix$. The \textbf{House} data set was too large to implement the first method.}\label{unif-acceleration-table}
\scriptsize
\begin{tabular}{ |>{\centering}m{11em}   >{\centering\arraybackslash}m{8em}>{\centering\arraybackslash}m{8em}  >{\centering\arraybackslash}m{8em} >{\centering\arraybackslash} m{8em}|}
 \hline
 & \multicolumn{4}{c|}{\textbf{Edge-loss Optimal Cycles (\pr \eqref{eq:edgelossgeneral})}} \\
\cline{3-4}
  & \textbf{T}  & \textbf{$ \partial_2$}  & \textbf{$\partial_2[:, \hat \goodtriangles]$}  & \textbf{$\partial_2[\goodedges, \hat \goodtriangles]$}  \\  [0.5ex]  \hline \hline
    \multirow{4}{*}{\textbf{Small Data Set (Senate)}} & 
 $T\EU\NI$ & 1.06& 1.03 &	0.41  \\  &
  $T\EU\I$ &1.25 &1.23	& 0.60 \\  &
    $T\EL\NI$ &1.05&  1.05 &	0.41   \\   &
  $T\EL\I$  & 1.23 &1.19 & 0.65 \\ 
  \hline 
  \multirow{4}{*}{\textbf{Medium Data Set (House)}} & 
 $T\EU\NI$ & 184.70 & 122.72 &	47.10  \\ &
  $T\EU\I$ &188.88 & 147.27	&  64.64 \\  &
    $T\EL\NI$ &184.41&  121.80 &	46.02    \\   &
  $T\EL\I$ & 193.01 & 146.46 & 63.87 \\ [0.5ex] \hline \hline
   & \multicolumn{4}{c|}{\textbf{Triangle-loss Optimal Cycles (\pr \eqref{eq:trianglelossgeneral})}} \\ \cline{3-4}
  & \textbf{\textbf{T}}  & \textbf{$\partial_{2_{zero}}$}  & \textbf{$\partial_2[:,\hat {\mathcal{F}}_{2}]$}  & \textbf{$\goodvolmatrix$} \\[0.5ex] 
 \hline 
 \hline
 \multirow{2}{*}{\textbf{Small Data Set (Senate)}}& 
 $T\TU\NI$    & 23.25   & 0.99  & 0.59 \\  &
  $T\TU\I$   & 25.31  & 1.06   & 0.66   \\ \hline
  \multirow{2}{*}{\textbf{Medium Data Set (House)}} & 
 $T\TU\NI$   
   &  -  &	286.10 &   194.70 \\ &
  $T\TU\I$  
    & -	& 317.45  &  237.73\\\hline 
\end{tabular}
\label{tab:implementationcompare}
\end{table}

\vspace{0.1in}
\noindent \emph{Triangle-loss optimal cycles}

As discussed in \se \ref{acceleratation technique}, there are also multiple approaches to creating the input to the triangle-loss problems. To recap, we restrict the boundary matrix $\partial_2$ to $\goodvolmatrix$ for a particular cycle representative $\optimalrep^i$. We can do so in various ways: (i) zeroing out the columns of $\partial_2$ not in $\hat{\mathcal{F}}_2$ but maintaining the original size of the boundary matrix, (iia) building the entire boundary matrix $\partial_2$ once and then deleting the columns not in $\hat{\mathcal{F}}_2$ for each representative, (iib) building the columns in $\hat{\mathcal{F}}_2$ iteratively for each representative, and (iiia/b) in conjunction with (iia) or (iib) respectively, reducing the rows of the boundary matrix of $\partial_2$ to only include the rows born before the death time of the cycle $\mathcal{F}_1$. 

In \tab \ref{unif-acceleration-table}, we summarize the computation time of solving Programs 
\ref{itm:tri_NIU} and
\ref{itm:tri_IU}
 to find triangle-loss optimal cycles with three different sized boundary matrices as input: (i) zeroing out, (iib) deleting partial columns, and (iiib) deleting partial rows and columns. Note that (iia) and (iib) both result in the same boundary matrix $\partial_2[:, \hat{\mathcal{F}}_2]$. We again use the \textbf{Senate} and \textbf{House} data sets for analysis. We see that deleting partial rows and columns is the most efficient among the three implementations, which again matches intuition that reducing the number of variables accelerates the optimization problem. 

We also ran experiments on the real-world data sets to compare the timing of building $\partial_{2}[ \mathcal{F}_1 , \hat {\mathcal{F}}_{2}]$ via methods (iiia) and (iiib) and summarize the results in the last two rows of \tab \ref{tab:realworldata}. We find that approach (iiia), where we build the entire matrix $\partial_2$ and then delete columns for each cycle representative, is in general faster than approach (iiib), where the boundary matrix $\partial_2[\mathcal{F}_1, \hat{\mathcal{F}}_2]$ is iteratively built for each representative. However, this latter approach can be more useful for large data sets, whose full boundary matrix $\partial_2$ might be too large to construct. For example, building the full boundary matrix for the \textbf{Genome} data set caused Julia to crash due to the large number  of $2$-simplices ($453{,}424{,}290$ triangles for the \textbf{Genome} data set and $3{,}357{,}641{,}440$ triangles for the \textbf{H3N2} data set). Whereas, by implenting (iiib) where we rebuild a part of the boundary matrix for each representative, we were able to optimize $115$ out of the $117$ cycle representatives for the \textbf{Genome} data set and $52$ of $57$ cycle representatives for the \textbf{H3N2} data set.

\subsection{Coefficients of optimal cycle representatives in data sets from \se \ref{sec: realworlddata} and \se \ref{sec: randompointclouds}}
\label{coefficient}
As discussed in \se \ref{secl0l1}, the problem of solving an $\ell_0$ optimization is desirable for its interpretability but doing so is NP-hard \cite{NPhardL0}. Often, $\ell_0$ optimization is approximated by an $\ell_1$ optimization problem, which is solvable in polynomial time. If the coefficients of a solution of the $\ell_1$ problem are in $\{-1,0,1\}$, then it is in fact an $\ell_0$ solution to the restricted optimization problem where we require solutions to have entries in $\{-1, 0, 1\}$ \cite{Escolar2016, Obayashi2018}. 

We find that $99.50\%$ of the original, unoptimized cycle representatives obtained from data sets described in \se \ref{sec: realworlddata} and $99.91\%$ of the unoptimized cycle representatives obtained from data sets described in \se \ref{sec: randompointclouds} have coefficients in $\{-1,0,1\}$. All unoptimized cycle representatives turned out to have integral entries.

We then systematically check each solution of the eight programs
\ref{itm:edge_NIU},
\ref{itm:edge_IU},
\ref{itm:edge_NIL},
\ref{itm:edge_IL},
\ref{itm:tri_NIU}, 
\ref{itm:tri_IU}, and
\ref{itm:tri_NIA}, 
\ref{itm:tri_IA}
 across all data sets and all optimal cycle representatives from data discussed in Sections \ref{sec: realworlddata} and \ref{sec: randompointclouds}\footnote{We discuss the coefficients of the Erd\H{o}s-R\'enyi complexes of \se \ref{sec:erdos} in \se \ref{sec:erdosbehavior}.} found by Algorithms \ref{alg:edge} and \ref{alg:rdvvolumeoptimization} and Program \eqref{eq:escolarargmin} to see if the coefficients are integral or in $\{-1,0,1\}$. We analyze the $18{,}163$ optimal cycle representatives and find the following consistent results.

 \emph{All} optimal solutions to \pr \eqref{eq:escolarargmin} (edge-loss minimization of filtered cycle bases) and \emph{all but one}  of the solutions returned by Algorithm \ref{alg:edge} (edge-loss minimization of \emph{persistent} cycle bases)  had coefficients in $\{-1, 0, 1\}$; see the table in the Supplementary Material for details. The exceptional representative $\optimalrep\EU\NI$ occurred in the \textbf{C.elegans} data set, with coefficients in $\{-0.5,0,0.5\}$.  It corresponds to one of only a few cases where two intervals with equal birth and death time occur within the same data set; see \se \ref{duplicate intervals}. An interesting consequence of these fractional coefficients is that here, unlike all other cycle representatives from data discussed in \se \ref{sec: realworlddata} and \se \ref{sec: randompointclouds}, the $\ell_0$-norm and $\ell_1$-norm differ.  This accounts for the sole point that lies below the $y=1$ line in the first column of row (B) in \fig \ref{fig:lengthcompare}.

 \begin{figure}[h!]
\begin{center}
\includegraphics[width=\textwidth]{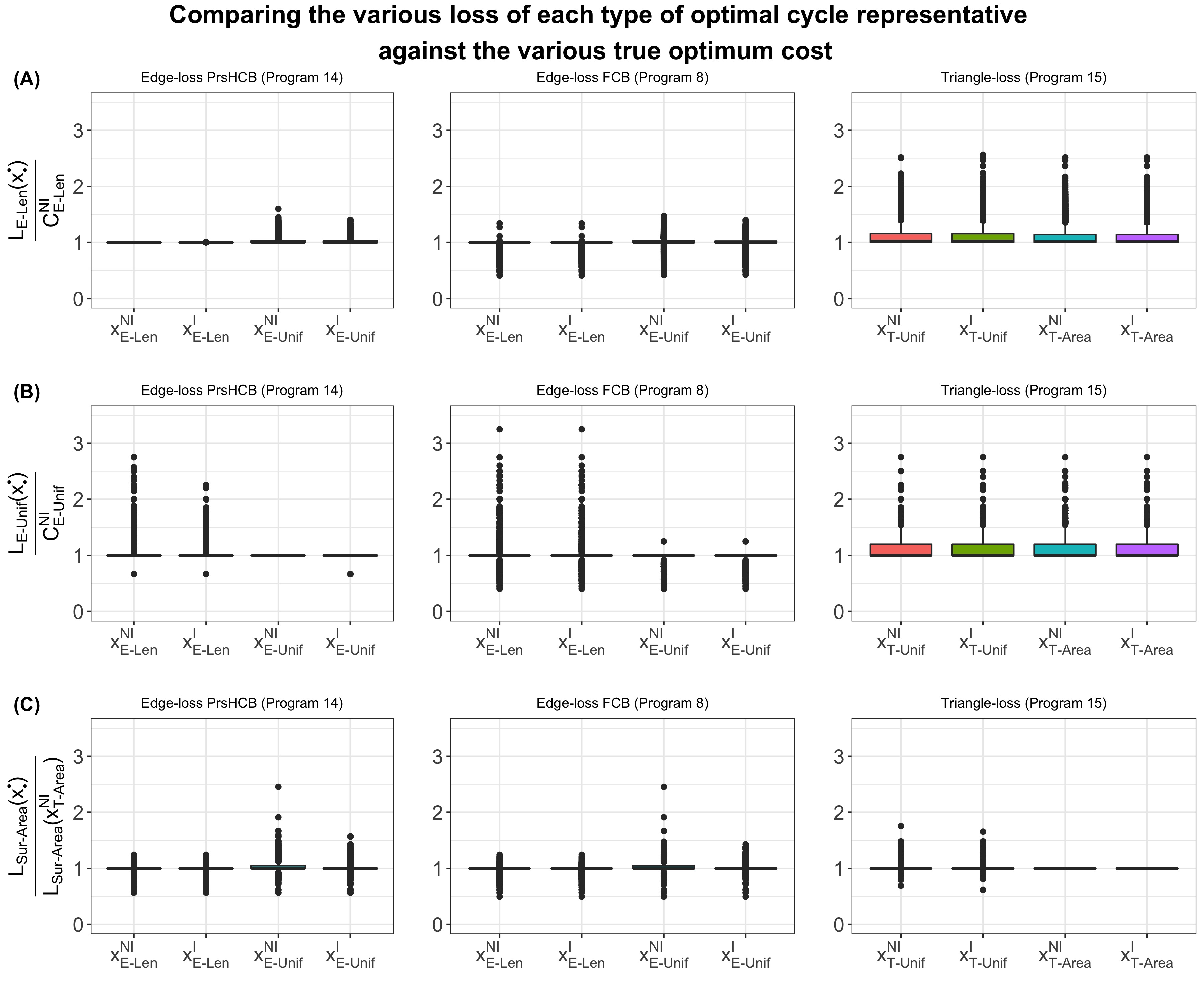}
\end{center}
\caption{Box plots of the ratios between (\textbf{A}) $L\EL(\optimalrep_\bullet^*)$ and $C\EL\NI$,  \textbf{(B)} $L\EU(\optimalrep_\bullet^*)$ and $C\EU\NI$, and  \textbf{(C)} $L_{Sur\text{-}Area}(\optimalrep_\bullet^*)$ and $L_{Sur\text{-}Area}(\optimalrep \TA\NI)$. 
Within each row, the denominator is fixed across all three columns, and corresponds to the $\setofpersistenthcyclebases$ cycles which are solutions to Programs \ref{itm:edge_NIU} row (\textbf{A}), \ref{itm:edge_NIL} row (\textbf{B}), or \ref{itm:tri_NIA} row (\textbf{C}).
The horizontal axis of each subplot is the type of optimal representative.
The cost of the optimal solutions to programs
\pr \eqref{eq:edgelossgeneral}, \pr \eqref{eq:escolarargmin}, and \pr \ref{itm:tri_NIA} was equal regardless of the presence of an integer constraint in nearly all examples (as discussed in \se \ref{coefficient}), resulting in two columns in each row having ratio 1. 
The data used in \textbf{(A)} and \textbf{(B)} aggregate over all cycle representatives from data described in Sections \ref{sec: realworlddata} and \ref{sec: randompointclouds}. The data used in \textbf{(C)} aggregate the $746$ cycle representatives from $40$ point clouds with ambient dimension of $2$ from data described in \ref{sec: randompointclouds}. We observe that some edge-loss and uniform-weighted-triangle-loss optimal cycles have a surveyor's area strictly smaller than the denominator in row (\textbf{C}); refer to Figure \ref{fig:areaExample} and \se  \ref{sec:comparing optimal generators against different loss functions} to see why this may happen. It is possible for $L\EU(\optimalrep_\bullet^*)$ to be strictly smaller than $C\NI\EU$ because the cycle $\optimalrep\NI\EU$ is calculated to be optimal relative to $\ell_1$ loss, not $L\EU$, which is a measure of $\ell_0$ loss. We observe this behavior in the first plot on the second row, discussed in detail in \se \ref{coefficient}.} 

\label{fig:lengthcompare}
\end{figure}

 On the one hand, this exceptional behavior could bear some connection to Algorithm \ref{alg:edge}.  
 Recall that Algorithm \ref{alg:edge} operates by removing a sequence of cycles from a cycle basis, replacing each cycle with a new, optimized cycle on each iteration (that is, we swap the $j+1^{th}$ element of $\obasis^{j}$ with an optimized cycle $\optimalrep$ in order to produce $Z^{j+1}$). Replacing optimized cycles in the basis is key, since without replacement it would be possible in theory to get a set of optimized cycles that no longer form a basis. We verified that if we modify Algorithm \ref{alg:edge} to skip the replacement step, we achieve $\{-1,0,1\}$ solutions for the exceptional \textbf{C.elegans} cycle representative (for the other repeated intervals we obtain the same optimal cycle with and without the replacement).  
 On the other hand, we find that even with replacement the GLPK solver obtains a solution with coefficients in $\{-1, 0, 1\}$.  Thus, every one of the cases considered produced $\{-1, 0, 1\}$ coefficients for at least one of the two solvers, and the appearance of fractional coefficients may be naturally tied to the specific  implementation of the solver used.

 When solving the triangle-loss problems by Algorithm \ref{alg:rdvvolumeoptimization}, we obtain one solution with coefficients of $2$ (for both the integral and non-integral problems) for one cycle representative from the logistic distribution data set. For that single representative, we rerun the optimizations with the additional constraint that it have coefficients with an absolute value less than or equal to one, which results in an infeasible solution. 

The surprising predominance of solutions in $\{-1, 0, 1\}$ suggests that in most cases, the modeler can reap both the computational advantage of $\ell_1$ solutions and the theoretical and interpretability advantages of $\ell_0$  solutions\footnote{Recall that, in the current discussion, $\ell_0$ optimality refers to the \emph{restricted} integer problem where coefficients are constrained to lie in $\{-1,0,1\}$.  The unrestricted problem (about which we have nothing to say) may have quite different properties.} by solving an $\ell_1$ optimization problem. Further, we find that the optimum cost is the same whether we require an integer solution or not for more than $99.97\%$ of solutions to \pr $\Edge\len$, $100\%$ of solutions to $\Edge\unif$, and $100\%$ of solutions to $\Tri\unif$. \emph{Thus, the modeler can drop the integral constraint to save computation time while still being able to achieve an integral solution in most cases.}

\subsection{Comparing optimal cycle representatives against different loss functions}\label{sec:comparing optimal generators against different loss functions}

We compare the optimal cycle representatives against different loss functions to study the extent to which the solutions produced by each technique vary. We consider two loss functions on an $H_1$ cycle representative $\optimalrep \in \Cycles_1(K)$:
$$
L\EL(\optimalrep) = \sum_{\sigma \in \supp(\optimalrep)} \mathrm{Length}(\sigma)
$$
where $\mathrm{Length}(\sigma)$ is the distance --- as designated by the metric $d$ used to define the VR complex --- between the two vertices of a $1$-simplex $\sigma$, and
$$
L\EU(\optimalrep) = ||\optimalrep||_0=|\supp(\optimalrep)|, 
$$ 
the number of 1-simplices (edges) in a representative.

We also consider two loss functions on 2-chains $\volvec \in \Chains_2(K)$, namely area-weighted loss:
$$
L\TA(\volvec) = \sum_{\tau \in \supp(\volvec)} \mathrm{Area}(\tau)$$
where $\mathrm{Area}(\tau)$ is the area of a $2$-simplex as computed by Heron's Formula, and uniform-weighted loss
$$
L\TU(\volvec) = ||\volvec||_0 = |\supp( \volvec)|.
$$

\begin{remark}
These weighted $\ell_0$ loss functions \emph{differ} from the objective functions used in the optimization problems presented in \se \ref{sec:programsandmethods}, which measure weighted $\ell_1$ norm.  However, weighted $\ell_0$ norm and weighted $\ell_1$ norm agree on solutions with $\{-1, 0, 1\}$ coefficients, and (as reported in \se \ref{coefficient}) nearly all cycle representatives for the \se \ref{sec: realworlddata} and \ref{sec: randompointclouds} data satisfy this condition, both pre- and post-optimization.
\end{remark}

In the special case where $\supp(\optimalrep)$ determines a simple closed polygonal curve $c$ with vertices $(p^1, q^1), \ldots, (p^n, q^n) \in \R^2$, we also use the Surveyor's Area Formula \cite{TheSurveyorsAreaFormula}  to quantify area of $\optimalrep$ as  
    \begin{align*}
    \textstyle
    L_{Sur\text{-}Area}(c)
    =
    \frac{1}{2}
    \left  |
        \sum_{i=1}^{n} p^i q^{i+1} - 
        p^{i+1} q^{i}
    \right |
    \end{align*}
where, by convention, $p^{i+1} = p^1$ and $q^{i+1} =  q^1$. We evaluate this function only when (i) the ambient point cloud of the VR complex is a subset of $\R^2$, (ii) $\supp(\optimalrep)$ forms a graph-theoretic cycle when regarded as a subset of edges in the combinatorial graph formed by 1-skeleton of $K$, and (iii) no pair of distinct closed line segments intersect one another.

In the case when we compute the loss function of a corresponding optimal solution, we use the notation for the cost $C_\bullet^* := L_\bullet(\optimalrep_\bullet^*)$ to an edge-loss problem that finds optimal solution $\optimalrep_\bullet^*$, and $C_\bullet^* := L_\bullet(\volvec_\bullet^*)$ to a triangle-loss problem that finds optimal solution $\volvec_\bullet^*$. For instance, $C\EU\NI=L\EU(\optimalrep\EU\NI)$. We will also compute the loss functions of optimal solutions from differing optimizations. For instance, $L\EL(\optimalrep\EU\NI)$, and in that case, we do not use the $C_\bullet^*$ notation.

\begin{figure}[hbt!]
\begin{center}
\includegraphics[width=\textwidth]{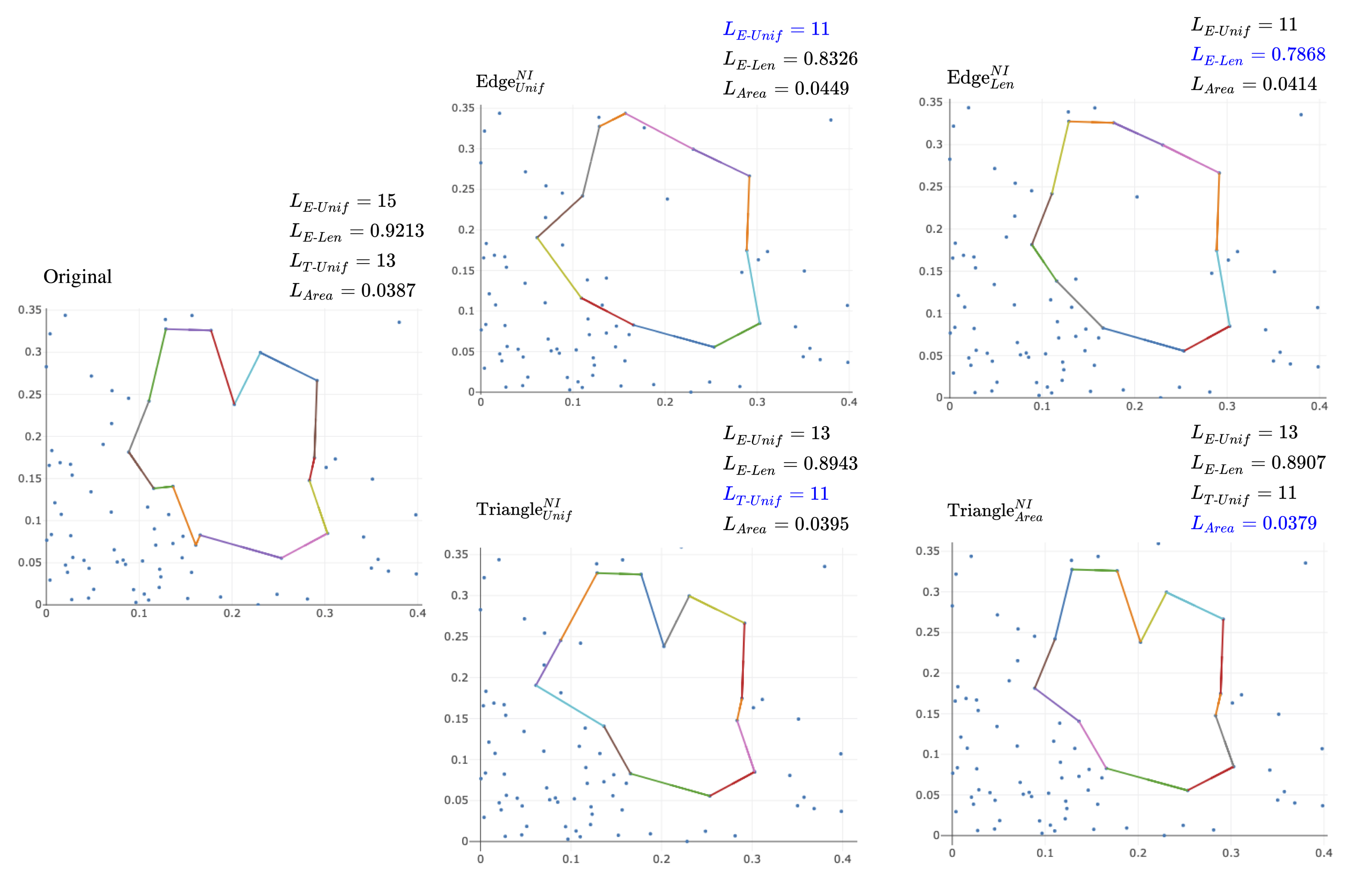}
\end{center}
\caption{Examples of different optimal cycles and cost against different loss functions using a point cloud of $100$ points with ambient dimension $2$ randomly drawn from a normal distribution. The upper left corner of each subfigure labels the optimization algorithm used to optimize the original cycle representative. The upper right corner of each subfigure records the different measures of the size of the optimal representative. Blue text represents the measure an algorithm sets out to optimize. 
}\label{fig:Examplesofeachoptimalcycles} 
\end{figure}

\fig \ref{fig:Examplesofeachoptimalcycles} shows an example of various optimal cycle representatives obtained from Programs
\ref{itm:edge_NIU},
\ref{itm:edge_NIL},
\ref{itm:tri_NIU}, and
\ref{itm:tri_NIA}
on an example point cloud drawn from the normal distribution in $\R^2$. In this example, solutions obtained from Algorithm \ref{alg:edge} and \pr \eqref{eq:escolarargmin} are the same. Each subfigure is labeled by program in the upper left corner. The values of different loss functions evaluated on each optimal representative appear in the upper right corner.  
We do not compute $L\TU$ or $L\TA$ of the optimal edge-loss minimal cycle representatives, as no bounding $2$-chain for this $1$-cycle is specified in the optimization.\footnote{We formulated an Obayashi-style linear program similar to \pr \eqref{eq:trianglelossgeneral} to compute the volume of edge-loss optimal cycles but in many cases it had no feasible solution.} We observe that various notions of optimality lead to differing cycle representatives, yet each solution to an optimization problem minimizes the loss function it is intended to optimize. This will not always be the case, as we will see momentarily.

\fig\ref{fig:lengthcompare} reports ratios on the losses $L\EU$, $L\EL$, and $L_{Sur\text{-}Area}$\footnote{Recall, we only compute $L_{Area}$ on the 2-dimensional distribution data.} for the eight $\setofpersistenthcyclebases$ optimization problems detailed in \se \ref{sec:programsandmethods} as well as the four edge-loss $\setoffilteredcyclebases$ problems from \pr \eqref{eq:escolarargmin}, evaluated on the data from Sections \ref{sec: realworlddata} and \ref{sec: randompointclouds}. 
These ratios suggest that the uniform-weighted and length-weighted edge-loss cycles do minimize what they set out to minimize, namely, the number of edges and the total length, respectively. We also observe that intuitively the less-constrained solutions to the $\setoffilteredcyclebases$ \pr \eqref{eq:escolarargmin} can have a lower cost than the more-constrained solutions to the $\setofpersistenthcyclebases$ \pr \eqref{eq:edgelossgeneral}. 
 
We also see that the edge-loss-minimal cycles have similar loss in terms of length and number of edges ($L\EL$ and $L\EU$) whereas the triangle-loss-minimal cycles can have larger losses ($L\EL(\optimalrep\TU)$ and $L\EU(\optimalrep\TU)$). We find that $63.28\%$ of the $L\EU$ minimal cycle representatives are also $L\EL$ minimal while $99.66\%$ of the $L\EL$ minimal cycle representatives are also $L\EU$ minimal across all cycle representatives from all data sets for $\setofpersistenthcyclebases$ cycles. Similarly, we find that $61.31\%$ of the $L\EU$ minimal cycle representatives are also $L\EL$ minimal while $99.32\%$ of the $L\EL$ minimal cycle representatives are also $L\EU$ minimal across all cycle representatives from all data sets for $\setoffilteredcyclebases$ cycles. This suggests that modelers can often use the length-weighted minimal cycle to substitute the uniform-weighted minimal cycle. However, the triangle-loss cycles can potentially provide very different results. 

Counterintuitively, the $L\TA$ optimal cycle representative might not be the representative that encloses the smallest surveyor's area. As shown in \fig\ref{fig:lengthcompare}, we observe that $15.55\%$ of $\optimalrep\EU\NI$, $13.14\%$ of $\optimalrep\EU\I$, $23.59\%$ of $\optimalrep\EL\NI$, and $23.59\%$ of  $\optimalrep\EL\I$ for the cycles from $\setofpersistenthcyclebases$ using Program \eqref{eq:edgelossgeneral} have an area smaller than that of the triangle-loss area-weighted optimal cycle $\optimalrep\NI\TA$. 
Similarly,  $15.55\%$ of $\optimalrep\EU\NI$, $12.87\%$ of  $\optimalrep\EU\I$, $24.53\%$ of $\optimalrep\EL\NI$, and $24.53\%$ of $\optimalrep\EL\I$ for the cycles from $\setoffilteredcyclebases$ using Program \eqref{eq:escolarargmin} have an area smaller than that of the triangle-loss area-weighted optimal cycle $\optimalrep\NI\TA$. Lastly, $3.22\%$ of $\optimalrep^{I}\TU$, $2.81\%$ of $\optimalrep^{NI}\TU$, and $2.95\%$ of $\optimalrep^{I}\TA$ for the cycles found using Program \eqref{eq:trianglelossgeneral} have an area smaller than that of the triangle-loss area-weighted optimal cycle $\optimalrep\NI\TA$.  

\begin{figure}[h!]
\begin{center}
\includegraphics[width=\textwidth]{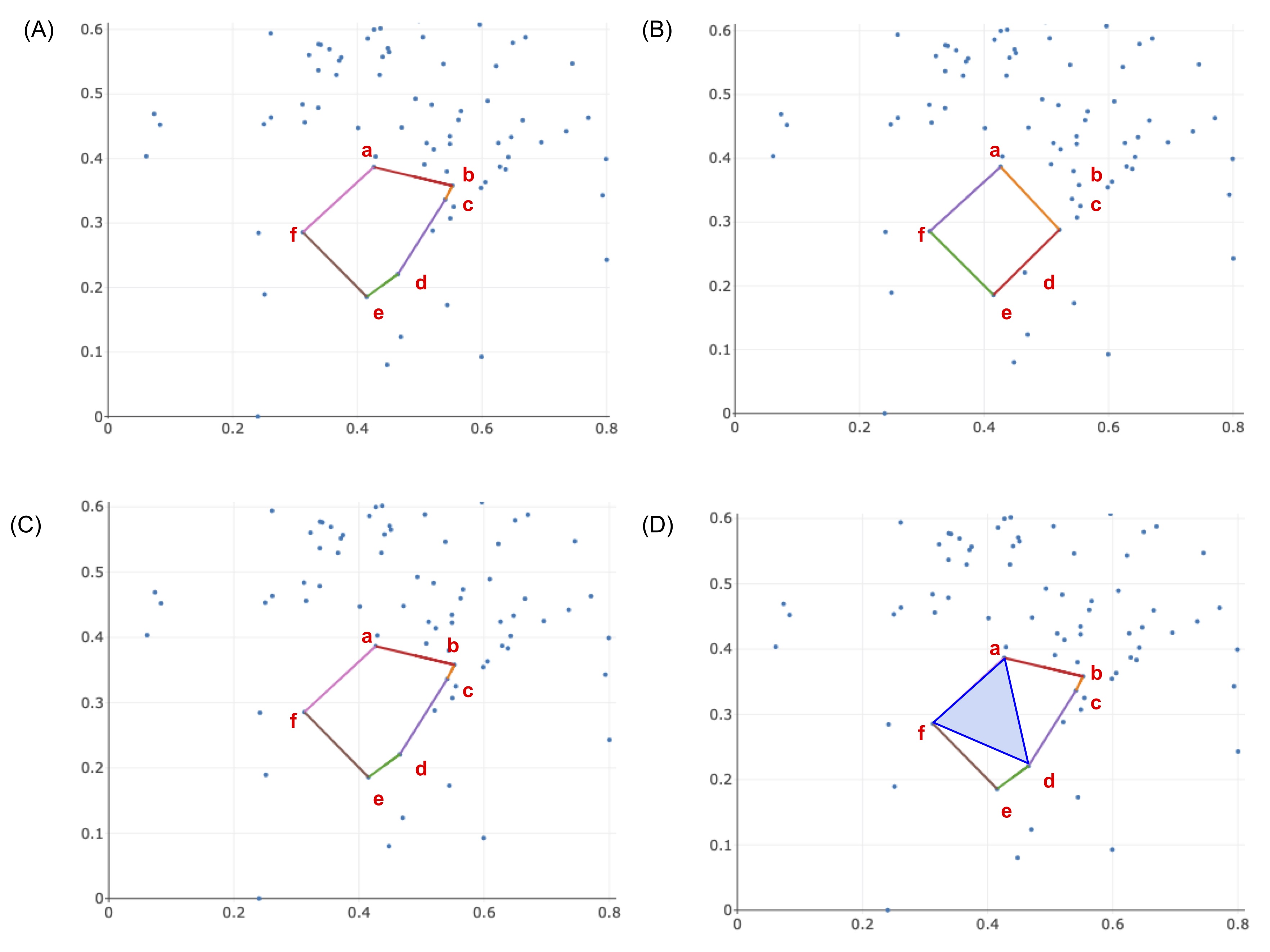}
\end{center}
\caption{An example illustrating when the area enclosed by the triangle-loss area-weighted optimal cycle, solution to \pr \ref{itm:tri_NIA}, can be larger than the area enclosed by the edge-loss length-weighted minimal cycle, solution to \pr \ref{itm:edge_NIL}. \textbf{(A)} is the original cycle of a representative point cloud in $\mathbb{R}^2$ drawn from the normal distribution, \textbf{(B)} is the length-weighted edge-loss optimal cycle, \textbf{(C)} is the area-weighted triangle-loss optimal cycle, in this example, it is the same cycle as the original cycle, \textbf{(D)} is the area-weighted minimal cycle where the blue shaded area marks the triangle born at the death time of the cycle.  
Constraint Equation \eqref{obacond1} specifies that the area-weighted optimal cycle must contain the 2-simplex born at the death time of the cycle. Therefore, this cycle must contain $(a,d,f)$ because it was born at the death time. The length-weighted minimal cycle does not have this constraint, and as such, can result in a smaller area. 
}\label{fig:areaExample}
\end{figure}

In Figure \ref{fig:areaExample}, we provide an example illustrating why the triangle-loss area-weighted optimal cycle, solving Programs \ref{itm:tri_NIA}, or
\ref{itm:tri_IA}, might not be the cycle that encloses the smallest surveyor's area. Another reason why the area-weighted triangle-loss cycles could have a larger enclosed area is that in the optimization problems, the loss function is the sum of the triangles the cycle bounds, not the real enclosed area. Therefore, the area-weighted triangle-loss cycle will have the optimal area-weighted optimal cost, but not necessarily the smallest enclosed area. 

\subsection{Comparative performance and precision of LP solvers}\label{Computational cost of the various optimization techniques}

Our experiments demonstrate that the choice of linear solver may impact speed, frequency of obtaining integer solutions, and frequency of obtaining $\ell_0$ optimal solutions. While these particular results are subject to change due to regular updates to each platform, they illustrate the degree to which these factors can vary.

As discussed in \se \ref{sec:timecomparisons}, the GLPK solver performs much slower than the Gurobi solver in an initial set of experiments. The GLPK solver also finds non-integral solutions when solving a linear programming problem in Programs
\ref{itm:edge_NIU}, and
\ref{itm:edge_NIL}
 more often than the Gurobi solver. On the same set of experiments as in \se \ref{sec:timecomparisons}, when finding the $\setoffilteredcyclebases$ using \pr \eqref{eq:escolarargmin}, $9.74\%$ of the edge-loss length-weighted minimal cycle representatives have non-integral entries, and $8.32\%$ of the edge-loss uniform-weighted minimal cycle representatives have non-integral entries when using the GLPK solver, whereas when using the Gurobi solver, $0.12\%$ of the length-weighted minimal cycle representatives have non-integral entries, and $0.04\%$ of the uniform-weighted minimal cycle representatives have non-integral entries. For the length-weighted minimal cycle representatives, the non-integral solutions differ from an $\ell_0$ optimal solution by a margin of machine error with both solvers. However, for the uniform-weighted minimal cycle representatives, the GLPK solver has $1.83\%$ of its non-integral solutions differing from an $\ell_0$ optimal solution by a margin not of machine epsilon, and the Gurobi solver has $0.02\%$ of its non-integral solutions differing from an $\ell_0$ optimal solution by a margin greater than machine epsilon. For the GLPK solver, when solving \pr
 \ref{itm:edge_NIU}, instead of finding an integral solution, it occasionally finds a solution with fractional entries that sum to $1$. For example, instead of assigning an edge a coefficient of $1$, it sometimes assigns two edges each with a coefficient of $0.5$. In that way, the solution is still $\ell_1$ optimal, but no longer $\ell_0$ optimal. \emph{Thus, the choice of linear solver may affect the optimization results.}

\subsection{Statistical properties of optimal cycle representatives with regard to various other quantities of interest}

\noindent \emph{Support of a representative forming a single loop in the underlying graph}

The support of the original cycle, $\supp(\originalrep) \subseteq \Simplices_1(K)$, need not be a cycle in the graph-theoretic sense.  Concretely, this means that the nullity, $p$, of column submatrix $\partial_1[:,\originalrep]$ may be strictly greater than 1 (in \fig \ref{fig:areaExample}, for example, $p=2$).  We refer to $p$ informally as the ``number of loops'' in $\originalrep$. 


We are interested in exploring how often the support of an original cycle representative forms a single loop in the underlying graph. We analyze each of the 360 synthetic data sets of various dimensions and distributions discussed in \se \ref{sec: randompointclouds} as well as the 100 Erd\H{o}s-R\'enyi random complexes discussed in Section \ref{sec:erdos} and display the results in \fig \ref{fig:loopsbreakdown}. We find that the majority of the original cycle representatives have one loop. 
After optimizing these cycle representatives with the edge-loss methods, we verify that all $\setoffilteredcyclebases$ and $\setofpersistenthcyclebases$ optimal cycles only have one loop, whereas $0.13\%$ of the triangle-loss cycles have $2$ loops. However, we observe that $91.93\%$ of the optimal cycle representatives for Erd\H{o}s-R\'enyi complexes have $1$ loop, $5.81\%$ have $2$ loops, and $2.14\%$ have more than $2$ loops, with $15$ as the maximum number of loops. 

 \begin{figure}[]
\begin{center}
\includegraphics[width=\textwidth]{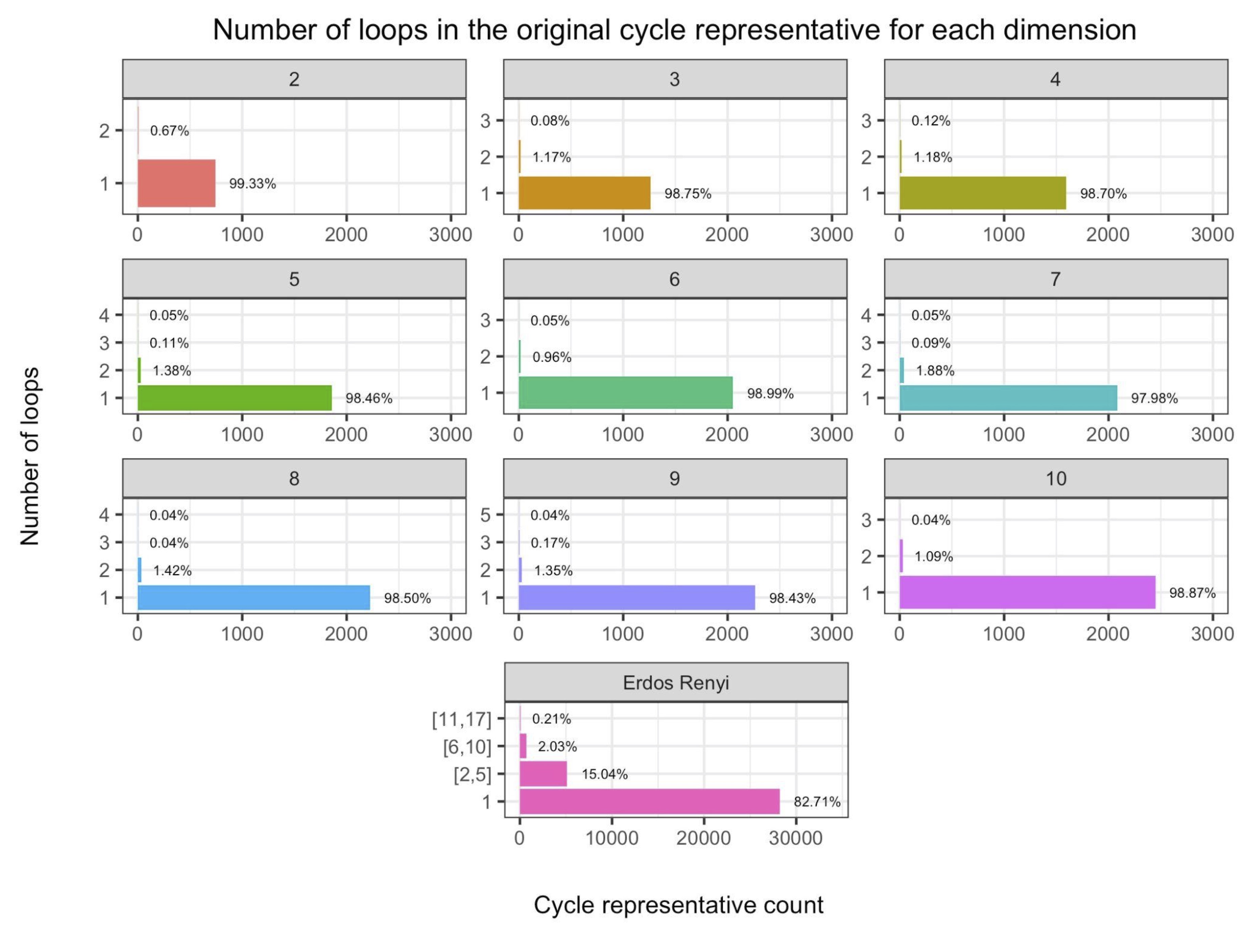}
\end{center}
\caption{  
(Rows 1-3) Number of loops in the original cycle representative aggregated by dimension (labeled by subfigure title) in the  $360$ randomly generated distribution data sets discussed in Section \ref{sec: randompointclouds} and (Row 4) same for the  Erd\H{o}s-R\'enyi random complexes discussed in Section \ref{sec:erdos}, where we bin cycle representatives that have 2-5 loops, 6-10, loops, or more than 10 loops. The horizontal axis is the number of cycle representatives and the vertical axis is the number of loops in the original representative. We observe that for the distribution data, an original cycle representative can have up to 5 loops in higher dimensions, and in general, it is uncommon to find an original representative with multiple loops. However, we observe that $17.47\%$ of the cycle representatives for Erd\H{o}s-R\'enyi complexes have more than $1$ loop, with a maximum number of $17$ loops in a cycle representative.}
\label{fig:loopsbreakdown}
\end{figure}

As shown in \fig \ref{fig:reductioncompare} the reduction in size of the original cycle, formalized as $\frac{C^*_\bullet}{L_\bullet(\originalrep)}$,  correlates closely with the reduction in the number of loops by the optimization.

\begin{figure}[]
    \centering
    \includegraphics[width=.7\textwidth]{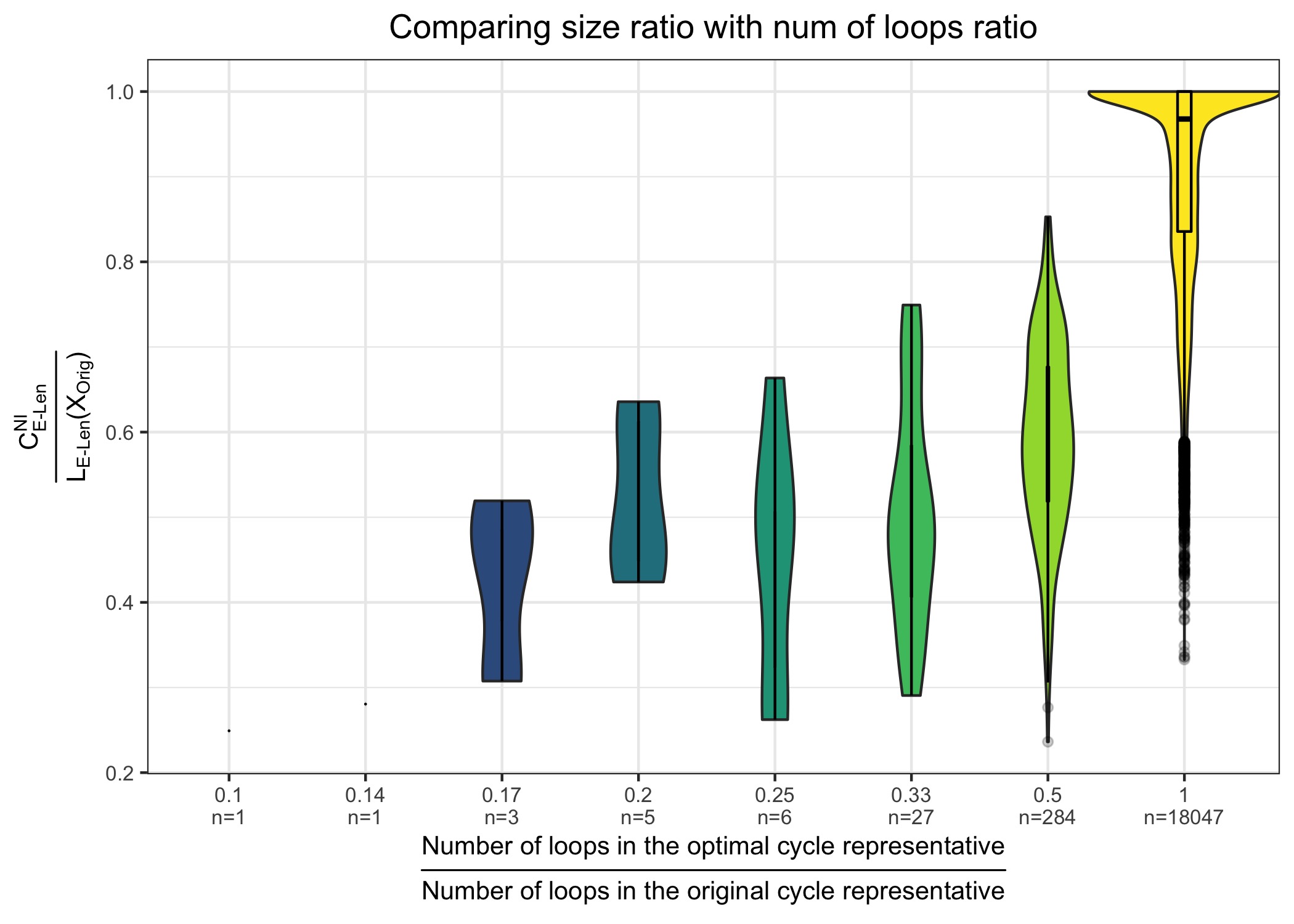}
    \caption{Violin plot of the effectiveness of the optimization as a function of the number of loops in the original cycle representative.  Results are aggregated over the data sets from \se \ref{sec: realworlddata} and \ref{sec: randompointclouds}. The $x$-axis shows the size reduction in terms of number of loops, and the $y$-axis shows the size reduction in terms of the length of the cycle. We see that in general, the reduction in size of the original cycle mostly comes from the reduction in the number of loops by the optimization. } 
    \label{fig:reductioncompare}
\end{figure}
 
\vspace{.1in}
\noindent \emph{Overall effectiveness of optimization ($L_\bullet(\optimalrep_\bullet^*)$ vs. $L_\bullet(\originalrep)$)} 

We compare the optimal representatives against the original cycle representatives\footnote{The remainder of this subsection excludes the Erd\H{o}s-R\'enyi cycles.} with respect to edge-loss functions $L\EU$ and $L\EL$. As shown in \fig \ref{fig:effectivenessall}, we find that the optimizations are in general effective in reducing the size of the cycle representative, especially for representatives with larger size. On each of the subfigures, the horizontal axis is the size of the original cycle representative and the vertical axis is the ratio between the loss of each optimal representative and the loss of the original representative:
$$\frac{C^*_\bullet}{L_\bullet(\originalrep)}.$$

\begin{figure}[]
\begin{center}
\includegraphics[width=0.95\textwidth]{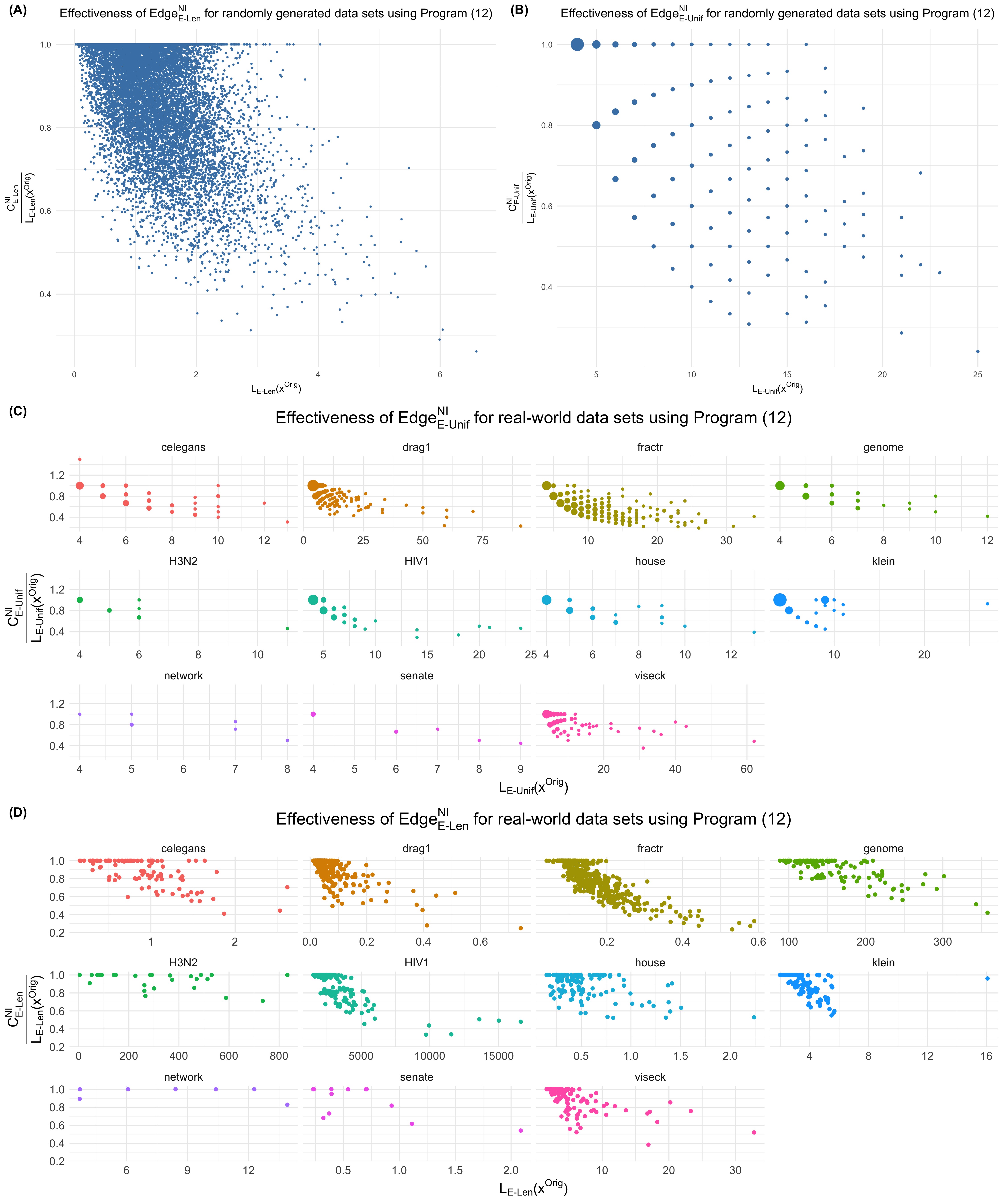} 
\end{center}
\caption{The effectiveness of length-weighted and uniform-weighted optimization for the data sets in Sections \ref{sec: realworlddata} and \ref{sec: randompointclouds} in reducing the size of the original cycle representative found by the persistence algorithm. In each subfigure, the horizontal axis is the size of the original representative and the vertical axis is the ratio between the size of the optimal representative and the size of the original representative. The uniform-weighted graphs appear more sparse because reductions in the cost $L\TU(\originalrep)$ can only come in multiples of the reciprocal of the original length. The node size in the uniform-weighted graphs corresponds to the number of overlapping points. 
}\label{fig:effectivenessall} 
\end{figure}

The average ratio $\frac{C\NI\EU}{L\EU(\originalrep)}$ is $83.17\%$, aggregated over cycle representatives obtained from data described in \se \ref{sec: realworlddata} and $90.35\%$  aggregated over cycle representatives obtained from data described in \se \ref{sec: randompointclouds} for cycles obtained from \pr \eqref{eq:edgelossgeneral}. The average ratio $\frac{C\NI\EL}{L\EL(\originalrep)}$ is $87.02\%$ over cycle representatives obtained from data described in \se \ref{sec: realworlddata} and $90.41\%$ over cycle representatives obtained from data describedra \se \ref{sec: randompointclouds} for cycles obtained from Program \eqref{eq:edgelossgeneral}. The average ratio $\frac{C\NI\TU}{L\TU(\originalrep)}$ is $88.34\%$ over cycle representatives obtained from data described in \se \ref{sec: realworlddata} and $95.54\%$ over cycle representatives obtained from data described in \se \ref{sec: randompointclouds} for cycles obtained from Program \eqref{eq:trianglelossgeneral}.

\noindent \emph{Comparing solutions to  integral programs and non-integral programs ( $\optimalrep^{NI}_\bullet$ vs. $\optimalrep^{I}_{\bullet}$)}

Among all cycle representatives found by solving \pr \eqref{eq:edgelossgeneral}, $66.38\%$ of them have $\optimalrep^{NI}\EU = \optimalrep^{I}\EU$, and  $99.51\%$ of them have $\optimalrep^{NI}\EL = \optimalrep^{I}\EL$. We find $\optimalrep^{NI}\TU = \optimalrep^{I}\TU$ for $74.27\%$ of the cycle representatives and $\optimalrep^{NI}\TA = \optimalrep^{I}\TA$ for $100\%$ of the cycle representatives when using the triangle-loss \pr \eqref{eq:trianglelossgeneral}. Thus, the presence or absence of integer constraints rarely impacts the result of an area- or length-weighted program, but often impacts solutions of a uniform-weighted program. We saw in \se \ref{coefficient} that essentially all solutions had coefficients in $\{-1, 0, 1\}$ regardless of integer or non-integer constraints. As such, we conjecture that the higher rate of different solutions in the uniform-weighted problems could result from a larger number of distinct optimal solutions and be a feature of particular choice of solution selected by the linear solvers, rather than the non-existence of a particular integer solution.

\noindent \emph{Cycle representative size for different distributions and dimensions}

\fig \ref{fig:gen_num_breakdown} provides a summary of the size and number of cycle representatives found for each distribution data set described in Section \ref{sec: randompointclouds}. We observe that there tend to be more and larger (with respect to $\ell_0$ norm) representatives in higher dimensions.
\\
\\
\noindent \emph{Duplicate intervals in the barcode}
\label{duplicate intervals}

Of all data sets analyzed, only \textbf{Klein} and \textbf{C.elegans} have barcodes in which two or more intervals had equal birth and death times (that is, bars with multiplicity $\ge 2$). Among the $107$ total intervals of the \textbf{C.elegans} data set, there are $75$ unique intervals, $10$ intervals with multiplicity two, and one interval each with multiplicity three, four, and five. The duplicate bars in the \textbf{C.elegans} data set are noteworthy for having produced the sole example of an optimized cycle representative $\optimalrep\NI\EU$ with coefficients outside $\{-1, 0, 1\}$ (in particular, it had coefficients in $\{-0.5, 0.5\}$). 

 \begin{figure}[]
\begin{center}
\includegraphics[width=0.9\textwidth]{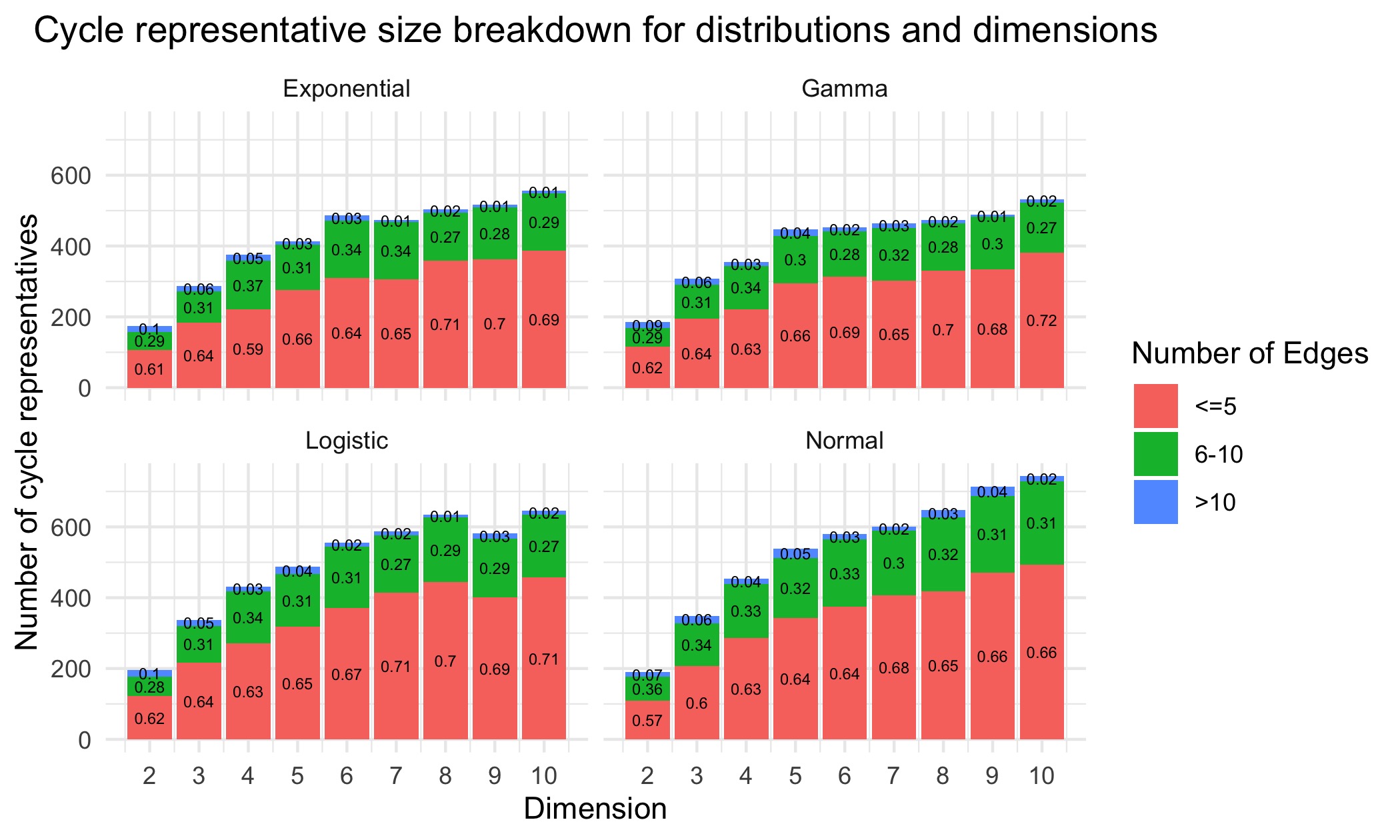} 
\end{center}
\caption{The number of original cycle representatives and the number of edges within each original representative for data described in \se \ref{sec: randompointclouds}. These plots aggregate all cycle representatives for each dimension of a particular distribution. The horizontal axis for each subplot is the dimension of the data set, and the vertical axis is the number of cycle representatives found in each dimension. In general, we see there are more cycle representatives in higher dimensional data sets. Each bar is partitioned by the number of edges of the representative. We observe that as dimension increases, there tend to be more cycle representatives with more edges. 
}\label{fig:gen_num_breakdown}
\end{figure}

Among the $257$ total intervals of the \textbf{Klein} data set, there are $179$ unique intervals, $1$ interval that is repeated twice, and $2$ intervals that are repeated $38$ times. For the \textbf{Klein} data set, if we replace the distance matrix provided by \cite{roadmap2017} with the Euclidean distance matrix calculated using Julia (the maximum difference between the two matrices is on the scale of $10^{-5}$), we obtain only one interval that is repeated twice. This indicates that duplicate intervals are rare in practice, at least in dimension 1. 

\vspace{.1in}
 \noindent \emph{Edge-loss cycle representatives $\setoffilteredcyclebases$ vs. $\setofpersistenthcyclebases$}

We find that for $84.52\%$ of \ref{itm:edge_NIU}, $90.84\%$ of \ref{itm:edge_IU}, $93.49\%$ of \ref{itm:edge_NIL}, and $93.49\%$ of \ref{itm:edge_IL}, the $\setoffilteredcyclebases$ edge-loss cycle representatives found by \pr \eqref{eq:escolarargmin} and the $\setofpersistenthcyclebases$ edge-loss cycles from \pr \eqref{eq:edgelossgeneral} are the same, i.e. the $\ell_1$ norm of their difference is $0$. As mentioned in Remark \ref{rmk:filteredversuspersistent}, the $\setoffilteredcyclebases$ cycles may not have the same death time as $\originalrep$. For the real-world data sets, $6.72\%$ of the $\eqref{itm:edge_NIL}$ and $\eqref{itm:edge_IL}$,  $7.65\%$ of the $\eqref{itm:edge_NIU}$ and $4.48\%$ of $\eqref{itm:edge_IU}$ have lifetimes different than $\originalrep$. For the randomly generated distribution data sets, $7.11\%$ of the $\eqref{itm:edge_NIL}$ and $\eqref{itm:edge_IL}$,  $8.06\%$ of the $\eqref{itm:edge_NIU}$ and $4.25\%$ of $\eqref{itm:edge_IU}$ have lifetimes different than $\originalrep$. All cycle representatives with lifetimes different than $\originalrep$ have death time beyond that of $\originalrep$.

\subsection{Optimal cycle representatives for Erd\H{o}s-R\'enyi random clique complexes}
\label{sec:erdosbehavior}

We observe qualitatively different behavior in cycle representatives from the Erd\H{o}s-R\'enyi random clique complexes.
Among the $34{,}214$ cycle representatives from the 100 dissimilarity matrices found by solving Programs \eqref{eq:edgelossgeneral} and \eqref{eq:trianglelossgeneral}, we find that $91.04\%$ of the original cycle representatives have entries in $\{-1,0,1\}$ and $99.75\%$ of the original cycle representatives have integral entries. We have $3.89\%$ of the length-weighted edge-loss representatives, $4.49\%$ of the uniform-weighted edge-loss representatives, and $3.52\%$ of the uniform-weighted triangle-loss representatives with entries not in $\{-1,0,1\}$. We find $2.66\%$ of the length-weighted edge-loss representatives, $3.57\%$ of the uniform-weighted edge-loss representatives, and $1.58\%$ of the uniform-weighted triangle-loss representatives with non-integral entries when not requiring integral solutions.

We find $\frac{L_{E-Unif}(\optimalrep^{NI}_{E-Unif})}{L_{E-Unif}(\optimalrep^{I}_{E-Unif})}> 1$ for $1.07\%$ of the cycle representatives and $\frac{L_{E-Len}(\optimalrep^{NI}_{E-Len})}{L_{E-Len}(\optimalrep^{I}_{E-Len})}> 1$ for $1.09\%$ of the representatives. All such representatives have entries outside of $\{-1,0,1\}$ and involve some fractional entries. An average of $96.75\%$ of the nonzero entries in the reduced boundary matrices are in $\{-1,1\}$, $2.15\%$ in $\{-2,2\}$, and $0.27\%$ with an absolute value greater than or equal to $3$. 

Because of the non-integrality of some original cycle representatives found by the persistence algorithm, we fail to find an integral solution for $0.27\%$ of the edge-loss representatives and $0.11\%$ of the triangle-loss representatives. 

A partial explanation for this  behavior can be found in the work of Costa et al. \cite{erdos-renyi-RP2}. Here, the authors show that a connected two-dimensional simplicial complex for which every subcomplex has fewer than three times as many edges as vertices must have the homotopy type of a wedge of circles, 2-spheres, and real projective planes. With high probability, certain ranges of thresholds for the i.i.d. dissimilarity matrices on which the Erd\H{o}s-R\'enyi random complexes are built produces random complexes with approximately such density patterns at each vertex. Thus, some of the persistent cycles are highly likely to correspond to projective planes. Because of their non-orientability, the corresponding minimal generators could be expected to have entries outside of the range $\{-1, 0, 1\}.$

 \section{Conclusion}\label{discussion}

In this work, we provide a theoretical, computational, and empirical user's guide to optimizing cycle representatives against four criteria of optimality: total length, number of edges, internal volume, and area-weighted internal volume. Utilizing this framework, we undertook a study on statistical properties of minimal cycle representatives for $\Homologies_1$ homology found via linear programming. In doing so, we made the following four main contributions.
\begin{enumerate}
    \item We developed a publicly available code library \cite{li_thompson} to compute persistent homology with rational coefficients, building on the software package Eirene \cite{eirene} and implemented and extended algorithms from \cite{Escolar2016, Obayashi2018}  for computing minimal cycle representatives. The library employs standard linear solvers (GLPK and Gurobi) and implements various acceleration techniques described in \se \ref{acceleratation technique} to make the computations more efficient. 
    \item We formulated specific recommendations concerning procedural factors that lie beyond the scope of the optimization problems per se (for example, the process used to generate inputs to a solver) but which bear directly on the overall cost of computation, and of which practitioners should be aware. 
    \item We used this library to compute optimal cycle representatives for a variety of real-world data sets and randomly generated point clouds.  Somewhat surprisingly, these experiments demonstrate that computationally advantageous properties are typical for persistent cycle representatives in data. Indeed, we find that we are able to compute uniform/length-weighted optimal cycles for all data sets we considered, and that we are able to compute triangle-loss optimal cycles for all but six cycle representatives, which fail due to the large number of triangles (more than $20$ million) used in the optimization problem. Computation time information is summarized in \tab \ref{tab:realworldata} and \tab \ref{tab:distributiondata}. 
    
    Consequently, heuristic techniques may provide efficient means to extract solutions to cycle representative optimization problems across a broad range of contexts. For example, we find that edge-loss optimal cycles are faster to compute than triangle-loss optimal cycles for cycle representatives with a longer persistence interval, whereas for cycles with shorter persistence intervals, the triangle-loss cycle can be less computationally expensive to compute.
    
    \item We provided statistics on various minimal cycle representatives found in these data, such as their effectiveness in reducing the size of the original cycle representative found by the persistence algorithm and their effectiveness evaluated against different loss functions. In doing so, we identified consistent trends across samples that address the questions raised in \se \ref{intro}.
    \begin{enumerate}
        \item Optimal cycle representatives are often significant improvements in terms of a given loss function over the initial cycle representatives provided by persistent homology computations (typically, by a factor of 0.3 to 1.0). Interestingly, we find that area-weighted triangle-loss optimal cycle representatives can enclose a greater area than length- or uniform-weighted optimal cycle representatives.  
        \item We find that length-weighted edge-loss optimal cycles are also optimal with respect to a uniform-weighted edge-loss function upwards of $99\%$ of the time in the data we studied. This suggests that one can often find a solution that is both length-weighted minimal and uniform-weighted minimal by solving only the length-weighted minimal optimization problem. However, the triangle-loss optimal cycles can have a relatively higher length-weighted edge-loss or uniform-weighted edge-loss than the length/uniform-weighted minimal cycles. Thus, computing triangle-loss optimal cycles might provide distinct information and insights. 
        \item Strikingly, all but one $\ell_1$ optimal representatives were also $\ell_0$ optimal (that is, $\ell_0$ optimal among cycles taking $\{0,1,-1\}$ coefficients;  $\ell_0$ optimality among cycles taking $\Z$ coefficients was not tested) in the real-world and synthetic point cloud data. Thus, it appears that solutions to the NP-hard problem of finding $\ell_0$ optimal cycle representatives can often be solved using linear programming in real data. In the Erd\H{o}s-R\'enyi random complexes, qualitatively different behavior was found; this may relate to the fact that spaces in this random family contain non-orientable subcomplexes with high probability.
    \end{enumerate}

 Several questions lie beyond the scope of this text and merit future investigation.  For example, while the methods discussed in \se \ref{sec:programsandmethods} apply equally to homology in any dimension, we have focused our empirical investigation exclusively in dimension one; it would be useful and interesting to compare these results with homology in higher dimensions.  It would likewise be interesting to compare with different weighting strategies on simplices, and loss functions other than $\ell_0$ and $\ell_1$, e.g. $\ell_2$.  Future work may also consider whether the modified approach to edge-loss minimization \pr \eqref{eq:edgelossgeneral} could be incorporated into persistence solvers themselves, as pioneered in \cite{Escolar2016}.  Unlike the programs formulated in this earlier work, \pr \eqref{eq:edgelossgeneral} requires information about the death times of cycles in addition to their births; typically this information is not available until after the persistence computation has already finished, so new innovations would probably be needed to make progress in this direction.

\end{enumerate}



\section{Conflict of Interest Statement} 
The authors declare that the research was conducted in the absence of any commercial or financial relationships that could be construed as a potential conflict of interest.

\section{Author Contributions}

G.H-P. wrote the Eirene code. L.L. and C.T. wrote the rest of the code and performed all experiments. L.L. created all figures and tables. G.H-P. and L.Z. designed, directed, and supervised the project. G.H-P. developed the theory in the Supplementary Material. All authors contributed to the analysis of the results and to the writing of the manuscript.

\section{Funding}
This material is based upon work supported by the National Science Foundation under grants no. DMS-1854683, DMS-1854703, and DMS-1854748. 

\section{Acknowledgments} 
The authors are grateful to David Turner for helpful advice on selection of linear solvers. They also thank and acknowledge the initial work done by Robert Angarone and Sophia Wiedemann on this study.

 \section{Supplemental Data}

Details regarding implementation and proofs of correctness for the algorithms discussed in \se \ref{sec:programsandmethods} as well as an additional figure and table may be found in the Supplementary Material.
 
\section{Data Availability Statement}
The data sets and code for this study can be found in the repository \cite{li_thompson}.

{\small
	\bibliographystyle{abbrv}
\bibliography{main}
}

\newpage

\appendix

\onecolumn

\maketitle

\section{Review: Standard algorithm to compute persistent homology cycle bases, by matrix decomposition}
\label{sec:computingcyclereps}

The first step in Algorithms 1 and 2 is to compute a persistent homology cycle basis.  The standard method to compute such a basis invokes a so-called $R = DV$ decomposition  of the boundary matrices $\partial_n$ \cite{cohen2006vines}.  Here we  provide a brief review of this process; further details may be found in \cite{cohen2006vines, SMVDualities11}.
  
  To begin, we must place a total order on each set $\Simplices_n(K)$, in ascending order of birth.  This naturally allows us to regard $\partial_n$ as an element of $G^{|\Simplices_{n-1}(K)| \times |\Simplices_{n}(K)|}$.  The \emph{low}  function on a matrix $A \in G^{k \times l}$ is defined by
    \begin{align*}
        \low: \underbrace{\{j : A[:,j] \neq 0 \}}_{\mathrm{domain}(\low)} \to \Z, \quad j \mapsto \max \{i : A[i,j] \neq 0 \}.
    \end{align*}
We say that $A$ is \emph{reduced} if $\low$ is injective.  An $R = DV$ decomposition is a matrix equation where $R$ is reduced and $V$ is invertible and upper triangular.

Suppose that $R_n = \partial_n V_n$ is such a decomposition for each $n$.   Let $\low^n$ be the low function of $R_n$, and let $\Gamma_n = \{(j, \low(j)) : R[:,j] \neq 0\}$ be the graph $\low^n$.  It can then be shown (\cite{cohen2006vines, SMVDualities11}) that
    \begin{enumerate}
        \item Each set $\Simplices_n(K)$ partitions into three disjoint subsets: $B_n \sqcup B^*_n \sqcup H_n$, where $B_n$ is the image of $\low^{n+1}$ and $B^*_n$ is the domain of $\low^n$.
        \item If $B_n^*(t)$ denotes the set of simplices in $B_n$ born by time $t$, then $\partial_n[:, B^*_n(t)]$ is a basis for the space of boundaries $\Boundaries_n(K_t)$.
        \item Let $\hatgraph_n$ denote the subset of $\Gamma_n$ consisting of those pairs $(\sigma, \tau)$ such that $\birth(\sigma) \neq \birth(\tau)$, and let $E_n = \{ \tau : (\sigma, \tau) \in \hatgraph \} \cup H_n$.  Then $V_n[:,H_n] \cup R_n[:,E_n] $ is a persistent homology cycle basis.  The lifespan of $V_n[:, \sigma]$ is $[\birth(\sigma), \infty)$ and the lifespan of $R_n[:, \tau]$ is $[\birth(\sigma), \birth(\tau))$, where $\sigma$ is the unique simplex such that $(\sigma, \tau) \in \Gamma$.  
        \item In particular, the barcode of $K_\bullet$ may be read off from the $R = DV$ decompositions.
    \end{enumerate}

\section{Correctness of Algorithms 1 and 2}

Here we provide proofs of correctness for Algorithms 1 and 2.  As the details are primarily technical in nature, the exposition is fairly terse.  The arguments are primarily self-contained, however we begin by recalling one result from \cite{eirene}, which will be used in the proof of Algorithm 1.

\subsection{Review: characterization of persistent homology bases}

\newcommand{\ei}{{\epsilon_i}}
\newcommand{\eineg}{{\epsilon_{i-1}}}
\newcommand{\ej}{{\epsilon_j}}
\newcommand{\ejneg}{{\epsilon_{j-1}}}
\newcommand{\xij}{X^{i, j}}
\newcommand{\yij}{Y^{i, j}}
\newcommand{\Qij}{Q^{i,j}}
\newcommand{\qij}{q^{i, j}}
\newcommand{\eij}{E^{i,j}}

We will invoke a result from \cite{eirene} and \cite{henselman2017} which requires one new definition and one new notational convention.  For notation, let $n$ be given, fix $\ei < \ej$, and set
    \begin{align*}
        \xij: = \Cycles_n(K_\ei) \cap \Boundaries_n(K_\ej)
        &&
        \yij: = \Cycles_n(K_\eineg) + \Boundaries_n(K_\ejneg)
        &&
        \Qij: = \frac{\xij}{\xij \cap \yij}
    \end{align*}
We then define $\qij$ as the quotient map 
    \begin{align*}
        \qij:   \xij \to \Qij 
    \end{align*}

\begin{definition}
A subset $E \subseteq \xij$  is an \emph{$(i,j)$-basis} if $\qij|_E$ is injective and $\qij(E)$ is a basis of $\Qij$.  
\end{definition}



\begin{theorem}[\cite{eirene, henselman2017}]
\label{thm:phindependencecharacterization}
A family of $n$-cycles $E$ is a persistent homology cycle basis iff
    \begin{align*}
        E^{i,j} := E \cap (\xij - \yij)
    \end{align*}
is an $(i,j)$ basis for all $i$ and $j$.
\end{theorem}




\begin{remark}
A special case of Theorem \ref{thm:phindependencecharacterization} (for simplex-wise filtrations) also appeared in \cite[Theorem 1]{wu}.  This construction is also closely related to that of \emph{pair groups}, c.f.  \cite{persistenthomologyasurvey}.
\end{remark}    

\subsection{Correctness of Algorithm \ref{alg:edge}}

We restate Algorithm \ref{alg:edge} for ease of reference.

\begin{algorithm}
\caption{Edge-loss persistent cycle minimization}
\label{alg:edge}
\begin{algorithmic}[1]
\STATE Compute a persistent homology basis $\hcyclebasis$ for homology in dimension 1, with coefficients in $\Q$,  using the standard matrix decomposition procedure described in the Supplementary Material. Arrange the elements of $\hcyclebasis$ into an ordered sequence $\obasis^0 = (\obasisel^{0,1}, \ldots, \obasisel^{0,m})$.
\FOR{$j = 0, \ldots, m-1$}
\STATE Solve Program \eqref{eq:edgelossgeneral} to optimize the $j+1$th element of $\obasis^{j}$.  Let $\optimalrep$ denote the solution to this problem, and define $\obasis^{j+1}$ by replacing the $j+1$th element of $\obasis^{j}$ with $\optimalrep$.  Concretely, $\obasisel^{j+1,j+1} = \optimalrep$, and $\obasisel^{j+1,k} = \obasisel^{j,k}$ for $k \neq j$.
\ENDFOR
\STATE Return $\hcyclebasis^*: = \{\obasisel^{m,1}, \ldots, \obasisel^{m,m}\}$, the set of elements in $\obasis^m$.
\end{algorithmic}
\end{algorithm}

Recall that Program \eqref{eq:edgelossgeneral} optimizes the $j$th element of an ordered sequence of cycle representatives $\obasis = (\obasisel^1, \ldots, \obasisel^m)$.  In particular, it seeks to minimize $\originalrep := \cycle^j$.  To define this program, we first construct a matrix $A$ such that $A[:, i] = \cycle^i$ for $i = 1, \ldots, m$.  We then define  three index sets, $\goodcycleindices, \goodtriangles, \goodedges$ such that 
    \begin{align*}
        \goodcycleindices &= \{ i :  \birth(\cycle^i) \le \birth(\originalrep), \;  \death(\cycle^i) \le \death(\originalrep), \; i \neq j \} \\
        \goodtriangles &= \{\sigma \in \Simplices_2(K) : \birth(\sigma) \le \birth(\originalrep)\} 
        \\
        \goodedges &= \{\sigma \in \Simplices_1(K) : \birth(\sigma) \le \birth(\originalrep)\}
    \end{align*} 

Program \eqref{eq:edgelossgeneral} can then be defined as follows.

\begin{align}
\begin{split}
    \text{minimize   } & ||W \optimalrep ||_1 = \sum_{i=1}^N  (x^+_i + x_i^-)\\
   \text{subject to  } &  
      (\optimalrep^+ - \optimalrep^- )= \originalrep[\goodedges] +   \partial_2[\goodedges, \goodtriangles]  \q + A[\goodedges, \goodcycleindices] \p \\
      & \p \in \Q^{\goodcycleindices} \\
      & \q \in \Q^{\goodtriangles} \\      
      & \optimalrep \in G^{\goodedges } \\      
      & \optimalrep^+, \optimalrep^- \geq 0 
      \end{split}
      \tag{14}
      \label{eq:edgelossgeneral}
\end{align}

\newcommand{\vanishset}{\mathcal{S}}
\newcommand{\orderset}{\mathcal{T}}

\begin{theorem}
For each $k$, the family of cycles $\{\obasisel^{k,1}, \ldots, \obasisel^{k,m}\}$ constructed in Algorithm \ref{alg:edge} is a persistent homology cycle basis.  Moreover, lifespans are preserved, in the sense that
    \begin{align}
        \persinterval(\obasisel^{0,l}) = \persinterval(\obasisel^{k,l})
        \label{eq:samelifespan}
    \end{align}
for all $k$ and $l$.
\end{theorem}
\begin{proof}
We proceed by induction on $k$, the base case $k = 0$ begin clear.  Assume the desired conclusion holds for $k$.   For ease of reference, put
    \begin{align*}
        \obasis := (\obasisel^1, \ldots, \obasisel^m) := (\obasisel^{k,1}, \ldots, \obasisel^{k,m}) = \obasis^{k}.
        &&
        \originalrep := \obasisel^{k+1}
        &&
        [\epsilon_i, \epsilon_j) := \persinterval(\originalrep)
    \end{align*}
We may then partition $\goodcycleindices$ as the disjoint union $\vanishset \sqcup \orderset$, where 
    \begin{align*}
        \vanishset &= \{l \in \goodcycleindices : \birth(\cycle^l) < \birth(\cycle^{k+1})\; \text{ or } \; \death(\cycle^l) < \death(\cycle^{k+1}) \} \\
        \orderset &= \{l \in \goodcycleindices :  \birth(\cycle^l) = \birth(\cycle^{k+1}), \; \death(\cycle^l) = \death(\cycle^{k+1}), \; l \neq k+1 \}.
    \end{align*}
An optimal solution to Program \eqref{eq:edgelossgeneral} can then be expressed in  form 
    \begin{align*}
        \optimalrep 
        = 
        \obasisel^{k+1} + \partial_2[\goodedges, \goodtriangles] \q + A[\goodedges, \goodcycleindices] \q 
        = 
            \obasisel^{k+1}
            +
            \underbrace{\partial_2[\goodedges, \goodtriangles] \q}_{\cycleu}
            +
            \underbrace{A[\goodedges, \vanishset] \q[\vanishset]}_{\cyclev}
            +
            \underbrace{A[\goodedges, \orderset] \q[\orderset]}_{\cyclew}
    \end{align*}
where 
    \begin{align*}
        \cycleu \in \Boundaries(K_{\epsilon_i}) \subseteq \xij \cap \yij 
        &&
        \cyclev \in \xij \cap \yij
        &&
        \cyclew \in \spann(\{\obasisel^t : t \in \orderset \}).
    \end{align*}
Now put
    \begin{align*}
        F &: = \{\obasisel^t : t \in \orderset \} \sqcup \{\obasisel^{k+1}\}\\
        F' &: = \{\obasisel^t : t \in \orderset \} \sqcup\{\obasisel^{k+1} + \cyclew\} \\
        F'' &: = \{\obasisel^t : t \in \orderset \} \sqcup\{\obasisel^{k+1} + \cycleu + \cyclev + \cyclew\} =  \{\obasisel^t : t \in \orderset \} \sqcup \{\optimalrep\}\\        
    \end{align*}
Since $\cyclew \in \spann(\{\obasisel^t : t \in \orderset \})$, it is easily argued that $\spann(F) = \spann(F')$.  Thus $\spann(\qij(F))=\spann(\qij(F'))=\spann(\qij(F'')) = \Qij$.  Dimension counting thus implies that $\qij(F'')$ is an $(i,j)$-basis.  
    
Given this observation, it is straightforward to verify that $\{\obasisel^{k+1, 1}, \ldots, \obasisel^{k+1, m} \}$ is a bona-fide cycle basis and $\persinterval(\optimalrep) = \persinterval(\originalrep)$.  The desired conclusion follows.
\end{proof}

This establishes our primary objective:

\begin{theorem}
The set $\hcyclebasis^*$ returned by Algorithm \ref{alg:edge} is a bona fide persistent homology cycle basis of the filtered simpicial complex $K_\bullet$.
\end{theorem}

\subsection{Correctness of Algorithm \ref{alg:rdvvolumeoptimization}}

Recall that \cite{Obayashi2018} defines a \emph{persistent volume} for a birth-death pair $(\sigma_{b_i}, \sigma_{d_i})$ as an $(n+1)$ chain $\volvec \in \Chains_{n+1}(K_{d_i})$ such that
\begin{align}
    \volvec   & = \sigma_{d_i} + \sum_{\sigma_k \in \mathcal{F}_{n+1}} \alpha_k\sigma_k \label{obacond1} \\
    (\partial_{n+1} \volvec)_\tau  & = 0 \quad \forall \tau \in \mathcal{F}_n \label{obacond2}\\
    (\partial_{n+1} \volvec)_{\sigma_{b_i}}  & \ne 0, \label{obacond3}
\end{align}
where 
    \begin{align}
        {\mathcal F}_l := \{ \simplex_k \in \Simplices_l(K) : b_i < k < d_i \}
        \label{eq:fdef}
    \end{align}
is the family of $l$-simplices whose birth time lies strictly between $b_i$ and $d_i$.
The linear program associated to $(\sigma_{b_i}, \sigma_{d_i})$ in \cite{Obayashi2018} can then be summarized as 
\begin{align}
\begin{split}
    \text{minimize } & \loss(\volvec) \\
    \text{subject to } 
    & \eqref{obacond1}, \eqref{obacond2}, \eqref{obacond3}\\
    & \textbf{v} \in \Chains_{n+1}(K_{d_i}) 
\end{split}
\tag{10}
\label{eq:generalminimalvolume}
\end{align}

Let us restate Algorithm \ref{alg:rdvvolumeoptimization} and the corresponding optimization problem, Program \eqref{eq:trianglelossgeneral}, for ease of reference.

\begin{algorithm}
\caption{Triangle-loss persistent cycle minimization}
\label{alg:rdvvolumeoptimization}
\begin{algorithmic}[1]
\STATE Place a filtration-preserving linear order $\le\dimss{l}$ on $\Simplices_l(K)$ for each $l$.
\STATE Compute an $R = \partial_{n+1} V$ decomposition as described in \cite{cohen2006vines} and the Supplementary Material.  We then obtain a set $\Gamma$ 
of birth/death pairs $(\sigma, \tau)$.
\STATE For each $(\sigma, \tau) \in \Gamma$ such that $\birth(\sigma) < \birth(\tau)$,  put 
    \begin{align*}
        \mathcal{F}_n &:= \{\sigma' \in \Simplices_n(K) : \birth(\sigma') \le \birth(\tau), \; \sigma \lneq^{(n)} \sigma'\} 
        \\
        \mathcal{F}_{n+1} &: = \{ \tau' \in \Simplices_{n+1}(K) : \birth(\sigma) \le \birth(\tau'), \; \tau' \lneq^{(n+1)} \tau \} 
    \end{align*}
    and ${\hat {\mathcal{F}}}_{n+1}:= \mathcal{F}_{n+1} \cup \{\tau\}$.  Compute a  solution to the corresponding Program \eqref{eq:trianglelossgeneral}, and denote this solution by  $\optimalrep^{\sigma, \tau}$. 
    \STATE Put   
        $
            \hat \deathbasis: = \{ \partial_{n+1} (\optimalrep^{\sigma, \tau}) : (\sigma, \tau ) \in  \Gamma \; \text{ and } \; \birth(\sigma) < \birth(\tau)\}$ 
            and let $\hat \deathbasis' := \{ \cycle \in \calm : \death(\cycle) = \infty  \}$, where $\calm$ is a persistent homology cycle basis calculated by the standard $R=DV$ method.
    \STATE Return $\deathbasis: = \hat \deathbasis \cup \hat \deathbasis'.$
\end{algorithmic}
\end{algorithm}

Recall that we refer to Program \eqref{eq:trianglelossgeneral} as the  \emph{general triangle-loss problem}.
\begin{align}
\begin{split}
 \text{minimize } & ||W \mathbf{v} ||_1 = \sum_{i=1}^N (v_i^+ + v_i^-)  \\
\text{subject to } &  \partial_{n+1}[ \sigma , \hat {\mathcal{F}}_{n+1} ] \volvec \neq 0     \\
&  \partial_{n+1}[\mathcal{F}_n, \hat {\mathcal{F}}_{n+1} ] \volvec = 0 \\
 & \volvec_{\tau} = 1\\
     & \mathbf{v}^+, \mathbf{v}^- \ge 0 \\
& \mathbf{v}^+, \mathbf{v}^- \in G^{ \hat {\mathcal{F}}_{n+1}}
\end{split}
\tag{15}
\label{eq:trianglelossgeneral}
\end{align}

To verify that Algorithm \ref{alg:rdvvolumeoptimization} returns a bona-fide persistent homology cycle basis, let us begin by placing the elements of $K$ into a sequence $(\sigma_1, \ldots, \sigma_{|K|})$ by ordering simplices first by birth time, second (that is, breaking ties when birth times agree) by dimension, and finally (breaking ties when birth times and dimensions agree) by the chosen linear orders $\le\dimss{l}$.  It is simple to verify that this rule defines a unique linear order on $K$, and that the filtration $K'_\bullet$ defined by
    \begin{align*}
        K'_i : = \{ \simplex_1, \ldots, \simplex_i\}
    \end{align*}
is a simplex-wise refinement of $K_\bullet$.

\begin{theorem}
\label{thm:specialcaseofobayashi}
Let $(\sigma, \tau)$ be a birth-death pair, and choose $b_i, d_i$ such that  $(\sigma_{b_i}, \sigma_{d_i})  = (\sigma, \tau)$.  Then the sets ${\mathcal F}_n$ and ${\mathcal F}_{n+1}$  defined  in Algorithm \ref{alg:rdvvolumeoptimization} both satisfy \eqref{eq:fdef}.  Consequently, Program \eqref{eq:trianglelossgeneral} is a special case of Program \eqref{eq:generalminimalvolume}.
\end{theorem}
\begin{proof}
The proof is a straightforward exercise in definition checking.
\end{proof}

Now let $\hat \hcyclebasis$ be a set containing the boundary of one optimal solution $\optimalrep^{\sigma, \tau}$ to Program  \eqref{eq:trianglelossgeneral} for each birth-death pair $(\sigma, \tau) = (\sigma_{b_i}, \sigma_{d_i})$ (even if $\birth(\sigma) = \birth(\tau)$)).  Let $\caln$ be a persistent homology cycle basis for $K_\bullet'$, and let $\hat \hcyclebasis ' = \{ \cycle \in \caln : \death(\cycle) = \infty \}$ be the collection of cycle representatives in $\caln$ that never die.  Then, by combining Theorem \ref{thm:specialcaseofobayashi} with Theorem 5 of \cite{Obayashi2018}, we find that $\hcyclebasis: = \hat \hcyclebasis \cup \hat \hcyclebasis'$ is a persistent homology cycle basis of $K'_\bullet$.  

If we assume that the bounding volumes used to obtain $\hat \hcyclebasis$ are the same as those used to obtain $\hat \deathbasis$ in Algorithm \ref{alg:rdvvolumeoptimization}, and, likewise, that $\hat \hcyclebasis' = \hat \deathbasis'$ (this will be true provided we use the order $\le\dimss{l}$ when ordering columns of $\partial_l$ for the $K_\bullet$ calculation) then 
    \begin{align*}
        {\mathcal D} = \{\cycle \in \hcyclebasis : \birth(\cycle) < \death(\cycle) \}
    \end{align*}
where $\birth$ and $\death$ are the birth and death functions of $K_\bullet$, not $K'_\bullet$.

From here, it remains only to verify that $\mathcal D$ is a bona fide persistent homology cycle basis $K_\bullet$.  This follows from the following general observation.

\begin{theorem}
Let $L_\bullet'$ be a refinement of a filtration $L_\bullet$ on a simplicial complex $L$, and let $\hcyclebasis$ be a persistent homology cycle bases for $L_\bullet'$.  If $\birth$ and $\death$ are the birth and death functions of $L_\bullet$ and ${\mathcal D}: = \{\cycle \in \hcyclebasis : \birth(\cycle) < \death(\cycle) \}$, then $\cald$ is a persistent homology cycle basis of $L_\bullet$.
\end{theorem}
\begin{proof}
Recall that, by definition, a set $\cale$ is a persistent homology cycle basis of $K_\bullet$ iff two criteria hold: (i) $\persinterval(\cycle)$ must be nonempty for each $\cycle \in \cale$, and (ii) $\{ [\cycle]\in \Homologies_n(K_{\epsilon_i}; G) : \epsilon_i \in \persinterval(\cycle) \}$ is a basis for $\Homologies_n(K_{\epsilon_i}; G)$ for each $i$.  Since $\hcyclebasis$ is a persistent homology cycle basis for $L'_\bullet$, it is a straightforward exercise in definition checking to verify that $\cald$ is a persistent homology cycle basis for $L_\bullet$.
\end{proof}

This establishes our primary objective:

\begin{theorem}[Correctness of Algorithm \ref{alg:rdvvolumeoptimization}]
Algorithm \ref{alg:rdvvolumeoptimization} returns a bona-fide persistent homology cycle basis of $K_\bullet$.
\end{theorem}

\section{Supplementary Tables and Figures}

 \begin{figure}[h!]
 \begin{center}
 \includegraphics[width=1\textwidth]{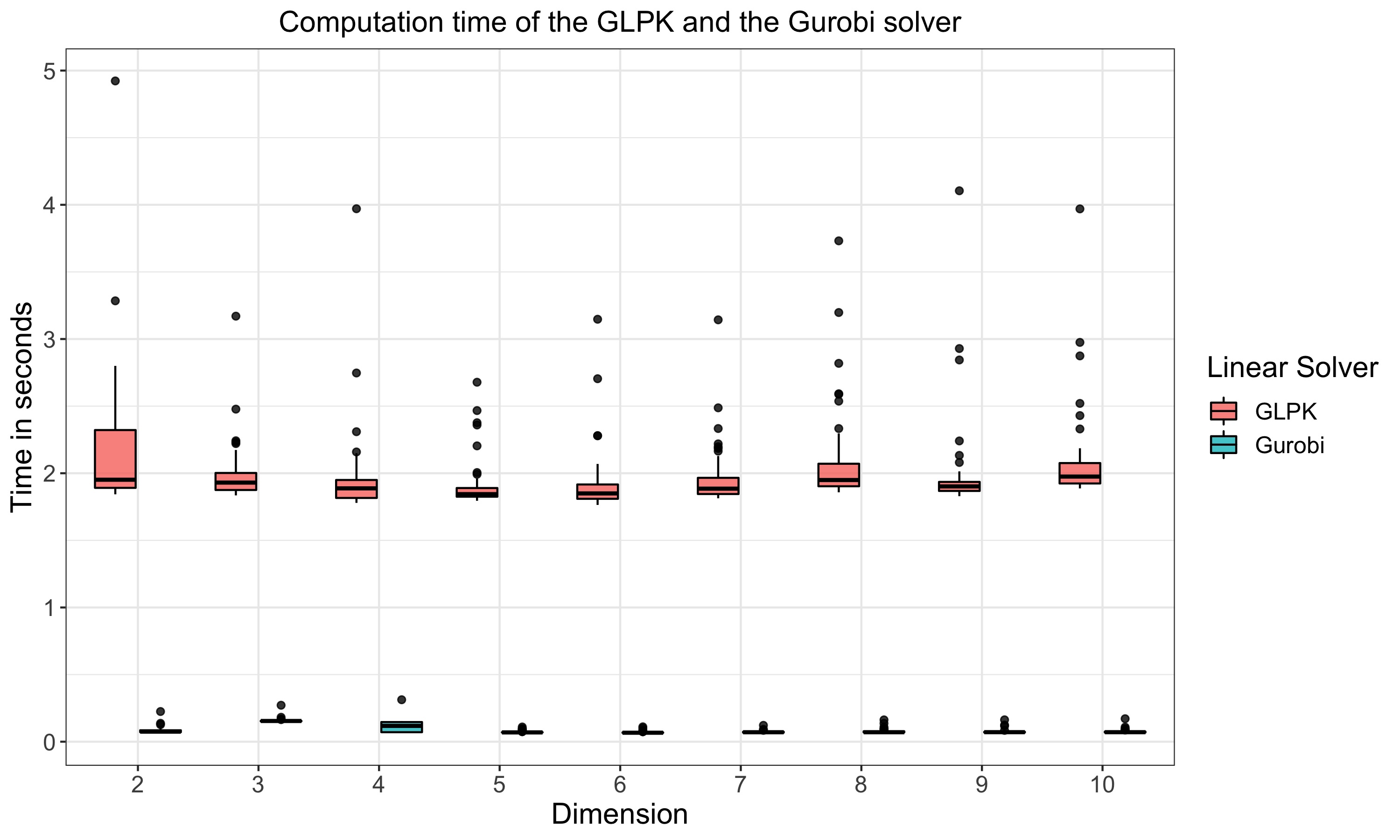}
 \end{center}
 \caption{Computation time of the GLPK linear solver (red) and the Gurobi linear solver (green) to solve the uniform/length-weighted edge-loss minimal problems in Algorithm 1. We perform experiments on $90$ data sets, 10 for each dimension 2-10, generated from the normal distribution. The horizontal axis is the dimension of the data set, and the vertical axis is the time it takes to solve an optimization problem. We observe that the Gurobi solver is consistently faster than the GLPK solver and that computation time seems fairly constant across dimension.}\label{fig:glpk_gurobi}
 \end{figure}

\begin{center}
\begin{table}[]
\caption{{\normalsize{Classifying the coefficients of the optimal cycles for all of the real-world data discussed in Section 5.1 as well as all of the synthetic sets discussed in Section 5.2. The rows are labeled by the coefficient type of the cycle representatives: ``Integral'' means the coefficients for the cycle representative $\optimalrep$ are in $\mathbb{Z}$ and ``In $\{-1,0,1\}$'' means the coefficients for the representative $\optimalrep$ are in $\{-1,0,1\}$. For the columns, $\optimalrep$ represents the optimal representative with its superscript indicating the type of optimization problem: $I$ for integer programming and $NI$ for linear programming, and its subscript indicating the type of optimal cycle: ${E\text{-}Len}, E\text{-}Unif, T\text{-}Unif$ refer to edge loss length-weighted minimal (minimizing total length of $1$-simplices), edge loss uniform (minimizing total number of $1$-simplices), and triangle loss uniform (minimizing the number of $2$-simplices a cycle representative bounds), respectively.}}}
\centering


{\scriptsize{ \begin{tabular}{ |>{\centering}m{7em} *{10}{>{\centering\arraybackslash}m{2.5em} }|}
 \hline
 & \multicolumn{10}{c|}{\textbf{Edge-loss filtered homological optimal cycles} (Program \eqref{eq:edgelossgeneral})} \\
\hline
& \multicolumn{4}{c}{\textbf{Randomly Generated Data Sets}} & & 
 \multicolumn{4}{c}{\textbf{Real-World Data Sets}} &  \\  \cline{2-5}  \cline{7-10}

\textbf{Coefficient Type} & $\x\I_{E\text{-}Len}$ & $\x\NI_{E\text{-}Len}$ & $\x\I_{T\text{-}Unif}$ & $\x\NI_{E\text{-}Unif}$ &  & $\x\I_{E\text{-}Len}$ & $\x\NI_{E\text{-}Len}$ & $\x\I_{E\text{-}Unif}$ & $\x\NI_{E\text{-}Unif}$ & \\
\hline
\textbf{Integral}  & $100\%$ &$100\%$&   $100\%$ & $100\%$ &  & $100\%$ &$100\%$&  $100\%$ & $100\%$ & \\
\textbf{In $\{-1, 0, 1\}$} &  $100\%$  & $100\%$   &$100\%$ & $100\%$ & & $100\%$ &$100\%$&  $100\%$ & $100\%$ & \\ \hline
 & \multicolumn{10}{c|}{\textbf{Edge-loss persistent homological optimal cycles} (\pr (8))}  \\\hline

& \multicolumn{4}{c}{\textbf{Randomly Generated Data Sets}} & & 
 \multicolumn{4}{c}{\textbf{Real-World Data Sets}} &  \\  \cline{2-5}  \cline{7-10}
 \textbf{Coefficient Type} & $\x\I_{E\text{-}Len}$ & $\x\NI_{E\text{-}Len}$ & $\x\I_{T\text{-}Unif}$ & $\x\NI_{E\text{-}Unif}$ &  & $\x\I_{E\text{-}Len}$ & $\x\NI_{E\text{-}Len}$ & $\x\I_{E\text{-}Unif}$ & $\x\NI_{E\text{-}Unif}$ & \\

\hline 
\textbf{Integral}  & $100\%$ &$100\%$&   $100\%$ & $100\%$  &&  $100\%$ & $100\%$ &  $100\%$ & $99.94\%$ & \\
\textbf{In $\{-1, 0, 1\}$} &  $100\%$  & $100\%$   &$100\%$ & $100\%$ &  & $100\%$ & $100\%$&  $100\%$ & $99.94\%$ & \\ \hline

& \multicolumn{10}{c|}{\textbf{Triangle-loss optimal cycles} (\pr \eqref{eq:trianglelossgeneral})}  \\\hline
& \multicolumn{4}{c}{\textbf{Randomly Generated Data Sets}} & & 
 \multicolumn{4}{c}{\textbf{Real-World Data Sets}} &  \\  \cline{2-5}  \cline{7-10}
 \textbf{Coefficient Type} & $\x\I_{T\text{-}Unif}$ &    $\x\NI_{T\text{-}Unif}$ & $\x\I_{T\text{-}Area}$ & $\x\NI_{T\text{-}Area}$  &  & $\x\I_{T\text{-}Unif}$ &    $\x\NI_{T\text{-}Unif}$ & $\x\I_{T\text{-}Area}$ & $\x\NI_{T\text{-}Area}$ &\\
\textbf{Integral}  &  $100\%$ &$99.99\%$ &$100\%$ &$100\%$  &  & $100\%$ &  $100\%$ &  $100\%$ & $100\%$  & \\
\textbf{In $\{-1, 0, 1\}$} &  $100\%$ &$99.99\%$ &$100\%$ &$100\%$  &  & $100\%$ &  $100\%$ &  $100\%$ & $100\%$  & \\ \hline

\end{tabular}
}}
\label{entry}
\end{table}
\label{tab:IntegerCoefficients}
\end{center}

\end{document}